\documentclass[11pt]{article}

\usepackage{xcolor}
\usepackage{amssymb,amsthm}
\usepackage{amsmath}
\usepackage{mathtools}
\usepackage{subcaption}
\usepackage{algorithm}
\usepackage{algorithmic}
\usepackage{hyperref}
\usepackage{natbib}
\usepackage{graphicx}
\usepackage{booktabs}
\usepackage{longtable,array}
\usepackage{enumitem}

\setcitestyle{authoryear,round,comma,aysep={},yysep={,},notesep={, }}

\hypersetup{colorlinks=true,allcolors=blue}

\newcommand{\E}{\mathbb{E}}
\newcommand{\R}{\mathbb R}
\newcommand{\A}{\mathcal A}
\newcommand{\norm}[1]{\left\Vert #1 \right\Vert}

\newcommand{\op}{\operatorname{op}}
\newcommand{\argmax}{\operatorname*{arg\,max}}
\newcommand{\argmin}{\operatorname*{arg\,min}}
\newcommand{\Reg}{\operatorname{Reg}}
\DeclareMathOperator{\diam}{diam}

\usepackage{geometry}
\usepackage{microtype}
\usepackage{placeins}

\usepackage[nameinlink,capitalise,noabbrev]{cleveref}
\usepackage{setspace}

\newtheorem{theorem}{Theorem}[section]
\newtheorem{lemma}[theorem]{Lemma}
\newtheorem{proposition}[theorem]{Proposition}
\newtheorem{corollary}[theorem]{Corollary}
\newtheorem{assumption}{Assumption}

\theoremstyle{definition}
\newtheorem{criterion}{Criterion}
\newtheorem{criterionvariant}{Criterion}[criterion]

\theoremstyle{remark}
\newtheorem{remark}[theorem]{Remark}
\crefname{assumption}{Assumption}{Assumptions}
\Crefname{assumption}{Assumption}{Assumptions}
\crefname{criterion}{Criterion}{Criteria}
\Crefname{criterion}{Criterion}{Criteria}
\crefname{criterionvariant}{Criterion}{Criteria}
\Crefname{criterionvariant}{Criterion}{Criteria}

\DeclareMathOperator{\dist}{dist}
\newcommand{\Pp}{\mathbb P}
\newcommand{\Fr}{\mathrm{F}}

\title{Frequency-Uniform Log-Regret Reinforcement Learning
for Diffusions}
\author{%
  Jose Blanchet\thanks{Department of Management Science and Engineering, Stanford University. 
Emails: \texttt{\{jose.blanchet; lhao; renyuanxu\}@stanford.edu}. Authors are listed alphabetically. }
  \qquad\qquad
  Lutong Hao\footnotemark[1]
  \qquad\qquad
  Renyuan Xu\footnotemark[1]
}
\date{October 26, 2026}

\begin{document}
\maketitle
\vspace{-25pt}
\begin{abstract}
High-frequency control problems are often studied through time discretization, but refining the timestep $\delta$ increases the discrete
horizon $H=T/\delta$ and can make standard finite-horizon regret guarantees
deteriorate even when the physical horizon $T$ is fixed. We study episodic
value-based reinforcement learning for Euler discretizations of controlled
diffusions. The performance is measured as the cumulative regret over $K$ episodes relative to the
optimal policy of the corresponding Euler model. 
Under stability and regularity conditions on value function gradients, Bellman residual derivatives, 
and fitted critics, we prove an $\mathcal{O}(\log K)$ regret bound with constants uniform in $\delta$.
For dynamics with drift linear in
unknown parameters and unknown constant diffusion covariance, we verify
localized versions of these conditions for a constrained, model-based
controlled-diffusion fitted-$Q$ algorithm (CD-FQ) and obtain high-probability
$\mathcal{O}(\log K\log\log K)$ regret for fixed and finite $T$, again with constants uniform in $\delta$.
\end{abstract}

\section{Introduction}
\label{sec:intro}
A learning guarantee for a stochastic control system should distinguish between how long the system operates and how often its decisions are updated. In a service system controlled over a fixed planning period, halving the time between interventions doubles the number of decision stages but leaves the physical duration unchanged. A regret bound polynomial in the stage count can therefore predict an arbitrarily large learning cost as the decision grid is refined. This sensitivity matters when choosing how frequently to observe, learn, and act: it obscures whether a performance guarantee reflects the difficulty of controlling the system or the temporal resolution used to describe it. Our objective is to identify conditions under which value-based reinforcement learning retains fast regret guarantees as that resolution increases.

Sequential decision problems arise throughout science and engineering whenever
present actions affect the uncertain evolution of a system.
Examples include treatment planning \citep{wang2024threshold},
robotic control \citep{djeumou2023learn},
inventory \citep{dai2013brownian}, service operations
\citep{ward2008admission}, and financial decisions
\citep{pham2009continuous}. Reinforcement learning for queueing networks \citep{dai2022queueing} and optimal execution in financial markets \citep{wang2026execution} illustrate the value of exploiting the structure
of the underlying stochastic dynamics in the context of learning. This perspective motivates our focus on controlled
diffusions, which provide tractable approximations of complex operational
systems by capturing the aggregate effects of their random events
\citep{harrison1997dynamic,braverman2020taylor}.  These models also provide a foundation for stochastic
control \citep{ata2026diffusion} and continuous-time reinforcement learning
\citep{wang2020stochastic}.

We study episodic value-based learning for Euler discretizations of
controlled diffusions in continuous state and action spaces. Over a fixed
physical horizon $T$, a timestep $\delta$ produces $H=T/\delta$ decision
stages per episode. The learner uses data from previous episodes to select
policies and is evaluated by cumulative conditional expected regret relative
to the optimal policy of the same Euler model. 

Existing finite-horizon analyses can achieve logarithmic regret in the
number of episodes while retaining polynomial dependence on $H$. For
example, the leading fast-rate term in the online fitted-$Q$ guarantee of
\citet{duan2024taming} contains an $H^4$ factor. Substituting $H=T/\delta$ gives a polynomial penalty in $1/\delta$ at fixed $T$, because it tracks only the number of stages and not how one-step rewards and state increments scale with the timestep.
This raises a distinct question:
\begin{center}
\begin{minipage}{0.8\linewidth}
\centering\itshape
How can we design online learning algorithms that retain fast rate \\and log regret as the observation and decision
frequency increases?\par
\end{minipage}
\end{center}

The question is particularly delicate for value-based learning. Because an action is held for only one timestep, its effect on the one-step value and
the action-dependent differences in the ordinary $Q$-function are of order
$\delta$. Learning must therefore resolve a shrinking action signal while
controlling error propagation across $T/\delta$ stages, and keep the deployed policies informative. This scaling
also motivates advantage-rate representations and continuous-time
$q$-learning \citep{tallec2019robust,jia2023qlearning}.

\subsection{Our work and contributions}

We develop a regret framework that guides the design and analysis of a fitted-Q algorithm.

\paragraph{A discrete-to-continuous fast-regret framework.}
We formulate episodic value-based learning for a family of Euler models indexed
by the timestep $\delta$, separating episode count, physical horizon, and decision frequency, and placing discrete-time regret analysis and diffusion
control in a common framework. The benchmark is the optimum of the
corresponding Euler model, so statistical learning regret remains distinct
from continuous-time discretization error
(\cref{sec:model-global-regret}). Our bounds
concern statistical performance at fixed physical horizon; the number of observations and the computational effort can still grow as the grid is refined.

\paragraph{Discretization-uniform logarithmic regret.}
We formulate criteria on fitted critics that are sufficient for fast regret,
separating the policy-performance argument from algorithm-specific verification
(\cref{def:fitted-envelope-interface,def:true-residual-certificate}). 
Under quadratic advantage and value-gradient stability conditions, Bellman-residual derivatives of order $\delta/\sqrt{k}$ after $k$ episodes yield critic action-gradient errors of the same order, with constants governed by physical time. Quadratic policy loss then gives
\begin{eqnarray*}
  \mathcal{O}(\delta /\sqrt{k}) \,\, \textrm{action-gradient error} \Longrightarrow  \mathcal{O}(\delta /k) \,\, \textrm{loss per stage}  \Longrightarrow \mathcal{O}(T /k) \,\, \textrm{loss per episode}.  
\end{eqnarray*}
Summing over episodes gives $\mathcal{O}(\log K)$ regret for fixed $T$, with constants uniform in $\delta$ (\cref{thm:residual-regret-transfer}).  An explicit fitted-critic construction shows that exponential growth in $T$ of the logarithmic-regret coefficient is unavoidable for a guarantee uniform over this general class (\cref{prop:physical-time-lower-bound}).

\paragraph{End-to-end learning to verify $\widetilde{\mathcal{O}}(\log K)$ regret.}
For controlled diffusions with drift linear in unknown parameters and unknown
constant diffusion covariance, we develop a constrained, model-based fitted-$Q$
algorithm (CD-FQ) that permits finitely many separated optimal branch points. 
It estimates the drift parameter by projected least squares and the covariance from residual increments, simulates from the fitted Euler model, fits constrained critics backward from the terminal reward (a discrete analogue of solving the HJB equation), and chooses a policy greedily. Under regularity, exploration, approximation, and simulation-budget conditions, it achieves high-probability $\mathcal{O}(\log K\log\log K)$ regret with constants uniform in $\delta$ (\cref{thm:end-to-end-localized-regret}). The analysis jointly establishes critic accuracy and the continued self-exploration property of the greedy policies.

\paragraph{Verification and numerical evidence.}
A model with two optimal branches jointly verifies all assumptions of the learning theorem (\cref{subsec:controlled-multibranch-model}). Queueing experiments examine cumulative regret, stability under timestep refinement, and transfer of learned diffusion policies to finite many-server queues (\cref{subsec:halfin-whitt-numerics}). These experiments extend beyond the theorem's smoothness, boundedness, and action-space assumptions and assess the practical value of exploiting diffusion structure.

\subsection{Related work}
\paragraph{Stability and fast rates in discrete-time RL.}
The closest fast-rate analyses are \citet{duan2024taming} and
\citet{wang2026fast}. \citet{duan2024taming} obtain quadratic performance bounds and fast offline
and online rates through Bellman and occupation-measure stability. 
\citet{wang2026fast} study value-based RL in discounted stochastic
control with continuous actions. They prove an $\mathcal{O}(n^{-1})$ expected policy gap for the deployed policy, 
where \(n\) indexes estimator accuracy and may represent sample size, but without an end-to-end online algorithm. 
Our analysis shares a related quadratic error-to-loss mechanism and
tracks how derivative bounds on stagewise Bellman residuals propagate
through the finite-horizon recursion. Under the stated conditions, this gives logarithmic cumulative online regret
with constants uniform in $\delta$ at the fixed physical horizon $T$. 
The regret bounds in \citet{duan2024taming} retain polynomial dependence on
the number of decision stages $H$, which does not by itself establish
uniformity as $\delta $ decreases at fixed $T$.
Our bound establishes this uniformity while having exponential
dependence on $T$, sharp under our assumptions (\cref{prop:physical-time-lower-bound}).

\paragraph{Episodic learning in continuous time.}
The square-root regret bounds for jump processes in \citet{sqrt_regret}
and fast rates for structured diffusions in
\citet{basei2022logarithmic,szpruch2021exploration} highlight the role of
structure and information in continuous-time learning. In particular,
\citet{szpruch2021exploration} distinguish explicit exploration from
self-exploration, where greedy exploitation supplies informative data.
\citet{guo2021linearconvex} analyze episodic learning for
linear-convex jump-diffusions through stability of feedback controls, linking
least-squares estimation error to sublinear regret.
For exploratory finite-horizon linear-quadratic control,
\citet{giegrich2024convergence} establish convergence of policy-gradient
methods with guarantees that remain stable as the action frequency increases.
We develop this mechanism for fitted critics with approximation error and
multiple optimal branch points, extending the scope under different model and regularity assumptions.

\paragraph{Continuous-time $q$-learning and function approximation.}
For deterministic policies, \citet{cheng2025bridging} derive a continuous-time policy gradient via an advantage-rate function, leading to a model-free algorithm stable across time discretizations.
The exploratory-control framework of \citet{wang2020stochastic} and
continuous-time $q$-learning of \citet{jia2023qlearning} replace vanishing
one-step action gaps with instantaneous advantages. Recent work establishes
regret guarantees for exploratory policy improvement and $q$-learning
\citep{tang2024regret}, as well as model-free, off-policy guarantees for value
and advantage approximation \citep{mou2026ellipticity}. We instead study
cumulative online regret for deterministic greedy fitted-$Q$ policies and
verify the required derivative accuracy for a learned parametric diffusion.

\paragraph{Discretization and continuous-time approximation.}
A related line of work asks how well discrete-time procedures recover
their continuous-time counterparts. Approximating continuous-time stochastic control by discrete models has a classical foundation in numerical stochastic control \citep{kushner2001numerical}. PhiBE \citep{zhu2024phibe} and the
high-order Bellman schemes of \citet{mou2026bellman} evaluate policies from
discretely observed diffusion trajectories, and Optimal-PhiBE
\citep{zhu2025optimal} and PhiBE-Q-Learning \citep{zhu2026phibeq} carry this
PDE-based approach over to policy optimization and off-policy learning. On the
implementation side, \citet{jia2026accuracy} quantify the error of executing a
continuous-time stochastic policy by resampling actions on a grid, and
\citet{pham2026discretization} bound the gap between regularized discrete-time
Bellman problems and the underlying continuous-time control problem. In each
case the quantity of interest is an approximation error, whether of a value
function, policy, or execution scheme. We ask a different question: how
much cumulative regret an online learner incurs against the Euler optimum at
step size~$\delta$, with bounds that hold uniformly in~$\delta$.

\paragraph{Roadmap.} The rest of the paper is organized as follows. Section~\ref{sec:model-global-regret} develops the residual-to-regret framework. Section~\ref{sec:localized-algorithm} introduces our method and states its main regret guarantee, proved in Section~\ref{sec:main-theorem-proof}. Section~\ref{sec:localized-discussion} presents numerical experiments. The appendices provide supporting proofs and implementation details, with notation and constants summarized in Appendix~\ref{app:localized-notation} for ease of reference.

\section{From Bellman residuals to logarithmic regret}
\label{sec:model-global-regret}
We formulate the Euler learning problem and identify sufficient
conditions for discretization-uniform fast regret.  The argument
combines three ingredients: the benchmark inequalities in
\cref{ass:benchmark-error-bounds}, fitted-envelope regularity in
\cref{def:fitted-envelope-interface}, and Bellman residual derivative
control in \cref{def:true-residual-certificate}.  Together they give
the general regret-transfer result in \cref{thm:residual-regret-transfer}.
The later sections establish a localized
version of this mechanism for a constrained fitted-$Q$ algorithm (CD-FQ).

Fix a physical horizon \(T>0\), a state space \(\R^n\), and a nonempty,
compact, convex action set \(\A\subset\R^d\).  Consider the controlled
diffusion
\begin{equation*}
  \mathrm d\bar X_t^\alpha
  =\mu^\star(\bar X_t^\alpha,\alpha_t)\,\mathrm dt
   +(\Sigma^\star)^{1/2}\,\mathrm dB_t,
  \qquad \bar X_0^\alpha\sim p_0,\quad 0\leq t\leq T.
\end{equation*}

Here $B$ is an $n$-dimensional Brownian motion on a filtered
probability space $(\Omega,\mathcal F,(\mathcal F_t),\mathbb P)$
satisfying the usual conditions, $\bar X_0^\alpha$ is
$\mathcal F_0$-measurable, and
$\alpha=(\alpha_t)_{0\leq t\leq T}$ is an $\mathcal A$-valued,
$(\mathcal F_t)$-progressively measurable control process.
The drift \(\mu^\star:\R^n\times\A\to\R^n\) and constant covariance
\(\Sigma^\star\in\R^{n\times n}\), \(\Sigma^\star\succeq0\), are unknown.
The running reward is \(r(x,a)\), and the terminal reward is \(u(x)\).

Fix a maximum timestep $\delta_{\max}>0$. We call a timestep
$\delta$ admissible if $0<\delta\leq\delta_{\max}$ and
$T/\delta\in\mathbb N$. The ceiling is the range over which our assumptions hold uniformly; it does not depend on the episode budget or confidence levels and is fixed when comparing physical horizons.

Our learning model is the Euler discretization with an admissible
timestep $\delta$, $H:=T/\delta$, and grid times $t_h=h\delta$.
Freezing the drift at the observed state \(X_h\) and chosen action \(A_h\)
over \([t_h,t_{h+1})\) gives
\begin{equation*}
  X_{h+1}
  =X_h+\delta\mu^\star(X_h,A_h)
       +\sqrt{\delta}\,(\Sigma^\star)^{1/2}\xi_{h+1},
  \qquad X_0\sim p_0,
\end{equation*}
where \(\xi_{h+1}\sim\mathcal N(0,I_n)\) is independent of all information
available when \(A_h\) is chosen.  Thus the Euler transition kernel is
\begin{equation}
  P^\star(\,\cdot\mid x,a)
  =
  \mathcal N\left(
      x+\delta\mu^\star(x,a),\,
      \delta\Sigma^\star
  \right),
  \label{eq:true-kernel}
\end{equation}
with its \(\delta\)-dependence suppressed.  This kernel generally differs
from the exact sampled diffusion law.  The one-step reward
\(\delta r(x,a)\) approximates the running reward over one interval, and
\(u(X_H)\) is paid at time \(T\).

A deterministic Markov policy is a sequence
\(\pi=(\pi_0,\ldots,\pi_{H-1})\) of Borel maps
\(\pi_h:\R^n\to\A\), with \(A_h=\pi_h(X_h)\).
Each episode uses a fresh initial state from \(p_0\) and fresh transition
noises, independent of preceding episodes.  The expected Euler return is
\begin{equation*}
  J_\delta(\pi)
  =
  \E_\pi\left[
      \sum_{h=0}^{H-1}\delta r(X_h,A_h)+u(X_H)
  \right].
\end{equation*}

For any transition kernel $P$, write $\mathcal{P}$ for its associated
expectation operator:
\[
  (\mathcal{P}v)(x,a)
  :=\int_{\mathbb R^n}v(x')P(dx'\mid x,a).
\]
The Bellman recursion starts from \(V_H^\star=u\) and proceeds backward via
\begin{equation*}
  Q_h^\star(x,a)
    =\delta r(x,a)+(\mathcal{P}^\star V_{h+1}^\star)(x,a),
  \qquad
  V_h^\star(x)=\max_{a\in\A}Q_h^\star(x,a),
  \qquad h=H-1,\ldots,0.
\end{equation*}
The following standing assumption ensures that this recursion is well defined.

\begin{assumption}[Well-posedness, reward growth, and moments]
\label{ass:basic}
The following conditions hold uniformly over the admissible timesteps.
\begin{enumerate}[label=(\roman*),leftmargin=2.2em]
  \item There is a function $W:\mathbb R^n\to[1,\infty)$
  and constants $B_r,B_u\geq0$ such that
  \[
    |r(x,a)|\leq B_rW(x),\qquad |u(x)|\leq B_uW(x),
    \qquad \int W(x)\,p_0(dx)<\infty.
  \]
  \item For some $c_W\geq0$, the kernel \eqref{eq:true-kernel}
  satisfies
  \begin{equation}
    \int W(y)\,P^\star(dy\mid x,a)
    \leq(1+c_W\delta)W(x),
    \qquad x\in\mathbb R^n,\ a\in\mathcal A.
    \label{eq:weight-moment-drift}
  \end{equation}
  \item The functions $r$, $u$, $\mu^\star$, and $W$ are continuous,
  and so is \((x,a)\mapsto\int W(y)\,P^\star(dy\mid x,a)\).
\end{enumerate}
\end{assumption}

Under \cref{ass:basic}, the Bellman recursion gives finite,
continuous optimal critics and value functions satisfying
$|Q_h^\star(x,a)|,|V_h^\star(x)|\leq C_hW(x)$ for finite constants
$C_h$. The action maxima are attained, and an optimal deterministic Borel Markov policy exists.
Iterating \eqref{eq:weight-moment-drift} gives $\E_\pi W(X_h)\le e^{c_Wt_h}\int W\,dp_0$, 
so every value-to-go is integrable. $|J_\delta(\pi)|\le B_J:=(TB_r+B_u)e^{c_WT}\int W\,dp_0$ 
for all policies and timesteps, and 
\[J_\delta^\star=\int V_0^\star\,dp_0=\sup_\pi J_\delta(\pi).\] 

Bounded rewards are the case $W\equiv1$, $c_W=0$. Later on, \cref{lem:parametric-growth-moment} 
verifies the conditions with $W(x)=\sqrt{1+\|x\|^2}$ in the localized model.
\\

For the derivative-based analysis below, fix an open neighborhood
\(\mathcal A^{\mathrm o}\supset\mathcal A\).

\begin{assumption}[Differentiable drift and reward]
\label{ass:true-drift}
The drift \(\mu^\star\) and running reward \(r\) admit \(C^1\) extensions to
\(\R^n\times\mathcal A^{\mathrm o}\). There are finite constants
\(L_{\mu,x}\) and \(L_{\mu,a}\) such that

\begin{equation}
  \sup_{(x,a)\in\R^n\times\mathcal A^{\mathrm o}}
  \norm{\partial_x\mu^\star(x,a)}_{\op}
  \le L_{\mu,x},
  \qquad
  \sup_{(x,a)\in\R^n\times\mathcal A^{\mathrm o}}
  \norm{\partial_a\mu^\star(x,a)}_{\op}
  \le L_{\mu,a}.
  \label{eq:true-drift-jacobians}
\end{equation}
\end{assumption}

Before episode \(k\), the learner selects \(\pi_k\) using prior data and
internal randomization.  Conditional on this information,
\(J_\delta(\pi_k)\) averages over fresh episode randomness.
The cumulative regret is
\begin{equation}
  \Reg(K,\delta)
  =
  \sum_{k=1}^K
  \left\{J_\delta^\star-J_\delta(\pi_k)\right\}.
  \label{eq:regret}
\end{equation}
We study this Euler-model regret as \(K\) and \(\delta\) vary with \(T\)
fixed.
\\

In what follows, we adopt a \emph{value-based reinforcement learning} approach:
after constructing an 
estimated critic sequence with each \(\widehat Q_h\in C^1(\mathbb R^n\times\mathcal A^{\mathrm o})\):
\[
\widehat Q=(\widehat Q_0,\ldots,\widehat Q_{H-1}),
\]
define the following greedy policy $\widehat\pi$ based on $\widehat Q$: 
\[
\widehat\pi_h(x)\in\argmax_{a\in\mathcal A}\widehat Q_h(x,a).\]

\subsection{Policy loss from critic action-gradient error}

\subsubsection{Optimal advantage performance difference}
\label{subsec:optimal-advantage-pdi}
Our analysis starts from an exact performance-difference identity under the trajectory law of the candidate policy. The \emph{optimal-advantage loss} \(V_h^\star(x)-Q_h^\star(x,a)\ge0\) is the one-stage loss from taking action \(a\) at state \(x\), followed by optimal continuation.

\begin{lemma}[Optimal-advantage performance-difference identity]
\label{lem:optimal-advantage-pdi}
Under \cref{ass:basic}, every deterministic Borel Markov policy $\pi$
satisfies
\begin{equation}
    J_\delta^\star-J_\delta(\pi)
    =
    \sum_{h=0}^{H-1}
    \mathbb E_\pi
    \left[
        V_h^\star(X_h)-Q_h^\star
        \bigl(X_h,\pi_h(X_h)\bigr)
    \right].
    \label{eq:episode-optimal-advantage-pdi}
\end{equation}
\end{lemma}
\begin{proof}
See \cref{app:performance-difference-proof}.
\end{proof}
This is the sign-reversed optimal-benchmark form of the classical
performance-difference lemma \citep{kakade2002approximately}. It involves 
only $V_h^\star$, $Q_h^\star$, and the law of $\pi$, so it needs no choice 
of optimal policy and allows nonunique maximizers.

\subsubsection{Benchmark inequalities and quadratic policy loss}
\label{subsec:quadratic-advantage}
By \cref{lem:optimal-advantage-pdi}, the suboptimality of a policy is
the sum of its stagewise optimal-advantage losses.  Two
benchmark properties of $Q^\star $ in \cref{ass:global-benchmark-geometry} make these losses small
for greedy policies.

Together with the fitted-envelope regularity and residual control in
\cref{def:fitted-envelope-interface,def:true-residual-certificate},
these properties give the fast-rate mechanism.  For a fixed physical
horizon \(T\), the backward recursion converts \(\mathcal{O}(\delta/\sqrt{k})\)
Bellman-residual derivatives into critic action-gradient error of the
same order.  The quadratic loss bound then squares this error at the
cost of a factor \(1/\delta\) before the \(H=T/\delta\) stages add up:
\[
\begin{aligned}
  &\text{residual derivative } \mathcal{O}\bigl(\frac{\delta}{\sqrt{k}}\bigr)
  \Rightarrow
  \text{action-gradient error } \mathcal{O}\bigl(\frac{\delta}{\sqrt{k}}\bigr)
  \\
  &\qquad\Rightarrow
  \text{stage loss } \mathcal{O}\bigl(\frac{\delta}{k}\bigr)
  \Rightarrow
  \text{episode loss } \mathcal{O}\bigl(\frac{T}k\bigr).
\end{aligned}
\]
Summing over \(K\) episodes gives logarithmic regret, uniformly in \(\delta\) when the benchmark, drift, and normalized residual constants are.\\

To include actions on the boundary of $\mathcal A$, define the
maximization normal cone
\[
  N_{\mathcal A}(a)
  :=\{v\in\mathbb R^d:\langle v,b-a\rangle\leq0
                       \text{ for every }b\in\mathcal A\}.
\]
If a differentiable function $q$ is maximized at $a$ over the convex
set $\mathcal A$, first-order optimality gives
$\nabla q(a)\in N_{\mathcal A}(a)$.

Using the extensions in \cref{ass:true-drift}, define the canonical
critic on the open action neighborhood by
\begin{equation}
  Q_h^\star(x,a)
  :=\delta r(x,a)+(\mathcal P^\star V_{h+1}^\star)(x,a),
  \qquad (x,a)\in\mathbb R^n\times\mathcal A^{\mathrm o},
  \label{eq:canonical-off-action-optimal-critic}
\end{equation}
where the Gaussian kernel uses the corresponding extension of the
drift; it agrees with the Bellman critic on $\mathbb R^n\times\mathcal A$. When it is differentiable, its normal-cone stationarity residual is
\begin{equation}
  d_h(x,a)
  :=\operatorname{dist}\bigl(\nabla_aQ_h^\star(x,a),
                              N_{\mathcal A}(a)\bigr).
  \label{eq:benchmark-stationarity-residual}
\end{equation}

\begin{assumption}[Quadratic advantage and value-gradient stability]
\label{ass:benchmark-error-bounds}
\label{ass:global-benchmark-geometry}
For every $h<H$, the canonical critic in
\eqref{eq:canonical-off-action-optimal-critic} is finite and belongs to
$C^1(\mathbb R^n\times\mathcal A^{\mathrm o})$, and its optimal envelope
$V_h^\star$ belongs to $C^1(\mathbb R^n)$.
There are constants $\bar m>0$ and $0\leq\kappa<\infty$, independent
of $h$ and the admissible timestep $\delta$, such that, for every
$x\in\mathbb R^n$ and $a\in\mathcal A$,
\begin{align}
  V_h^\star(x)-Q_h^\star(x,a)
  &\leq\frac{d_h(x,a)^2}{2\delta\bar m},
  \label{eq:quadratic-advantage-normal-cone}\\
  \|\nabla_xQ_h^\star(x,a)-\nabla V_h^\star(x)\|
  &\leq \kappa  d_h(x,a).
  \label{eq:value-gradient-stability}
\end{align}
\end{assumption}

The two inequalities play different roles:
The first converts stationarity error into
quadratic optimal-advantage loss; the second controls the effect of
action selection on the value gradient, so value errors can
propagate backward through the Bellman recursion in \cref{thm:residual-regret-transfer}.

At any optimal action $a^\star$, first-order optimality gives
$d_h(x,a^\star)=0$, so \eqref{eq:value-gradient-stability}  gives
\begin{equation}
  \nabla V_h^\star(x)=\nabla_xQ_h^\star(x,a^\star),
  \qquad a^\star\in\argmax_{a\in\mathcal A}Q_h^\star(x,a).
  \label{eq:optimal-envelope-gradient}
\end{equation}
These requirements allow nonunique optimal actions. They are stated here on the full action set; \cref{sec:localized-algorithm} uses them inside optimal bins by proving that fitted greedy actions enter those bins. Strong concavity with mixed-gradient control is one sufficient condition to \cref{ass:benchmark-error-bounds}:

\begin{proposition}[Strong concavity and mixed-gradient control]
\label{prop:strong-concavity-sufficient}
\label{prop:quadratic-advantage}
Suppose that for every $h<H$ the canonical critic in
\eqref{eq:canonical-off-action-optimal-critic} is finite and belongs to
$C^1(\mathbb R^n\times\mathcal A^{\mathrm o})$, and that there are
constants $\bar m>0$ and $0\leq\bar L_{xa}<\infty$, independent of $h$
and $\delta$, such that $Q_h^\star(x,\cdot)$ is
$\delta\bar m$-strongly concave and $\nabla_xQ_h^\star(x,\cdot)$ is
$\delta\bar L_{xa}$-Lipschitz on $\mathcal A$: for every
$x\in\mathbb R^n$ and $a,b\in\mathcal A$,
\begin{align*}
  \left\langle\nabla_aQ_h^\star(x,a)-\nabla_aQ_h^\star(x,b),
                   a-b\right\rangle
  &\leq-\delta\bar m\|a-b\|^2,\\
  \|\nabla_xQ_h^\star(x,a)-\nabla_xQ_h^\star(x,b)\|
  &\leq\delta\bar L_{xa}\|a-b\|.
\end{align*}
Then $V_h^\star\in C^1(\mathbb R^n)$ for every $h<H$ and satisfies \eqref{eq:optimal-envelope-gradient}.
\cref{ass:benchmark-error-bounds} holds with the same $\bar m$ and
$\kappa=\bar L_{xa}/\bar m$.
\end{proposition}

\begin{proof}
See \cref{app:global-geometry-proof}.
\end{proof}

\begin{remark}[Explanation of the curvature scaling]
\label{rem:global-euler-scaling}
The factors of $\delta$ match the Hamiltonian expansion of the Euler critic. Under smoothness and uniform derivative bounds justifying the expansion and its action derivatives,
\begin{equation*}
  Q_h^\star(x,a)=V_{h+1}^\star(x)+\delta\left\{
    r(x,a)+\mu^\star(x,a)^\top\nabla V_{h+1}^\star(x)
    +\tfrac12\operatorname{tr}\bigl(\Sigma^\star\nabla^2V_{h+1}^\star(x)\bigr)
  \right\}+o(\delta).
\end{equation*}
Uniform action-curvature and mixed-derivative bounds on the bracketed term thus give the order-$\delta$ scales in \cref{prop:strong-concavity-sufficient} when the remainder derivatives are controlled.
\end{remark}

At greedy actions, the stationarity
residual $d_h$ is controlled by the critic's action-gradient error.

\begin{corollary}[Performance of the greedy policy]
\label{cor:quadratic-performance-gradient-bound}
Under \cref{ass:benchmark-error-bounds}, whenever the squared
gradient errors below are integrable,
\begin{equation*}
    J_\delta^\star-J_\delta(\widehat\pi)
    \leq
    \frac{1}{2\delta\bar m}
    \sum_{h=0}^{H-1}
    \mathbb E_{\widehat\pi}
    \left[
        \left\|
            \nabla_a
            \bigl(
                \widehat Q_h-Q_h^\star
            \bigr)
            \bigl(
                X_h,\widehat\pi_h(X_h)
            \bigr)
        \right\|^2
    \right].
\end{equation*}
\end{corollary}
\begin{proof}
Fix \(h<H\) and \(x\in\mathbb R^n\), and set
\(\widehat a=\widehat\pi_h(x)\).
Since \(\widehat Q_h(x,\cdot)\) is differentiable on an open
neighborhood of \(\mathcal A\) and is maximized at \(\widehat a\),
first-order optimality on the convex set $\mathcal A$ gives
$\nabla_a\widehat Q_h(x,\widehat a)\in N_{\mathcal A}(\widehat a)$.
Using this normal vector in the distance defining $d_h$ gives
\begin{equation*}
    d_h(x,\widehat a)
    \leq
    \left\|
        \nabla_a(\widehat Q_h-Q_h^\star)(x,\widehat a)
    \right\|.
\end{equation*}
This estimate requires no curvature assumption. Substitution into
\eqref{eq:quadratic-advantage-normal-cone} gives
\begin{equation}
    V_h^\star(x)-Q_h^\star(x,\widehat a)
    \leq
    \frac{1}{2\delta\bar m}
    \left\|
        \nabla_a
        \bigl(
            \widehat Q_h-Q_h^\star
        \bigr)
        (x,\widehat a)
    \right\|^2.
    \label{eq:quadratic-advantage-gradient-error}
\end{equation}
The squared gradient errors are Borel measurable because
$\widehat Q_h$ and $Q_h^\star$ are $C^1$ and $\widehat\pi_h$ is Borel.
Taking expectations in \eqref{eq:quadratic-advantage-gradient-error},
using the assumed integrability, and summing over $h$ gives the result
by \cref{lem:optimal-advantage-pdi}, which applies under the standing
\cref{ass:basic}.
\end{proof}

The factor $1/\delta$ is the scale of \eqref{eq:quadratic-advantage-normal-cone}. The performance estimate needs action-gradient accuracy only at selected actions along the learned policy's trajectories. In particular, it imposes no concavity on $\widehat Q_h$.

\subsection{Residual accuracy and regret transfer}
The performance estimate reduces regret control to bounding critic
action-gradient errors at the selected actions.  We bound these by derivatives of the Bellman residuals under the true kernel.
The Euler structure and value-gradient stability control the backward
propagation, independently of how the fitted critics were constructed.

\subsubsection{Fitted critics and true residual certificates}
\label{subsec:critic-residual-interface}

Let \(\mathcal G_k\) be the sigma-field generated by real data through
episode \(k\) and all auxiliary randomness in the fitted-\(Q\) update.
In this subsection and \cref{subsec:direct-log-regret}, we condition on
\(\mathcal G_k\), fixing a critic sequence
\(\widehat Q_k\), its greedy envelope, and the greedy policy
\(\pi_{k+1}\) for episode \(k+1\). Expectations are over a fresh episode
under the true kernel.

We separate fitted-envelope regularity (\cref{def:fitted-envelope-interface}) from residual accuracy (\cref{def:true-residual-certificate}); the localized construction keeps the former and replaces the latter by the two-scale certificate of \cref{def:two-scale-true-residual-certificate}, introduced in
\cref{subsec:two-scale-main-guarantee}.

\begin{criterion}[Fitted-envelope regularity]
\label{def:fitted-envelope-interface}
Fix an update index \(k\). A fitted critic sequence
\(
    \widehat Q_k
    =
    (\widehat Q_{k,0},\ldots,\widehat Q_{k,H-1})
\)
satisfies the \emph{fitted-envelope regularity criterion} if,
for every \(h<H\),
\begin{enumerate}[label=(\roman*),leftmargin=2.2em]
    \item \(\widehat Q_{k,h}\) has a \(C^1\) extension to
    \(\mathbb R^n\times\mathcal A^{\mathrm o}\);

    \item its full-action envelope belongs to
    \(C^1(\mathbb R^n)\):
    \[
    \widehat V_{k,h}(x):=\max_{a\in A}\widehat Q_{k,h}(x,a)
    \in C^1(\mathbb R^n).
    \]
\end{enumerate}
At the terminal stage, set \(\widehat V_{k,H}:=u.\)
\end{criterion}

\Cref{def:fitted-envelope-interface} supplies the differentiability used by the envelope identity and backward transfer; it is verified for the algorithm in \cref{prop:algorithm-output-interface}.

\begin{lemma}[Greedy policy and envelope identity]
\label{lem:fitted-envelope-identity}
Suppose \cref{def:fitted-envelope-interface} holds. Then there exists
a deterministic Borel Markov policy \(\pi_{k+1}\) satisfying, for every
\(h<H\) and \(x\in\mathbb R^n\),
\begin{equation}
    \pi_{k+1,h}(x)
    \in
    \argmax_{a\in\mathcal A}\widehat Q_{k,h}(x,a).
    \label{eq:fitted-greedy-policy}
\end{equation}
For each such \(h\) and \(x\), every maximizing action \(a\), including the \(\pi_{k+1,h}(x)\), also satisfies
\begin{equation}
    \nabla\widehat V_{k,h}(x)
    =
    \nabla_x\widehat Q_{k,h}(x,a).
    \label{eq:selected-envelope-identity}
\end{equation}
\end{lemma}

\begin{proof}
See \cref{app:fitted-envelope-proof}.
\end{proof}

Define the true Bellman residual of the realized fitted critic by
\begin{equation}
    e_{k,h}(x,a)
    :=
    \delta r(x,a)
    +
    (\mathcal{P}^\star\widehat V_{k,h+1})(x,a)
    -
    \widehat Q_{k,h}(x,a),
    \qquad h<H.
    \label{eq:true-residual}
\end{equation}
Its extension to $\mathbb R^n\times\mathcal A^{\mathrm o}$ is the same Bellman expression with the fixed extensions of drift, reward, and critic; \cref{def:true-residual-certificate,def:two-scale-true-residual-certificate} require it to be finite, so the next-stage fitted value is absolutely integrable, and $C^1$. 
Set
\begin{equation*}
    \Delta Q_{k,h}
    :=
    \widehat Q_{k,h}-Q_h^\star,
    \qquad
    \Delta V_{k,h}
    :=
    \widehat V_{k,h}-V_h^\star,
    \qquad
    \Delta V_{k,H}=0.
\end{equation*}
\stepcounter{criterion}
\begin{criterionvariant}[Bellman-residual derivative certificate]
\label{def:true-residual-certificate}
Suppose \cref{def:fitted-envelope-interface} holds, and
choose the greedy policy in \eqref{eq:fitted-greedy-policy}. For
nonnegative radii \(d_{x,k}\) and \(d_{a,k}\) independent of $\delta$, we say that the fitted
critic sequence satisfies a \emph{true-residual derivative certificate
at the selected greedy actions} if every \(e_{k,h}\) has a \(C^1\) extension to
\(\mathbb R^n\times\mathcal A^{\mathrm o}\) and
\begin{equation}
    \sup_{h<H}\sup_{x\in\mathbb R^n}
    \left\|
        \nabla_xe_{k,h}
        \bigl(x,\pi_{k+1,h}(x)\bigr)
    \right\|
    \leq
    \delta d_{x,k},
    \qquad
    \sup_{h<H}\sup_{x\in\mathbb R^n}
    \left\|
        \nabla_ae_{k,h}
        \bigl(x,\pi_{k+1,h}(x)\bigr)
    \right\|
    \leq
    \delta d_{a,k}.
    \label{eq:residual-derivative-certificate}
\end{equation}
The radii may be \(\mathcal G_k\)-measurable; after conditioning on
the post-update history, they are fixed numbers.
\end{criterionvariant}

\Cref{def:true-residual-certificate} controls the state and action
residual derivatives at the fitted greedy actions, and supplies the accuracy input to
\cref{thm:residual-regret-transfer}.  
Since $d_{a,k}$ and $d_{x,k}$ are independent of $\delta$, the derivative bounds 
scale linearly in $\delta$, matching the Euler scaling of the action and state derivatives in \cref{lem:kernel-differentiation}.
The localized extension uses
\cref{def:two-scale-true-residual-certificate} to establish bin entry
before applying local derivative control; that certificate is verified
in \cref{thm:two-scale-true-residual}.

\subsubsection{Residual derivatives to regret transfer}
\label{subsec:direct-log-regret}
We now connect the residual derivative bounds in
\cref{def:true-residual-certificate} to regret. In \cref{ass:benchmark-error-bounds}, value-gradient stability closes a backward recursion for the value-gradient error, and the quadratic advantage bound converts action-gradient error into policy loss.
The transfer holds for arbitrary finite residual radii; a
root-$k$ rate then gives logarithmic regret in
\cref{thm:residual-regret-transfer}.
The true residual and critic-error definitions give
\[
  \Delta Q_{k,h}=\mathcal{P}^\star\Delta V_{k,h+1}-e_{k,h}.
\]
This identity propagates critic error backward through the true Euler
kernel. Since the covariance is constant, differentiating the kernel 
acts only through its mean, so action derivatives carry a factor 
$\delta$ and state derivatives are $I_n+\mathcal{O}(\delta)$.

\begin{lemma}[Differentiation of the Euler kernel]
\label{lem:kernel-differentiation}
Let \(v\in C^1(\R^n)\) satisfy \(\norm{\nabla v}_\infty<\infty\).
Under \cref{ass:true-drift}, the function
\((x,a)\mapsto(\mathcal{P}^\star v)(x,a)\) belongs to
\(C^1(\R^n\times\mathcal A^{\mathrm o})\), and
\begin{align*}
  \nabla_x(\mathcal{P}^\star v)(x,a)
  &=
  \{I_n+\delta\partial_x\mu^\star(x,a)\}^\top
  \mathcal{P}^\star[\nabla v](x,a),\\
  \nabla_a(\mathcal{P}^\star v)(x,a)
  &=
  \delta\{\partial_a\mu^\star(x,a)\}^\top
  \mathcal{P}^\star[\nabla v](x,a).
\end{align*}
\end{lemma}

\begin{proof}
See \cref{app:kernel-differentiation-proof}.
\end{proof}

For \(\tau\in[0,T]\) and \(L\ge0\), define the physical-time
propagation factor
\begin{equation*}
  \Phi_\tau(L)
  :=
  \begin{cases}
    \dfrac{e^{L\tau}-1}{L},&L>0,\\[1ex]
    \tau,&L=0.
  \end{cases}
\end{equation*}

\begin{theorem}[Logarithmic regret from root-$k$ residual accuracy]
\label{thm:residual-regret-transfer}
Let \(K\ge 2\). Suppose
\cref{ass:benchmark-error-bounds,ass:true-drift} hold. Suppose further
that, simultaneously for every update \(k=1,\ldots,K-1\), the fitted
critic satisfies
\cref{def:fitted-envelope-interface,def:true-residual-certificate} with
radii obeying
\begin{equation}
  d_{x,k}\leq\frac{c_x}{\sqrt{k}},
  \qquad
  d_{a,k}\leq\frac{c_a}{\sqrt{k}},
  \qquad k=1,\ldots,K-1,
  \label{eq:root-k-residual-rate}
\end{equation}
for constants \(c_x,c_a\geq0\) independent of \(k,K,\delta\).
Set
\[
  L_E:=L_{\mu,x}+\kappa L_{\mu,a},
  \qquad
  C:=c_a+L_{\mu,a}\Phi_T(L_E)(c_x+\kappa c_a).
\]
Then
\begin{equation}
  \Reg(K,\delta)
  \leq2B_J+\frac{TC^2}{2\bar m}\sum_{k=1}^{K-1}\frac1k
  \leq2B_J+\frac{TC^2}{2\bar m}\bigl(1+\log(K-1)\bigr).
  \label{eq:direct-logarithmic-regret}
\end{equation}
For fixed \(T\), the bound is \(\mathcal{O}(\log K)\) uniformly over admissible
timesteps $\delta $.
\end{theorem}

\begin{proof}
\underline{\emph{Step 1: residual-to-gradient transfer.}}
Steps 1 and 2 use only that the radii are finite and nonnegative; the
rate \eqref{eq:root-k-residual-rate} enters in Step 3.  For each update,
define
\begin{equation*}
  b_{E,k}:=d_{x,k}+\kappa d_{a,k},
  \qquad
  c_k:=d_{a,k}+L_{\mu,a}\Phi_T(L_E)b_{E,k}.
\end{equation*}
Fix an update index \(k\) and condition on \(\mathcal G_k\) as in
\cref{subsec:critic-residual-interface}.
By \cref{def:fitted-envelope-interface,ass:benchmark-error-bounds},
$\Delta V_{k,h}\in C^1(\R^n)$ for $h<H$, and $\Delta V_{k,H}=0$.
Define the possibly infinite quantities
\[
  E_{k,h}
  :=
  \sup_{x\in\R^n}\norm{\nabla\Delta V_{k,h}(x)}.
\]
We prove finiteness and bounds by backward induction from $E_{k,H}=0$.

Assume that \(E_{k,h+1}<\infty\).
Then the $C^1$ value error satisfies
\[
  |\Delta V_{k,h+1}(y)|
  \leq |\Delta V_{k,h+1}(0)|+E_{k,h+1}\|y\|,
\]
so it is absolutely integrable under every Euler transition, including
those with $a\in\mathcal A^{\mathrm o}$.  The optimal value function
has finite expectation by \cref{ass:benchmark-error-bounds}; hence
$\widehat V_{k,h+1}=V_{h+1}^\star+\Delta V_{k,h+1}$ is also integrable.  This justifies the error identity on the
open domain before differentiating it. The kernel lemma needs only the bounded gradient of $\Delta V_{k,h+1}$, not of $\widehat V_{k,h+1}$.
Applying \cref{lem:kernel-differentiation} to
\(\Delta Q_{k,h}=\mathcal P^\star\Delta V_{k,h+1}-e_{k,h}\) gives, on
\(\R^n\times\mathcal A\),
\begin{align}
  \nabla_a\Delta Q_{k,h}(x,a)
  &=
  \delta\{\partial_a\mu^\star(x,a)\}^\top
  \mathcal{P}^\star[\nabla\Delta V_{k,h+1}](x,a)
  -\nabla_ae_{k,h}(x,a),
  \notag\\
  \nabla_x\Delta Q_{k,h}(x,a)
  &=
  \{I_n+\delta\partial_x\mu^\star(x,a)\}^\top
  \mathcal{P}^\star[\nabla\Delta V_{k,h+1}](x,a)
  -\nabla_xe_{k,h}(x,a).
  \label{eq:state-error-identity}
\end{align}
Fix \(x\in\R^n\) and write \(\widehat a=\pi_{k+1,h}(x)\).
The drift bounds \eqref{eq:true-drift-jacobians} and the residual bounds
\eqref{eq:residual-derivative-certificate} of
\cref{def:true-residual-certificate} imply
\begin{equation}
  \norm{
    \nabla_a\Delta Q_{k,h}
      (x,\widehat a)
  }
  \le
  \delta d_{a,k}
  +\delta L_{\mu,a}E_{k,h+1}.
  \label{eq:action-error-bound}
\end{equation}

By \cref{lem:fitted-envelope-identity} and
\eqref{eq:optimal-envelope-gradient},
\begin{align*}
  \nabla\Delta V_{k,h}(x)
  &=
  \nabla_x\Delta Q_{k,h}(x,\widehat a)
  +\nabla_xQ_h^\star(x,\widehat a)
  -\nabla V_h^\star(x).
\end{align*}
Value-gradient stability \eqref{eq:value-gradient-stability} and the
greedy stationarity argument in the proof of
\cref{cor:quadratic-performance-gradient-bound} give
\[
  \norm{
    \nabla_xQ_h^\star(x,\widehat a)
    -\nabla V_h^\star(x)
  }
  \le
  \kappa d_h(x,\widehat a)
  \le
  \kappa
  \norm{\nabla_a\Delta Q_{k,h}(x,\widehat a)}.
\]
Combining this with \eqref{eq:state-error-identity}, the same drift and
residual bounds, and \eqref{eq:action-error-bound} gives, for every $x$,
\begin{align*}
  \norm{\nabla\Delta V_{k,h}(x)}
  &\leq \delta d_{x,k}+(1+\delta L_{\mu,x})E_{k,h+1}
    +\delta\kappa
      \bigl(d_{a,k}+L_{\mu,a}E_{k,h+1}\bigr)\\
  &=\delta b_{E,k}+(1+\delta L_E)E_{k,h+1}.
\end{align*}
The right-hand side is finite and independent of $x$. Taking the
supremum therefore gives
\begin{equation}
  E_{k,h}\leq\delta b_{E,k}+(1+\delta L_E)E_{k,h+1}<\infty,
  \label{eq:E-recursion}
\end{equation}
closing the induction.
Iterating \eqref{eq:E-recursion} from \(E_{k,H}=0\) gives
\begin{equation}
  E_{k,h}
  \le
  \delta b_{E,k}
  \sum_{j=0}^{H-h-1}(1+\delta L_E)^j
  \le
  b_{E,k}\Phi_{T-t_h}(L_E)
  \le b_{E,k}\Phi_T(L_E).
  \label{eq:value-gradient-transfer}
\end{equation}
The geometric-sum bound also holds at \(L_E=0\) by definition of
\(\Phi\). Substitution into \eqref{eq:action-error-bound} gives
\begin{equation}
  \sup_{x\in\R^n}
  \norm{\nabla_a\Delta Q_{k,h}\bigl(x,\pi_{k+1,h}(x)\bigr)}
  \leq\delta\left\{d_{a,k}
    +L_{\mu,a}\Phi_{T-t_{h+1}}(L_E)b_{E,k}\right\}
  \leq\delta c_k.
  \label{eq:action-gradient-transfer}
\end{equation}

\underline{\emph{Step 2: one-episode loss and cumulative regret.}}
Conditional on $\mathcal G_k$, each squared action-gradient error is
Borel measurable and bounded by $\delta^2c_k^2$ by
\eqref{eq:action-gradient-transfer}, hence integrable under the fresh
episode law. Thus \cref{cor:quadratic-performance-gradient-bound}
applies and, using \(H\delta=T\), gives
\begin{equation*}
  J_\delta^\star-J_\delta(\pi_{k+1})
  \leq\frac{H\delta^2c_k^2}{2\delta\bar m}
  =\frac{T}{2\bar m}c_k^2.
\end{equation*}
The factor \(\delta^2\) therefore cancels both the factor
\(1/\delta\) in the quadratic advantage bound and the stage count
\(H=T/\delta\). Summing over updates and bounding the initial episode
by \(2B_J\) gives, for arbitrary finite radii,
\begin{equation}
  \Reg(K,\delta)
  \leq2B_J+\frac{T}{2\bar m}\sum_{k=1}^{K-1}c_k^2.
  \label{eq:arbitrary-radius-cumulative-regret}
\end{equation}

\underline{\emph{Step 3: root-$k$ rate.}}
By \eqref{eq:root-k-residual-rate},
\(b_{E,k}\leq(c_x+\kappa c_a)/\sqrt{k}\), hence \(c_k\leq C/\sqrt{k}\).
Substituting into \eqref{eq:arbitrary-radius-cumulative-regret} and
using \(\sum_{k=1}^{K-1}k^{-1}\leq1+\log(K-1)\) proves
\eqref{eq:direct-logarithmic-regret}.
\end{proof}

\begin{remark}[Scope of the benchmark inequalities]
\label{rem:value-gradient-stability-relaxation}
On the full action set, the quadratic-advantage inequality implies that any action with $d_h(x,a)=0$ is optimal. The proof of \cref{thm:residual-regret-transfer}, however, uses both inequalities in \cref{ass:benchmark-error-bounds} only at the fitted greedy actions. \Cref{sec:localized-algorithm} exploits this by establishing bin entry together with backward transfer, so each selected action lies where the inequalities hold.
\end{remark}

\begin{remark}[Dependence on physical time]
The factor in \eqref{eq:value-gradient-transfer} depends on the remaining physical time \(T-t_h\), not on the number of remaining stages, so the recursion introduces no polynomial factor in \(H\) when \(L_{\mu,x},L_{\mu,a},\kappa,d_{x,k},d_{a,k}\) are uniform in \(\delta\). For $L_E>0$, $\Phi_T(L_E)=\mathcal{O}(e^{L_E T})$, so the regret bound is $\mathcal{O}(T e^{2L_E T}\log K)$, which is discretization-uniform for fixed $T$. It extends to nonuniform grids with steps $\Delta_h\leq\delta$ (\cref{app:nonuniform-time-grids}). By \cref{prop:physical-time-lower-bound}, the exponential rate cannot be improved under the hypotheses of \cref{thm:residual-regret-transfer}.
\end{remark}

\Cref{prop:physical-time-lower-bound} shows that exponential growth of the large-$K$ logarithmic-regret coefficient is unavoidable under these general transfer conditions.

\begin{proposition}[Sharpness of the residual-to-regret transfer]
\label{prop:physical-time-lower-bound}
Fix $\lambda,m,c_x>0$. There exist a one-dimensional model with
structural constants independent of $T$ and $\delta$, and fitted critics
satisfying the hypotheses of \cref{thm:residual-regret-transfer} with
$L_E=\lambda$, $\bar m=m$, $\kappa=0$, $d_{x,k}=c_x/\sqrt{k}$, and $d_{a,k}=0$,
whose regret satisfies
\begin{equation}
  \lim_{\delta\downarrow0}\lim_{K\to\infty}
  \frac{\Reg(K,\delta)}{c_x^2\log K}
  =\frac1{2m}\int_0^T\Phi_t(\lambda)^2\,dt
  \sim\frac{e^{2\lambda T}}{4m\lambda^3}
  \quad(T\to\infty),
  \label{eq:physical-time-continuum-lower-bound}
\end{equation}
where the timestep limit is along admissible refinements at fixed $T$.
\end{proposition}

\begin{proof}
See \cref{app:physical-time-lower-bound-proof}.
\end{proof}

The rate is attained with $\kappa=0$. 
This is a sharpness result for
bounds derived from
the general residual-to-regret transfer, though not a
minimax lower bound over learning algorithms. Its unbounded-drift example used in the proof lies outside the localized
model class in \cref{sec:localized-algorithm}.

\section{Localized learning: algorithm and main result}
\label{sec:localized-algorithm}
\Cref{sec:model-global-regret} identifies sufficient benchmark
properties and residual accuracy for fast regret.  We now localize the mechanism for a concrete learning algorithm.
With several separated optimal branches, the benchmark inequalities
may hold only inside unknown bins.  The learner still maximizes over
the full action set, so the proof must establish that its fitted
greedy actions lie in regions where the inequalities apply.

The two-scale residual control in
\cref{def:two-scale-true-residual-certificate} addresses this requirement.
A coarse action-residual bound establishes bin
entry jointly with backward transfer.  Sharper residual bounds inside
the bins then give local gradient control and quadratic policy loss.
The local geometry also bounds distance to the optimal-action set,
which supports the coverage and self-exploration arguments across epochs.
The algorithm does not use the unknown bins; they enter only in the
analysis.
\Cref{tab:localized-notation} in \cref{app:localized-notation} collects
the notation and constants used here and in the proof.

\subsection{Model assumptions and setup}
\label{subsec:critic-and-approximation}

This subsection states the assumptions together with the notation
needed.
\cref{ass:localized-parametric-model,ass:localized-optimal-branches,ass:exploration-self-excitation} 
restrict the environment; \cref{ass:envelope-stable-critic-class,ass:predictable-critic-design,ass:localized-MC-budget} 
are choices the learner controls; \cref{ass:projected-two-scale-Bellman-approximation} links the two. 
Basic structural consequences are stated alongside these assumptions; verification of
fitted-output regularity is given in \cref{subsec:main-criteria-verification}.

\subsubsection{Experiment, epochs, and the parametric model.}
Fix a target number of episodes $K\geq2$ and one admissible timestep $0<\delta\leq\delta_{\max}$ throughout the localized construction.  All probabilities refer to this fixed-$\delta$ 
experiment; uniformity in $\delta$ refers only to deterministic constants.

Let $n_0\geq1$ be the planned
number of initial exploration episodes. If
$K\leq n_0$, the learner explores for all $K$ episodes and the trivial
bound $\Reg(K,\delta)\leq2B_JK$ applies.  Henceforth consider the nontrivial
case $1\leq n_0<K$.
After the exploration episodes, the learner uses the doubling schedule
\begin{equation}
  n_j=2^j n_0,\qquad j=0,1,\ldots,
  \label{eq:doubling-schedule}
\end{equation}
with epochs indexed by $j=0,1,\ldots,J_K$, where $J_K$ is the last
index for which $n_j<K$.  At the beginning of epoch $j$, the learner
has observed $n_j$ complete episodes, constructs a policy $\pi^{(j)}$,
and deploys it for the next $n_j$ episodes, truncating the final epoch
at episode $K$ if necessary.  Hence
\begin{equation}
  J_K+1=\left\lceil\log_2(K/n_0)\right\rceil
  \leq \bar J_K,
  \qquad
  \bar J_K:=1+\lceil\log_2K\rceil.
  \label{eq:epoch-count-bound}
\end{equation}
The deterministic bound $\bar J_K$ is used in confidence allocations,
avoiding circular dependence between $J_K$ and the number of exploration episodes.

\begin{assumption}[Parametric Euler dynamics and reward access]
\label{ass:localized-parametric-model}
There is a known feature map
$\Phi_\mu:\mathbb R^n\times\mathcal A\to\mathbb R^{n\times p}$ and an unknown
parameter $\theta^\star\in\mathbb R^p$ such that
\begin{equation*}
  \mu^\star(x,a)=\Phi_\mu(x,a)\theta^\star,
  \qquad
  \|\theta^\star\|\leq B_\theta.
\end{equation*}
The parameter bound $B_\theta$ is known.
The map $\Phi_\mu$ has a $C^1$ extension to
$\mathbb R^n\times\mathcal A^{\mathrm o}$ and satisfies
\begin{equation*}
  \|\Phi_\mu\|_\infty\leq B_{\Phi_\mu},
  \qquad
  \|\partial_x\Phi_\mu\|_\infty\leq L_{\Phi_\mu,x},
  \qquad
  \|\partial_a\Phi_\mu\|_\infty\leq L_{\Phi_\mu,a}.
\end{equation*}
The learner knows constants
$0<\lambda_0\leq\Lambda_0$, independent of the timestep, episode budget,
and confidence levels, such that $\lambda_0 I_n\preceq\Sigma^\star\preceq\Lambda_0 I_n$.

The stage reward $r$ and terminal reward $u$ are known to the learner.  The stage reward
$r$ has a $C^1$ extension to
$\mathbb R^n\times\mathcal A^{\mathrm o}$ with bounded first derivatives,
and the terminal reward $u \in C^{1,1}(\mathbb R^n)$ with bounded
first derivative.  
\end{assumption}

A linear-in-parameter drift with known features and an unknown bounded parameter is standard in continuous-time episodic RL \citep{basei2022logarithmic}, \citep[(H.3(1))]{guo2021linearconvex}. In discrete time the
counterpart is the linear MDP.

\begin{remark}[Regularity and a sufficient moment condition]
\label{lem:parametric-growth-moment}
Suppose \cref{ass:localized-parametric-model}
holds, then \cref{ass:true-drift} holds with
$L_{\mu,s}=B_\theta L_{\Phi_\mu,s}$, $s\in\{x,a\}$, and $\|\mu^\star\|_\infty\leq B_{\Phi_\mu}B_\theta$.
If, in addition, $\int\|x\|\,p_0(dx)<\infty$, then \cref{ass:basic}
holds with
\[
  W(x)=\sqrt{1+\|x\|^2},
  \qquad c_W=B_{\Phi_\mu}B_\theta+\frac{n\Lambda_0}{2},
\]
and the following time-homogeneous constants:
\(
  B_r=\sup_{a\in\mathcal A}|r(0,a)|+\|\nabla_xr\|_\infty,
B_u=|u(0)|+\|\nabla u\|_\infty.
\)
See proof in \cref{app:parametric-growth-moment-proof}.
\end{remark}

\begin{remark}[Stage-dependent extension]
\label{rem:stage-dependent-features}
All results extend directly to known stage-dependent features, 
with each feature map satisfying the same level and derivative bounds
uniformly in $h$.
\end{remark}

\subsubsection{Optimal-action sets and local bin geometry.}
For every $h<H$ and $x\in\mathbb R^n$, define the full optimal action set
\begin{equation*}
  \mathcal A_h^\star(x)
  :=
  \argmax_{a\in\mathcal A}Q_h^\star(x,a).
\end{equation*}

\begin{assumption}[Separated bins with local strong concavity]
\label{ass:localized-optimal-branches}
There exist constants
$M_\star<\infty$, $\bar m>0$, $\bar L_{xa}<\infty$,
$\gamma>0$, independent of $h$ and $\delta$, with the
following properties.  For every $(h,x)$, there are at most $M_\star$
closed convex bins
\(
  B_{h,\ell}^\star(x)\subseteq\mathcal A,\,
  \ell\in\mathcal I_h(x),
\)
whose graphs are Borel measurable.  Define
\begin{equation*}
  \mathcal U_h^\star(x)
  :=
  \bigcup_{\ell\in\mathcal I_h(x)}B_{h,\ell}^\star(x).
\end{equation*}
Then:
\begin{enumerate}[label=(\roman*),leftmargin=2.2em]
  \item Every optimal action belongs to the union, and every bin contains a
  unique globally optimal branch point:
  \begin{equation*}
    \mathcal A_h^\star(x)\subseteq\mathcal U_h^\star(x),
    \qquad
    \mathcal A_h^\star(x)\cap B_{h,\ell}^\star(x)
    =\{a_{h,\ell}^\star(x)\}.
  \end{equation*}

  \item On each bin, the benchmark is strongly concave and its state
  gradient is Lipschitz in the action: for all
  $a,b\in B_{h,\ell}^\star(x)$,
  \begin{align}
    \left\langle
      \nabla_aQ_h^\star(x,a)-\nabla_aQ_h^\star(x,b),
      a-b
    \right\rangle
    &\leq-\delta\bar m\|a-b\|^2,
    \notag\\
    \|\nabla_xQ_h^\star(x,a)-\nabla_xQ_h^\star(x,b)\|
    &\leq\delta\bar L_{xa}\|a-b\|.
    \label{eq:branch-mixed-gradient}
  \end{align}

  \item There is a normalized margin outside the union of all bins:
  \begin{equation}
    V_h^\star(x)
    -
    \sup_{a\in\mathcal A\setminus\mathcal U_h^\star(x)}
      Q_h^\star(x,a)
    \geq\delta\gamma,
    \label{eq:union-entry-margin}
  \end{equation}
  with the condition understood as vacuous when
  $\mathcal U_h^\star(x)=\mathcal A$.

\end{enumerate}
\end{assumption}

Part (ii) localizes the standard global action concavity of the linear-convex 
literature, which is the case $M_\star=1$ \citep[(H.1(2))]{szpruch2021exploration}, 
\citep[(H.1(3))]{guo2021linearconvex}; with $M_\star>1$, optimal actions may be nonunique. 
Part (iii) is a gap condition in the spirit of \citet[Assumption~3.3]{he2021logarithmic}, but weaker: 
it is required only outside the union of bins, and its $\delta$ normalization keeps it from 
degenerating as $\delta\downarrow0$.

Assumption \ref{ass:localized-optimal-branches} provides local sufficient 
conditions for the benchmark inequalities in \cref{ass:global-benchmark-geometry}. 
Together with the regularity established in Lemma \cref{lem:optimal-envelope-regularity}, 
it gives both inequalities on the union of optimal bins, with \(\kappa=\bar L_{xa}/\bar m\). 
Theorem \cref{thm:backward-union-entry} establishes that the fitted greedy actions enter this union.

\begin{lemma}[Optimal-envelope regularity from local geometry]
\label{lem:optimal-envelope-regularity}
Under
\cref{ass:localized-parametric-model,ass:localized-optimal-branches},
the following conclusions hold.
For each $h<H$ the optimal critic $Q_h^\star \in
C^1(\mathbb R^n\times\mathcal A^{\mathrm o})$, and its optimal action set
$\mathcal A_h^\star(x)$ is nonempty and compact, with a closed graph as a
correspondence in $x$.  Moreover,
$V_h^\star\in C^1(\mathbb R^n)$ and
\begin{equation}
  \nabla V_h^\star(x)
  =
  \nabla_xQ_h^\star(x,a)
  =
  \nabla_xQ_h^\star(x,a'),
  \qquad
  a,a'\in\mathcal A_h^\star(x).
  \label{eq:optimal-active-gradient-coherence}
\end{equation}
The common active gradient is continuous in $x$ and bounded uniformly
in $h$ and $\delta$.
\end{lemma}
\begin{proof}
  The proof is given in \cref{app:optimal-envelope-regularity}.
\end{proof}

The uniform positive margin in \eqref{eq:union-entry-margin}, together
with concavity inside each bin, ensures that every optimal action at a
state is approached by optimal actions at nearby states.  This forces the active state gradients to agree.  The lemma gives $C^1$
regularity for each $(h,\delta)$ and a uniform gradient bound.

\subsubsection{Coverage and self-exploration}
For a deterministic Markov policy $\pi$, define its per-episode population
information matrix
\begin{equation*}
  \Gamma(\pi)
  :=
  \E_\pi\left[
    \frac1H\sum_{h=0}^{H-1}
    \Phi_\mu(X_h,\pi_h(X_h))^\top
    \Phi_\mu(X_h,\pi_h(X_h))
  \right].
\end{equation*}
For analysis, define the policy neighborhood of the optimal-action sets for a $\varrho>0$:
\begin{equation*}
  \Pi_{\rm loc}(\varrho)
  :=
  \left\{
    \pi:
    \sup_{h<H}\sup_{x\in\mathbb R^n}
    \dist(\pi_h(x),\mathcal A_h^\star(x))
    \leq \varrho
  \right\}.
\end{equation*}
The learner need not know these sets or any policy in this class. We prove its 
fitted greedy policies belong to such neighborhood after the initial exploration phase.

\begin{assumption}[Coverage and self-exploration of optimal policies]
\label{ass:exploration-self-excitation}
There are a known Borel exploration policy $\pi^{\rm exp}$ and known constants
$\lambda_{\rm exp},\lambda_{\rm loc}>0$, uniform over the admissible
timesteps, such that
\begin{equation}
  \Gamma(\pi^{\rm exp})\succeq2\lambda_{\rm exp}I_p,
  \label{eq:exploration-population-coverage}
\end{equation}
and every Borel optimal policy $\pi^\star$ satisfies
\begin{equation}
  \Gamma(\pi^\star)\succeq4\lambda_{\rm loc}I_p.
  \label{eq:optimal-policy-self-excitation}
\end{equation}
Here an optimal policy is a deterministic Markov policy with
$\pi_h^\star(x)\in\mathcal A_h^\star(x)$ for every $(h,x)$.
Set $\lambda_\mu=\min\{\lambda_{\rm exp},\lambda_{\rm loc}\}$. 
\end{assumption}

The two conditions have distinct roles. Coverage \eqref{eq:exploration-population-coverage}
asks for a known policy whose trajectories are informative about 
the unknown drift parameter; it imposes no requirement that this policy 
be close to optimal, either in its actions or its return. By contrast, 
self-exploration \eqref{eq:optimal-policy-self-excitation} of optimal policies is a structural property of the true environment. 
Similar coverage and self-exploration conditions appear in \citet[(H.3, Theorem 3.8)]{szpruch2021exploration}, 
\citet[(H.4)]{guo2021linearconvex}, and \citet[(H.1(2))]{basei2022logarithmic}.

We do \emph{not} assume self-exploration of learned policies. 
Instead, they are derived once entry has been established, from the following extension to a neighborhood
of the optimal-action sets.

\begin{lemma}[Neighborhood self-exploration from optimal policies]
\label{lem:optimal-policy-neighborhood-excitation}
Suppose
\cref{ass:localized-parametric-model,ass:localized-optimal-branches,ass:exploration-self-excitation}
hold. There is a known radius $\varrho_{\rm cov}>0$,
independent of $\delta$, such that, uniformly over admissible timesteps,
\begin{equation}
  \Gamma(\pi)\succeq2\lambda_{\rm loc}I_p,
  \qquad \pi\in\Pi_{\rm loc}(\varrho_{\rm cov}).
  \label{eq:local-self-excitation}
\end{equation}
\end{lemma}

\begin{proof}
The quantitative bound and an explicit choice of radius are
given in \cref{app:optimal-envelope-regularity}.
\end{proof}

In the remainder of the paper, fix a known positive radius $\varrho_{\rm cov}$
as in \cref{lem:optimal-policy-neighborhood-excitation}.

\subsubsection{Critic class and feasible envelopes}
Fix an admissible timestep $\delta > 0$. At epoch $j$ and
stage $h$, use the affine critic family with coefficient dimension
$q_j\in\mathbb N$:
\begin{equation}
  Q_{h,j,w}(x,a)
  :=q_{0,h,j}(x,a)+\Psi_{h,j}(x,a)^\top w,
  \qquad
  w\in\mathcal C_{h,j}\subset\mathbb R^{q_j}.
  \label{eq:epoch-critic-sieve}
\end{equation}
The underlying sieve level is chosen deterministically from $(n_j,\delta)$
before the Monte Carlo labels are generated. For each $\delta$, the level
is nondecreasing in $n_j$, and the coefficient dimension is nondecreasing
with the sieve level. At a fixed sieve level, the features $\Psi_{h,j}(x,a)$ do not scale
with $\delta$, while offsets $q_{0,h,j}(x,a)$ and feasible sets $C_{h,j}$ may.
This lets the critic carry an order-one value level and order-$\delta$ action differences, 
which is the separation of scales behind continuous-time $q$-learning \citep{jia2023qlearning}. 
For example, $q_{0,h,j}(x,a)=\delta r(x,a)$ builds in the known reward.

\begin{assumption}[Envelope-stable shape-preserving critic class]
\label{ass:envelope-stable-critic-class}
For every admissible timestep and every $(j,h)$, the set
$\mathcal C_{h,j}$ is nonempty, compact, and convex.
The offsets and features have $C^1$ extensions to
$\mathbb R^n\times\mathcal A^{\mathrm o}$. At each fixed sieve level,
the quantities
\begin{equation*}
  \sup_h\sup_{w\in\mathcal C_{h,j}}\|w\|,
  \qquad \sup_h\|\Psi_{h,j}\|_\infty,
  \qquad \sup_h\|\nabla_x q_{0,h,j}\|_\infty,\quad \sup_h\|\nabla_a q_{0,h,j}\|_\infty
\end{equation*}
have known finite upper bounds independent of the admissible timestep.
The feature derivatives satisfy
\begin{equation*}
  \sup_h\sup_{x,a}\|\nabla_s\Psi_{h,j}(x,a)\|_{\rm op}
  \leq L_{\Psi,s,j},
  \quad s\in\{x,a\}.
\end{equation*}
Here $L_{\Psi,s,j}$ is finite and known; it may depend on the sieve
level but is independent of the admissible timestep at a fixed level.

There are constants $G_0,M_0$, independent of the sieve level and the
admissible timestep, such that active state gradients are bounded and
Lipschitz across states: uniformly in $(j,h,w)$, for every $x,y$ and
any actions $a_x,a_y$ maximizing
$Q_{h,j,w}(x,\cdot)$ and $Q_{h,j,w}(y,\cdot)$, respectively,
\begin{equation}
  \|\nabla_xQ_{h,j,w}(x,a_x)\|\leq G_0,
  \qquad
  \|\nabla_xQ_{h,j,w}(x,a_x)
    -\nabla_xQ_{h,j,w}(y,a_y)\|
  \leq M_0\|x-y\|.
  \label{eq:fitted-active-gradient-Lipschitz}
\end{equation}
The terminal reward satisfies the same $C^{1,1}$ bounds 
$\|\nabla u\|_\infty\le G_0$ and $[\nabla u]_{\rm Lip}\le M_0$. 

\end{assumption}

The compact convex feasible set $\mathcal C_{h,j}$ with
uniform bounds is analogous to Definition 1.1 and Definition 3.1 of
\citet{szpruch2021exploration}: in both cases the
purpose is to keep the learned object in a set on which regularity
constants are uniform.
The fitted-output consequences of this class condition are proved in
\cref{subsec:main-criteria-verification}, where
\cref{prop:algorithm-output-interface}
verifies \cref{def:fitted-envelope-interface} for the algorithm.

\subsubsection{Synthetic design and projected Bellman approximation}
The design matrices below enter the approximation assumption. Simulation
budgets are discussed in \cref{subsec:estimation-scheme}.

\begin{assumption}[Predictable synthetic critic design]
\label{ass:predictable-critic-design}
For every $(h,j)$, the learner can choose design points
$(Z_{j,h,i})_{i\leq N_{j,h}^Q}$ in
$\mathbb R^n\times\mathcal A$, measurable with respect to the pre-simulation
history and fixed before the fresh stage-$h$ draws, with
$N_{j,h}^Q\geq q_j$, such that
\begin{equation}
  S_{j,h}^Q
  :=\sum_{i=1}^{N_{j,h}^Q}
    \Psi_{h,j}(Z_{j,h,i})\Psi_{h,j}(Z_{j,h,i})^\top
  \succeq c_{Q,j}N_{j,h}^QI_{q_j},
  \label{eq:deterministic-critic-gram}
\end{equation}
where $c_{Q,j}>0$ is known and uniform in $h$ and $\delta$, but may depend on the sieve level.  
\end{assumption}

The positive-definite Gram matrix in
\cref{ass:predictable-critic-design} makes the unpenalized least-squares
objective strictly convex, so its minimizer over the nonempty compact
convex feasible set is unique.  We therefore use constrained least
squares for the theoretical critic update. Conditioning on the entire design before sampling
the fresh labels gives a concentration bound normalized by
$S_{j,h}^Q$, as proved in \cref{app:MC-concentration-proof}.

Requiring the learner to place design points so that the
Gram matrix is well conditioned is a
\emph{synthetic} design principle. The points $Z_{j,h,i}$ are simulated, so \eqref{eq:deterministic-critic-gram} is something
the learner can enforce by construction at no regret cost, rather than a
hypothesis about the environment or the data-generating policy.

Using the covariance bounds fixed in \cref{ass:localized-parametric-model},
set
\[
  \mathcal S_0
  :=\{S=S^\top:\lambda_0I_n\preceq S\preceq\Lambda_0I_n\}.
\]
For each admissible timestep, an \emph{admissible kernel} belongs to the
Gaussian family
\begin{equation}
  P_{\theta,\Sigma}(\cdot\mid x,a)
  :=\mathcal N\!\left(x+\delta\Phi_\mu(x,a)\theta,\delta\Sigma\right),\qquad
  \|\theta\|\leq B_\theta,\qquad \Sigma\in\mathcal S_0.
  \label{eq:admissible-model-family}
\end{equation}
The approximation condition below is imposed over this entire family.
For a feasible continuation $v$, an admissible kernel $P$ with associated
operator $\mathcal P$, and a stage $h$, define
\begin{equation*}
  G_{h}^{P,v}(x,a)
  :=\delta r(x,a)+(\mathcal{P}v)(x,a).
\end{equation*}
Let $\mathcal V_{h+1,j}$ be the set of envelopes generated by the feasible
stage-$(h+1)$ epoch-$j$ critic class in \eqref{eq:epoch-critic-sieve}; at
$h=H-1$, take $\mathcal V_{H,j}=\{u\}$.

For each realization of the prescribed predictable design, define the
noiseless constrained least-squares projection by
\begin{equation}
  \bar w_{j,h}^{P,v}
  :=\argmin_{w\in\mathcal C_{h,j}}
    \sum_{i=1}^{N_{j,h}^Q}
    \left|G_h^{P,v}(Z_{j,h,i})-Q_{h,j,w}(Z_{j,h,i})\right|^2,
  \qquad \bar Q_{j,h}^{P,v}
  :=Q_{h,j,\bar w_{j,h}^{P,v}}.
  \label{eq:noiseless-constrained-projection}
\end{equation}
By \cref{ass:envelope-stable-critic-class,ass:predictable-critic-design}, 
the minimizer is unique, measurable, and fixed before the Monte Carlo labels.

The vector $\boldsymbol\rho$ used below is specified with
the computational budgets in \cref{subsec:estimation-scheme}.

\begin{assumption}[Two-scale Bellman approximation]
\label{ass:projected-two-scale-Bellman-approximation}
There are constants
\(
  b_{\rm entry},\allowbreak c_a^{\rm glob},\allowbreak
  c_x^{\rm loc},\allowbreak {c_a^{\rm loc}\geq0},
\)
independent of $(\delta,j,h,K,\boldsymbol\rho)$,
with the following property.  For every admissible $\delta$, every
epoch--stage pair $(j,h)$, every kernel $P$ in
\eqref{eq:admissible-model-family}, every
$v\in\mathcal V_{h+1,j}$, and every realization of the learner's
prescribed predictable design in \cref{ass:predictable-critic-design},
the projection \eqref{eq:noiseless-constrained-projection} satisfies
\begin{align}
  \sup_{x,a}\|\nabla_a(G_h^{P,v}-\bar Q_{j,h}^{P,v})(x,a)\|
  &\leq\delta\left(
    b_{\rm entry}+\frac{c_a^{\rm glob}}{\sqrt{n_j}}
  \right),
  \label{eq:projected-entry-approximation}\\
  \sup_x\sup_{a\in\mathcal U_h^\star(x)}
    \|\nabla_s(G_h^{P,v}-\bar Q_{j,h}^{P,v})(x,a)\|
  &\leq\frac{\delta c_s^{\rm loc}}{\sqrt{n_j}},
  \qquad s\in\{x,a\}.
  \label{eq:projected-local-approximation}
\end{align}
The entry floor is strictly below the normalized entry margin:
\begin{equation*}
  \diam(\mathcal A)b_{\rm entry}<\frac{\gamma}{2}.
\end{equation*}
\end{assumption}
Here, ``glob'' refers to bounds over the full
action set $\mathcal A$, while ``loc'' refers to bounds over the union
of optimal-action bins $\mathcal U_h^\star(x)$. Both bounds are uniform
over all states $x\in\mathbb R^n$.

This is a derivative form of approximate Bellman completeness, 
uniform over the kernel family and continuation class, as in 
discrete-time analogues like ``for all $f$ in the class''. 
Derivatives are needed because the greedy policy is driven by the action gradient. 
The $\delta$ factor makes the condition scale-free, and the global 
floor $b_{\rm entry}$ needs only to leave a strict entry margin.
For the constructed example in \cref{subsec:controlled-multibranch-model} with two optimal branch points,
\cref{app:example-design-projection} verifies it using a Fourier--Chebyshev sieve, with
constants uniform in $\delta $.

\subsection{Estimation scheme and learning algorithm}
\label{subsec:estimation-scheme}
At the start of epoch $j$, the learner checks empirical coverage.
If it passes, the learner uses all previously collected real
transitions to estimate the dynamics.  It then generates fresh
synthetic Bellman targets under the estimated kernel, fits the critics
backward in time, and deploys the greedy policy for the rest of the epoch.
If the check fails, it instead deploys the exploration policy and skips both fits.
We specify the construction here, and prove finite-sample accuracy
in \cref{subsec:main-criteria-verification}.

\paragraph{Dynamics estimation from real transitions.}
After $n_j$ episodes, there are $n_jH=n_jT/\delta$
observed Euler transitions.  Flatten the episode--stage indices into a
single chronological index: for episode $k\in\{1,\ldots,n_j\}$ and stage
$h\in\{0,\ldots,H-1\}$, let
\[
  i=(k-1)H+h+1,
  \qquad
  X_i:=X_{k,h},
  \quad
  A_i:=A_{k,h},
  \quad
  X_{i,+}:=X_{k,h+1}.
\]
The coverage check uses the empirical model information after $n$ episodes,
\[
  \widehat\Gamma_n
  :=\frac1{nH}\sum_{k=1}^{n}\sum_{h=0}^{H-1}
    \Phi_\mu(X_{k,h},A_{k,h})^\top
    \Phi_\mu(X_{k,h},A_{k,h}).
\]
The observed regression response is the state increment
\begin{equation}
  Y_i
  :=X_{i,+}-X_i
  =\delta\Phi_\mu(X_i,A_i)\theta^\star+\sqrt\delta\,\eta_i,
  \qquad
  \eta_i\mid\mathcal F_{i-1} \sim\mathcal N(0,\Sigma^\star).
  \label{eq:epoch-increment-regression}
\end{equation} 
Both $Y_i$ and $\Phi_\mu(X_i,A_i)$ are known after transition $i$ is observed;
$\theta^\star$ and $\eta_i$ are unknown.

Define the projected
least-squares estimator directly by
\begin{equation}
  \widehat\theta_j
  =
  \Pi_{\{\theta:\|\theta\|\leq B_\theta\}}
  \left(
    \operatorname*{arg\,min}_{\theta\in\mathbb R^p}
    \sum_{i\leq n_jH}
    \|Y_i-\delta\Phi_\mu(X_i,A_i)\theta\|^2
  \right).
  \label{eq:epoch-drift-estimator}
\end{equation}
Here $\argmin$ denotes the minimum-norm least-squares solution, so the estimator is well defined and measurable even for a singular Gram matrix.

Let $\Pi_{\mathcal S_0}^{\Fr}$ denote metric projection in Frobenius norm
onto the covariance set fixed above.  Define
\begin{equation}
  \overline\Sigma_j
  :=\frac{1}{n_jH\delta}
  \sum_{i\leq n_jH}
  (Y_i-\delta\Phi_\mu(X_i,A_i)\widehat\theta_j)
  (Y_i-\delta\Phi_\mu(X_i,A_i)\widehat\theta_j)^\top,
  \qquad
  \widehat\Sigma_j
  :=\Pi_{\mathcal S_0}^{\Fr}(\overline\Sigma_j).
  \label{eq:epoch-covariance-estimator}
\end{equation}
The learned drift and kernel are
\begin{equation}
  \widehat\mu_j(x,a)=\Phi_\mu(x,a)\widehat\theta_j,
  \qquad
  \widehat P_j(\cdot\mid x,a)
  =\mathcal N\left(
     x+\delta\widehat\mu_j(x,a),
     \delta\widehat\Sigma_j
   \right).
  \label{eq:epoch-learned-kernel}
\end{equation}
Write $\widehat{\mathcal{P}}_j$ for the expectation operator associated
with $\widehat P_j$.
The projections in \eqref{eq:epoch-drift-estimator}--
\eqref{eq:epoch-covariance-estimator} ensure that every learned kernel
belongs to \eqref{eq:admissible-model-family} for every data realization.

\paragraph{Synthetic labels and backward critic fitting.}
The two sample counts serve different purposes. At stage $h$, $N_{j,h}^Q$ is the number of design rows $Z_{j,h,i}=(x_i,a_i)$ (repeats allowed), and $M_{j,h}$ is the number of fresh draws from $\widehat{P}_j$ (estimated using $n_jH$ observed transitions) per row, so the fit uses $N_{j,h}^QM_{j,h}$ synthetic transitions and no additional environment interaction.

Starting from $\widehat V_{j,H}=u$, work backward over $h=H-1,\ldots,0$.
At each stage, fix the predictable design from
\cref{ass:predictable-critic-design} before drawing fresh conditionally
independent samples and forming the labels
\begin{align}
  X'_{j,h,i,m}
  &\sim\widehat P_j(\cdot\mid Z_{j,h,i}),
  \qquad i\leq N_{j,h}^Q,\quad m\leq M_{j,h},
  \label{eq:epoch-synthetic-transitions}\\
  Y_{j,h,i}^Q
  &:=\delta r(Z_{j,h,i})
    +\frac1{M_{j,h}}\sum_{m=1}^{M_{j,h}}
      \widehat V_{j,h+1}(X'_{j,h,i,m}).
  \notag
\end{align}
The continuation $\widehat V_{j,h+1}$ is fixed by the preceding backward
steps. Compute the exact constrained least-squares fit
\begin{equation}
  \widehat w_{j,h}
  :=\argmin_{w\in\mathcal C_{h,j}}
    \sum_{i=1}^{N_{j,h}^Q}
      \bigl(Y_{j,h,i}^Q-Q_{h,j,w}(Z_{j,h,i})\bigr)^2,
  \label{eq:epoch-constrained-least-squares}
\end{equation}
set $\widehat Q_{j,h}:=Q_{h,j,\widehat w_{j,h}}$, and define its
full-action envelope and a measurable greedy policy by
\begin{equation}
  \begin{aligned}
    \widehat V_{j,h}(x)&:=\max_{a\in\mathcal A}\widehat Q_{j,h}(x,a),\\
    \pi_h^{(j)}(x)&\in\argmax_{a\in\mathcal A}\widehat Q_{j,h}(x,a),
    \qquad h=0,\ldots,H-1.
  \end{aligned}
  \label{eq:epoch-greedy-policy}
\end{equation}
The algorithm below combines these updates with the coverage check.

\begin{algorithm}[htbp]
\caption{Controlled-diffusion fitted-$Q$ (CD-FQ)}
\label{alg:explore-enter-greedy}
\begin{algorithmic}[1]
  \REQUIRE Episode budget $K$; number of exploration episodes $n_0$; rewards $(r,u)$;
  exploration policy $\pi^{\rm exp}$; coverage threshold $\lambda_\mu$;
  critic-class schedule and design rule; budgets
$(N_{j,h}^Q,M_{j,h})$.
  \STATE Execute $\pi^{\rm exp}$ for $\min\{n_0,K\}$ episodes and retain all
  transitions; stop if $K\leq n_0$.
  \FOR{$j=0,\ldots,J_K$}
    \STATE Set $n_j\gets2^j n_0$ and compute $\widehat\Gamma_{n_j}$.
    \IF{$\widehat\Gamma_{n_j}\not\succeq\lambda_\mu I_p$}
      \STATE Set $\pi^{(j)}\gets\pi^{\rm exp}$.
    \ELSE
      \STATE Select the prescribed epoch-$j$ critic class and fix the design rule.
      \STATE Estimate $\widehat P_j$ from all $n_jH$ observed transitions
      using \eqref{eq:epoch-drift-estimator}--\eqref{eq:epoch-learned-kernel}.
      \STATE Set $\widehat V_{j,H}\gets u$.
      \FOR{$h=H-1,H-2,\ldots,0$}
        \STATE Fix the predictable design as in \cref{ass:predictable-critic-design}.
        \STATE Draw fresh conditionally independent samples and form labels
        via \eqref{eq:epoch-synthetic-transitions}.
        \STATE Solve \eqref{eq:epoch-constrained-least-squares} exactly; form
        $\widehat Q_{j,h}$, $\widehat V_{j,h}$, and $\pi_h^{(j)}$
        as in \eqref{eq:epoch-greedy-policy}.
      \ENDFOR
    \ENDIF
    \STATE Execute $\pi^{(j)}$ for $\min\{n_j,K-n_j\}$ episodes and retain
    the new transitions.
  \ENDFOR
\end{algorithmic}
\end{algorithm}

On the main high-probability event, the coverage check passes after the
exploration episodes. The learner neither knows nor tests the bins;
entry and containment follow from the analysis, and the guarantees hold
for any measurable greedy selection
(\cref{prop:algorithm-output-interface}).

\paragraph{Computational budgets.} 
Fix component failure probabilities
$\rho_{\rm cov},\rho_\theta,\rho_\Sigma,\rho_{\rm MC}\in(0,1)$ whose sum is
less than one, and write
\begin{equation*}
  \boldsymbol\rho
  :=(\rho_{\rm cov},\rho_\theta,\rho_\Sigma,\rho_{\rm MC}).
\end{equation*}
The budgets below limit the Monte Carlo fitting error to
$\mathcal{O}(\delta/\sqrt{n_j})$ in both state and action derivatives.

\begin{assumption}[Monte Carlo and design budgets]
\label{ass:localized-MC-budget}

Fix a universal constant $C_{\rm MC}\geq2$.
There are constants \(c_N, c_{{\rm MC},x}, c_{{\rm MC},a}>0\), independent of $(\delta,j,h,K,\boldsymbol\rho)$,
such that, for every $(j,h)$, the positive integer budgets
$N_{j,h}^Q$ and $M_{j,h}$ satisfy $N_{j,h}^Q\geq\max\{c_Nn_j,q_j\}$.

Define the dimension and confidence factor
\begin{equation*}
  \kappa_j^Q
  :=2q_j\log2
    +4\log\frac{C_{\rm MC}\bar J_KH}{\rho_{\rm MC}}.
\end{equation*}
The simulation budgets satisfy, for $s\in\{x,a\}$,
\begin{equation}
  \frac{L_{\Psi,s,j}G_0}{\delta\sqrt{c_{Q,j}N_{j,h}^Q}}
  \sqrt{\frac{\delta\Lambda_0\kappa_j^Q}{M_{j,h}}}
  \leq\frac{c_{{\rm MC},s}}{\sqrt{n_j}}.
  \label{eq:localized-MC-normalized-budget}
\end{equation}
If $N_{j,h}^Q\asymp n_j$ and
$L_{\Psi,s,j}/\sqrt{c_{Q,j}}$ is uniformly bounded, a sufficient
choice is
\(
  M_{j,h}\asymp \kappa_j^Q/\delta
  =\widetilde{\mathcal{O}} \left(q_j/\delta\right).
\)
\end{assumption}

This condition constrains the learner's simulation sizes
$N_{j,h}^Q$ and $M_{j,h}$.  For any design satisfying
\cref{ass:predictable-critic-design} and the design-size lower bound,
\eqref{eq:localized-MC-normalized-budget} can be met by increasing
$M_{j,h}$.  The factor
$G_0\sqrt{\delta\Lambda_0\kappa_j^Q/M_{j,h}}$ bounds the Monte Carlo
coefficient error in the $S_{j,h}^Q$ norm; the remaining factor
converts it into a $\delta$-normalized derivative error
(\cref{app:MC-concentration-proof}).
The two-level sample, with $N_{j,h}^Q$ design points
and $M_{j,h}$ next-state draws per point, is that of sampling-based fitted
value iteration \citep{munos2008finite}.  Conditional on the entire
design, \citet[Theorem~1]{abbasiyadkori2011improved} applies with initial
matrix $S_{j,h}^Q$ and final normalizing matrix $2S_{j,h}^Q$, yielding
the factor $\kappa_j^Q$ above.  This matrix choice is part of the
concentration argument; the regression objective is unpenalized.
Under the calibration above, an epoch uses
$\sum_hN_{j,h}^QM_{j,h}=\widetilde{\mathcal{O}}(Tn_jq_j/\delta^2)$ synthetic
transitions and no additional episodes.

\subsection{Two-scale residual control and main guarantee}
\label{subsec:two-scale-main-guarantee}

\paragraph{Two-scale residual control.}
The local geometry calls for the following two-scale counterpart of
\cref{def:true-residual-certificate}.  For an epoch-$j$ fitted critic,
write $\widehat Q_{j,h}$ and $\widehat V_{j,h}$ for its critic and
full-action envelope, using the notation of
\cref{def:fitted-envelope-interface} with $k$ replaced by $j$.
The true Bellman residual is
\begin{equation}
  e_{j,h}
  :=
  \delta r+\mathcal{P}^\star\widehat V_{j,h+1}-\widehat Q_{j,h},
  \qquad h<H.
  \label{eq:epoch-true-residual}
\end{equation}
The canonical-extension convention following \eqref{eq:true-residual}
also applies here. 

\begin{criterionvariant}[Two-scale residual control]
\label{def:two-scale-true-residual-certificate}
Suppose the epoch-$j$ critic satisfies
\cref{def:fitted-envelope-interface}.  For nonnegative radii
$d_{a,j}^{\rm glob}$ and $d_{s,j}^{\rm loc}$, $s\in\{x,a\}$, it satisfies
the \emph{two-scale certificate} if every $e_{j,h}$ has a $C^1$ extension
to $\mathbb R^n\times\mathcal A^{\rm o}$ and
\begin{align*}
  \sup_{h<H}\sup_{x,a}\|\nabla_ae_{j,h}(x,a)\|
  &\leq\delta d_{a,j}^{\rm glob},
  \\
  \sup_{h<H}\sup_x\sup_{a\in\mathcal U_h^\star(x)}
  \|\nabla_se_{j,h}(x,a)\|
  &\leq\delta d_{s,j}^{\rm loc},
  \qquad s\in\{x,a\}.
\end{align*}
The global action-derivative bound is used for bin entry;
the local certificate is restricted to the union of bins $\mathcal U_h^\star(x)$.
\end{criterionvariant}

The criterion does not require decay on its radii.  The local rates needed for
the regret guarantee are established in \cref{thm:two-scale-true-residual},
including their confidence dependence.

\paragraph{Main guarantee.}
Define the model confidence factor
\begin{equation}
  \ell_{\rm mdl}
  :=p+\log\frac{\bar J_Kn^2}
                    {\rho_\theta\wedge\rho_\Sigma}.
  \label{eq:localized-model-log-factor}
\end{equation}
The Monte Carlo confidence level $\rho_{\rm MC}$ is absorbed into the
simulation budget in \cref{ass:localized-MC-budget}; after that budget is
imposed, it does not enlarge the statistical regret numerator.

Fix $T$ and the structural parameters, we display the dependence of coefficients
and thresholds on $(K,\boldsymbol\rho)$ only for simplicity. 
The notation $\mathcal{O}(\cdot)$ hides constants independent of
$(K,\boldsymbol\rho,\delta)$.
The theorem involves two deterministic quantities constructed in
\cref{sec:main-theorem-proof}.  The exploration threshold
$n_\star(K,\boldsymbol\rho)$ is the number of exploration episodes after
which empirical coverage holds and every fitted greedy policy enters the
union of optimal bins.  The coefficient
$C_{\mathrm{grad\_err}}(K,\boldsymbol\rho)$ controls the action-gradient
error of each fitted critic at its greedy actions. On the event in the
following theorem, for every fitted epoch $j$ and stage $h<H$,
\[
  \frac1\delta\sup_x
  \left\|\nabla_a(\widehat Q_{j,h}-Q_h^\star)
    (x,\pi_h^{(j)}(x))\right\|
  \leq\frac{C_{\mathrm{grad\_err}}(K,\boldsymbol\rho)}{\sqrt{n_j}}.
\]
For notational simplicity, we use this same coefficient to bound both the local action-gradient error and the decaying part of the global entry radius. The global entry radius also contains the separate, nondecaying term \(b_{\rm entry}\) (\cref{cor:entry-threshold}). The threshold depends on $K$ only through
$\log\log K$ factors.

\begin{theorem}[High-probability
$\widetilde{\mathcal{O}}(\log K)$ localized regret]
\label{thm:end-to-end-localized-regret}
Fix $K\geq2$ and an admissible timestep $\delta$, and hence the experiment with
$H=T/\delta$ stages per episode. 
Suppose
\cref{ass:basic,ass:localized-parametric-model,ass:localized-optimal-branches,ass:exploration-self-excitation,ass:envelope-stable-critic-class,ass:predictable-critic-design,ass:projected-two-scale-Bellman-approximation,ass:localized-MC-budget}
hold.  Let $n_\star(K,\boldsymbol\rho)$ be the exploration threshold in
\eqref{eq:overall-warm-start-threshold} and
$C_{\mathrm{grad\_err}}(K,\boldsymbol\rho)$ the gradient-error coefficient in
\cref{cor:entry-threshold}.  Assume
$K>n_\star(K,\boldsymbol\rho)$, set
$n_0=n_\star(K,\boldsymbol\rho)$, and run
\cref{alg:explore-enter-greedy} with the doubling schedule
\eqref{eq:doubling-schedule}.  There is an event
$\mathcal E_{\rm all}(\delta)$ satisfying
\[
  \Pp\left(\mathcal E_{\rm all}(\delta)^{\,c}\right)
  \leq
  \rho_{\rm cov}+\rho_\theta+\rho_\Sigma+\rho_{\rm MC}
\]
on which
\begin{equation}
  \Reg(K,\delta)
  \leq
  2B_Jn_\star(K,\boldsymbol\rho)
  +\frac{T}{2\bar m}
    \{C_{\mathrm{grad\_err}}(K,\boldsymbol\rho)\}^2(J_K+1).
  \label{eq:end-to-end-regret-bound}
\end{equation}
Consequently, if the physical horizon $T$, the dimensions, structural constants, and component
confidence levels are fixed, then on the same event
\begin{equation*}
  \Reg(K,\delta)
  =\mathcal{O}(\log K\log\log K)
  =\widetilde{\mathcal{O}}(\log K).
\end{equation*}
The regret bound has no polynomial dependence on $1/\delta$.
The order above concerns $K$ at fixed $T$.
Its implicit constant may depend on $T$, including exponentially
through gradient propagation. An explicit bound on the horizon
dependence is given in \cref{rem:physical-horizon-dependence}.
\end{theorem}

\section{Proof of the main theorem}
\label{sec:main-theorem-proof}

Fix $(K,\delta)$ and the assumptions of
\cref{thm:end-to-end-localized-regret}.  The proof closes a loop over
epochs: empirical coverage at the epoch boundary $n_j$ gives
\cref{def:fitted-envelope-interface,def:two-scale-true-residual-certificate}
at epoch $j$; these criteria give bin entry and local gradient control
for $\pi^{(j)}$; and the resulting proximity of $\pi^{(j)}$ to the
optimal-action sets implies coverage
at the next fitted boundary $n_{j+1}$, when $n_{j+1}<K$.

\Cref{subsec:main-criteria-verification} verifies the two criteria on a simultaneous deviation event; \cref{subsec:main-criteria-consequences} derives, deterministically for one epoch, bin entry, local gradient control, and proximity to the optimal-action set; \cref{subsec:main-coverage-induction} closes the loop by forward induction over epochs; and \cref{subsec:main-regret-proof} sums the one-episode losses over epochs. Concentration arguments are in \cref{app:detailed-concentration} and detailed proofs and threshold calculations in \cref{app:main-theorem-details}.
The residual, transfer, and threshold constants are summarized in
\cref{tab:localized-notation}.

\subsection{Verification of criteria on fitted critics}
\label{subsec:main-criteria-verification}

\begin{proposition}[Verification of fitted-envelope regularity]
\label{prop:algorithm-output-interface}
Under \cref{ass:envelope-stable-critic-class,ass:predictable-critic-design},
every fit performed by \cref{alg:explore-enter-greedy} has a unique,
measurable coefficient minimizer in $\mathcal C_{h,j}$.
For every realization, each fitted critic sequence satisfies
\cref{def:fitted-envelope-interface}, with the stronger uniform bounds
\begin{equation}
  \widehat V_{j,h}\in C^{1,1}(\mathbb R^n),
  \qquad
  \|\nabla\widehat V_{j,h}\|_\infty\leq G_0,
  \qquad
  [\nabla\widehat V_{j,h}]_{\rm Lip}\leq M_0,
  \label{eq:verified-fitted-envelope-bounds}
\end{equation}
including $\widehat V_{j,H}=u$.
The greedy policy $\pi^{(j)}$ can be chosen jointly Borel in the history
and state.  For every $h<H$ and $x\in\mathbb R^n$, it satisfies the
envelope identity:
\[
  \nabla\widehat V_{j,h}(x)
  =\nabla_x\widehat Q_{j,h}\bigl(x,\pi_h^{(j)}(x)\bigr).
\]
\end{proposition}

\begin{proof}
First establish regularity for every
feasible critic, apply \cref{lem:fitted-envelope-identity}, and then
pass these properties to the measurable constrained-regression outputs.
See \cref{app:fitted-envelope-verification-proof}.
\end{proof}

For the realized continuation value at epoch $j$, define the exact learned-model
Bellman target
\[
  G_{j,h}:=\delta r+
    \widehat{\mathcal P}_j\widehat V_{j,h+1}.
\]
The true residual in \eqref{eq:epoch-true-residual} splits into model
estimation and critic fitting errors:
\begin{equation}
  e_{j,h}
  =
  (\mathcal{P}^\star-\widehat{\mathcal{P}}_j)\widehat V_{j,h+1}
  +(G_{j,h}-\widehat Q_{j,h}).
  \label{eq:epoch-true-residual-decomposition}
\end{equation}
Under the assumptions of the theorem below, the
model event $\mathcal E_{\rm mdl}
:=\mathcal E_\theta\cap\mathcal E_\Sigma$ controls drift score and
covariance noise at every epoch endpoint; and the Monte Carlo event
$\mathcal E_{\rm MC}$ controls Monte Carlo noise in all fits performed.  They satisfy
\[
  \Pp\bigl((\mathcal E_{\rm mdl})^c\bigr)
  \leq\rho_\theta+\rho_\Sigma,
  \qquad
  \Pp\bigl((\mathcal E_{\rm MC})^c\bigr)
  \leq\rho_{\rm MC}.
\]
Their constructions impose no empirical coverage condition.
On $\mathcal E_{\rm mdl}$, simultaneously at every epoch with
$\widehat\Gamma_{n_j}\succeq\lambda_\mu I_p$ and
$n_jT\geq\delta_{\max}\ell_{\rm mdl}$,
\begin{equation}
  \frac1\delta
  \|\nabla_s(\mathcal P^\star-\widehat{\mathcal P}_j)v\|_\infty
  \leq \mathcal O\!\left(\sqrt{\frac{\ell_{\rm mdl}}{n_jT}}\right),
  \qquad s\in\{x,a\},
  \label{eq:main-model-derivative-bound}
\end{equation}
uniformly over all $v\in C^{1,1}(\mathbb R^n)$ with
$\|\nabla v\|_\infty\leq G_0$ and $[\nabla v]_{\rm Lip}\leq M_0$.
This includes the adaptive fitted continuations in
\eqref{eq:verified-fitted-envelope-bounds}.
The constant hidden in \(\mathcal O(\cdot)\) is structural and uniform in
$(j,h,\delta,K,\boldsymbol\rho)$, with the horizon and structural parameters fixed.
On $\mathcal E_{\rm MC}$, simultaneously for all fitted critics,
\begin{equation}
\begin{aligned}
  \frac1\delta\sup_{x,a}
  \|\nabla_a(G_{j,h}-\widehat Q_{j,h})(x,a)\|
  &\leq b_{\rm entry}+\frac{c_a^{\rm glob}+c_{{\rm MC},a}}{\sqrt{n_j}},\\
  \frac1\delta\sup_x\sup_{a\in\mathcal U_h^\star(x)}
  \|\nabla_s(G_{j,h}-\widehat Q_{j,h})(x,a)\|
  &\leq\frac{c_s^{\rm loc}+c_{{\rm MC},s}}{\sqrt{n_j}},
  \quad s\in\{x,a\}.
\end{aligned}
  \label{eq:main-fitting-derivative-bounds}
\end{equation}
The approximation and Monte Carlo coefficients are those in
\cref{ass:projected-two-scale-Bellman-approximation,ass:localized-MC-budget};
they are independent of $(j,h,\delta,K,\boldsymbol\rho)$.
The explicit choices, constructions, and proofs of these estimates are given in
\cref{app:kernel-derivative-transfer-proof,app:MC-concentration-proof}.
Combining the estimates gives the following verification of
\cref{def:two-scale-true-residual-certificate}.

\begin{theorem}[Verification of \cref{def:two-scale-true-residual-certificate}]
\label{thm:two-scale-true-residual}
Suppose
\cref{ass:localized-parametric-model,ass:localized-optimal-branches,ass:envelope-stable-critic-class,ass:predictable-critic-design,ass:projected-two-scale-Bellman-approximation,ass:localized-MC-budget}
hold.  There is a nonnegative deterministic residual coefficient
$C_{\rm res}(K,\boldsymbol\rho)$ such that
\begin{equation}
  C_{\rm res}(K,\boldsymbol\rho)
  =\mathcal{O}\left(1+\sqrt{\ell_{\rm mdl}/T}\right).
  \label{eq:verified-certificate-numerator-size}
\end{equation}
This coefficient is independent of $(j,h,\delta)$ and depends on
$(K,\boldsymbol\rho)$ only through $\ell_{\rm mdl}$.
The intersection
$\mathcal E_{\rm mdl}\cap\mathcal E_{\rm MC}$ satisfies
\begin{equation}
  \Pp\bigl((\mathcal E_{\rm mdl}
    \cap\mathcal E_{\rm MC})^c\bigr)
  \leq\rho_\theta+\rho_\Sigma+\rho_{\rm MC}.
  \label{eq:certificate-event-probability}
\end{equation}
On this intersection event, the following holds simultaneously at
every epoch:
if $\widehat\Gamma_{n_j}\succeq\lambda_\mu I_p$ and
$n_jT\geq\delta_{\max}\ell_{\rm mdl}$, then the fitted output satisfies
\cref{def:two-scale-true-residual-certificate}, with radii satisfying
\begin{align}
  d_{a,j}^{\rm glob}
  &\leq b_{\rm entry}+
  \frac{C_{\rm res}(K,\boldsymbol\rho)}{\sqrt{n_j}},
  \label{eq:verified-entry-residual-rate}\\
  d_{s,j}^{\rm loc}
  &\leq
  \frac{C_{\rm res}(K,\boldsymbol\rho)}{\sqrt{n_j}},
  \quad s\in\{x,a\}.
  \label{eq:verified-local-residual-rate}
\end{align}
\end{theorem}

We use the common upper bound $C_{\rm res}(K,\boldsymbol\rho)$
for notational simplicity. The three decaying residual coefficients
may differ, and we may simply take their maximum.
\begin{proof}
The uniform envelope bounds in \cref{prop:algorithm-output-interface}
permit application of \eqref{eq:main-model-derivative-bound} to the adaptive
continuation.  Add \eqref{eq:main-fitting-derivative-bounds}
using \eqref{eq:epoch-true-residual-decomposition}.  The fitting numerators
are structural constants, while the model numerators are structural
multiples of $\sqrt{\ell_{\rm mdl}/T}$;
taking the maximum of the three resulting sums gives
\eqref{eq:verified-certificate-numerator-size} and the claimed rates.
The event probability follows by a union bound.  An explicit
choice of $C_{\rm res}(K,\boldsymbol\rho)$ is given in
\cref{app:two-scale-residual-verification-proof}.
\end{proof}

The sample-size condition in \cref{thm:two-scale-true-residual} holds at
every epoch with $n_j\geq n_\Sigma$, where
\begin{equation}
  n_\Sigma:=\left\lceil\frac{\delta_{\max}\ell_{\rm mdl}}{T}\right\rceil\vee1.
  \label{eq:model-sample-threshold}
\end{equation}
The coverage condition $\widehat\Gamma_{n_j}\succeq\lambda_\mu I_p$ is
established at every fitted epoch in \cref{subsec:main-coverage-induction}.

\subsection{Bin entry and local gradient control}
\label{subsec:main-criteria-consequences}
For a single epoch satisfying \cref{def:fitted-envelope-interface,def:two-scale-true-residual-certificate}, we derive bin entry and local gradient control, then a sample threshold beyond which the verified rates meet the required radius conditions. Write 
\begin{equation*}
  \Delta Q_{j,h}=\widehat Q_{j,h}-Q_h^\star,
  \qquad
  \Delta V_{j,h}=\widehat V_{j,h}-V_h^\star,
  \qquad
  \Delta V_{j,H}=0.
\end{equation*}
Take $\kappa=\bar L_{xa}/\bar m$ in
\cref{thm:residual-regret-transfer} and write
$L_E:=L_{\mu,x}+\kappa L_{\mu,a}$. Define the local and global
action-gradient radii after backward transfer by
\begin{equation}
\begin{aligned}
  D_j^{\rm loc}
  &:=d_{a,j}^{\rm loc}
    +L_{\mu,a}\Phi_T(L_E)
       \bigl(d_{x,j}^{\rm loc}+\kappa d_{a,j}^{\rm loc}\bigr),\\
  D_j^{\rm glob}
  &:=d_{a,j}^{\rm glob}
    +L_{\mu,a}\Phi_T(L_E)
       \bigl(d_{x,j}^{\rm loc}+\kappa d_{a,j}^{\rm loc}\bigr).
\end{aligned}
  \label{eq:entry-transfer-radius}
\end{equation}

\begin{theorem}[Backward union entry and local gradient transfer]
\label{thm:backward-union-entry}
Suppose
\cref{ass:basic,ass:localized-optimal-branches,ass:localized-parametric-model}
hold.  Fix an epoch-$j$ fitted critic satisfying
\cref{def:fitted-envelope-interface,def:two-scale-true-residual-certificate}.
If $\diam(\mathcal A)D_j^{\rm glob}\leq\gamma/2$,
then \(\pi_h^{(j)}(x)\in\mathcal U_h^\star(x)\) for every $h<H$ and $x\in\mathbb R^n$. 
Furthermore,
\begin{align}
  \sup_x\|\nabla\Delta V_{j,h}(x)\|
  &\leq
  \Phi_{T-t_h}(L_E)\bigl(d_{x,j}^{\rm loc}+\kappa d_{a,j}^{\rm loc}\bigr),
  \label{eq:localized-value-gradient-transfer}\\
  \sup_x
  \|\nabla_a\Delta Q_{j,h}
      (x,\pi_h^{(j)}(x))\|
  &\leq\delta D_j^{\rm loc},
  \label{eq:localized-action-gradient-transfer}\\
  \sup_x
  \dist(\pi_h^{(j)}(x),\mathcal A_h^\star(x))
  &\leq\frac{D_j^{\rm loc}}{\bar m}.
  \label{eq:localized-policy-set-distance}
\end{align}
\end{theorem}

\begin{proof}[Proof outline]
We follow the backward gradient-transfer argument of
\cref{thm:residual-regret-transfer}, establishing bin entry at each
stage before using the local benchmark inequalities.
For this proof outline, set
\[
  E_{j,h}:=\sup_x\|\nabla\Delta V_{j,h}(x)\|,
  \qquad
  A_{j,h}:=\sup_x
    \|\nabla_a\Delta Q_{j,h}(x,\pi_h^{(j)}(x))\|.
\]
Proceed backward from $E_{j,H}=0$.  Assuming the stated value-gradient
bound at $h+1$, kernel differentiation and the full-action residual bound give
\[
  \sup_{x,a}\|\nabla_a\Delta Q_{j,h}(x,a)\|
  \leq\delta\bigl(d_{a,j}^{\rm glob}+L_{\mu,a}E_{j,h+1}\bigr)
  \leq\delta D_j^{\rm glob}.
\]
For any $a^\star\in\mathcal A_h^\star(x)$ and
$a\in\mathcal A\setminus\mathcal U_h^\star(x)$, convexity of
$\mathcal A$, the entry margin, and the entry condition
$\diam(\mathcal A)D_j^{\rm glob}\leq\gamma/2$ give
\[
  \widehat Q_{j,h}(x,a^\star)-\widehat Q_{j,h}(x,a)
  \geq\delta\bigl(\gamma-\diam(\mathcal A)D_j^{\rm glob}\bigr)
  \geq\frac{\delta\gamma}{2}>0.
\]
Hence every fitted maximizer enters the union of bins, so the
local residual bounds in
\cref{def:two-scale-true-residual-certificate}
apply at the selected greedy actions. Greedy optimality and
strong concavity in a containing bin give the distance bound below;
the mixed-gradient bound and envelope identities give the recursion:
\begin{align*}
  A_{j,h}
  &\leq\delta\bigl(d_{a,j}^{\rm loc}+L_{\mu,a}E_{j,h+1}\bigr)
   \leq\delta D_j^{\rm loc},\\
  \sup_x\dist\bigl(\pi_h^{(j)}(x),\mathcal A_h^\star(x)\bigr)
  &\leq\frac{A_{j,h}}{\delta\bar m}
   \leq\frac{D_j^{\rm loc}}{\bar m},\\
  E_{j,h}
  &\leq\delta\bigl(d_{x,j}^{\rm loc}+\kappa d_{a,j}^{\rm loc}\bigr)
   +(1+\delta L_E)E_{j,h+1}.
\end{align*}
Backward summation closes the induction with the factor
$\Phi_{T-t_h}(L_E)$.
No optimal branch point is fixed in advance.  The full proof is in
\cref{app:backward-union-entry-proof}.
\end{proof}

The coverage argument also needs $D_j^{\rm loc}\leq\bar m \varrho_{\rm cov}$, so that \eqref{eq:localized-policy-set-distance} places the fitted policy in $\Pi_{\rm loc}(\varrho_{\rm cov})$. The next result gives a common sample threshold for this and the entry condition.

\begin{corollary}[Finite entry threshold and post-entry containment]
\label{cor:entry-threshold}
Suppose
\cref{ass:basic,ass:localized-optimal-branches,ass:localized-parametric-model}
hold.  Let the epoch-$j$ fitted critic satisfy
\cref{def:fitted-envelope-interface,def:two-scale-true-residual-certificate},
with radii satisfying
\eqref{eq:verified-entry-residual-rate}--\eqref{eq:verified-local-residual-rate}
for the fixed coefficient $C_{\rm res}(K,\boldsymbol\rho)$ above. Assume
$\diam(\mathcal A)b_{\rm entry}<\gamma/2$.
There is a nonnegative deterministic coefficient
$C_{\mathrm{grad\_err}}(K,\boldsymbol\rho)$, uniform over epochs and
independent of $\delta$, such that
\begin{equation}
  D_j^{\rm loc}
  \leq\frac{C_{\mathrm{grad\_err}}(K,\boldsymbol\rho)}{\sqrt{n_j}},\qquad
  D_j^{\rm glob}
  \leq b_{\rm entry}
      +\frac{C_{\mathrm{grad\_err}}(K,\boldsymbol\rho)}{\sqrt{n_j}}.
  \label{eq:entry-local-transfer-rates}
\end{equation}
For this coefficient, there is a finite deterministic episode threshold,
independent of $\delta$, such that
\[
  \diam(\mathcal A)D_j^{\rm glob}\leq\frac\gamma2,
  \qquad D_j^{\rm loc}\leq\bar m \varrho_{\rm cov}
\]
whenever $n_j$ exceeds that threshold.
For such epochs, all conclusions of \cref{thm:backward-union-entry} hold
and $\pi^{(j)}\in\Pi_{\rm loc}(\varrho_{\rm cov})$.
\end{corollary}

\begin{proof}
Substitute the residual rates into \eqref{eq:entry-transfer-radius}.
Since $\gamma/2-\diam(\mathcal A)b_{\rm entry}>0$,
taking $n_j$ sufficiently large makes
$\diam(\mathcal A)C_{\mathrm{grad\_err}}(K,\boldsymbol\rho)/\sqrt{n_j}$
no larger than this positive difference, and makes
$C_{\mathrm{grad\_err}}(K,\boldsymbol\rho)/\sqrt{n_j}$ no larger than
$\bar m\varrho_{\rm cov}$. These are the entry and proximity conditions,
respectively. Then apply \cref{thm:backward-union-entry}.
The common coefficient and threshold calculation are given in
\cref{app:entry-threshold-calibration}.
\end{proof}

\subsection{Coverage at every epoch}
\label{subsec:main-coverage-induction}

\Cref{thm:two-scale-true-residual} assumes empirical coverage
$\widehat\Gamma_{n_j}\succeq\lambda_\mu I_p$ at the epoch.
We now show that this holds at every fitted epoch.  Coverage at $n_j$
depends on the population information of all policies deployed before
$n_j$; each is either the exploration policy or a fitted policy from an
earlier epoch.

The exploration policy satisfies $\Gamma(\pi^{\rm exp})\succeq2\lambda_{\rm exp}I_p$ under
\cref{ass:exploration-self-excitation}.  Whenever the conclusions of
\cref{cor:entry-threshold} hold at epoch $j$, the fitted policy $\pi^{(j)}$ lies in
$\Pi_{\rm loc}(\varrho_{\rm cov})$, and
\cref{lem:optimal-policy-neighborhood-excitation} gives $\lambda_{\min}\Gamma(\pi^{(j)})
\geq2\lambda_{\rm loc}$.  Since
$\lambda_\mu=\min\{\lambda_{\rm exp},\lambda_{\rm loc}\}$, each such
policy $\pi\in\{\pi^{\rm exp}\}\cup\Pi_{\rm loc}(\varrho_{\rm cov})$ satisfies $\Gamma(\pi)\succeq2\lambda_\mu I_p$.
The following deviation event transfers this population
bound to the empirical matrix once the sample count meets the
deterministic coverage threshold calibrated in
\cref{app:coverage-concentration-proof}.

\begin{lemma}[Coverage deviation event]
\label{lem:main-coverage-event}
Suppose \cref{ass:localized-parametric-model} holds, and let $\pi_k$ be
the policy deployed in episode $k$; its population information
conditional on the pre-episode history is $\Gamma(\pi_k)$.  There is an
event $\mathcal E_{\rm cov}$ with
$\Pp((\mathcal E_{\rm cov})^c)\leq\rho_{\rm cov}$ on which,
simultaneously at every prescribed endpoint $n_j<K$
whose sample count is at least the deterministic episode threshold in
\eqref{eq:warm-start-thresholds},
\begin{equation}
  \frac1{n_j}\sum_{k=1}^{n_j}\Gamma(\pi_k)
  \succeq2\lambda_\mu I_p
  \quad\Longrightarrow\quad
  \widehat\Gamma_{n_j}\succeq\lambda_\mu I_p.
  \label{eq:coverage-event-implication}
\end{equation}
\end{lemma}

\begin{proof}
The construction in \cref{app:coverage-concentration-proof} assumes
neither population nor empirical coverage.  The conditional information
is computed under the original probability law; the event is used only
for the pathwise implication \eqref{eq:coverage-event-implication}.
\end{proof}

Let $n_\star(K,\boldsymbol\rho)$ be the deterministic
exploration threshold calibrated in
\eqref{eq:overall-warm-start-threshold}. It satisfies the model sample
requirement $n_\star(K,\boldsymbol\rho)\geq n_\Sigma$ from
\eqref{eq:model-sample-threshold}, the sample requirement for entry
and proximity in \cref{cor:entry-threshold}, and the coverage sample
requirement in \cref{lem:main-coverage-event}.
For $K>n_\star(K,\boldsymbol\rho)$, choose
$n_0=n_\star(K,\boldsymbol\rho)$, meeting all three thresholds.

The following theorem combines the deviation events with coverage
and post-entry exploration to verify the criteria throughout the run.

\begin{theorem}[End-to-end coverage and entry event]
\label{thm:coverage-entry-induction}
Fix the admissible timestep $\delta$ and its associated experiment.
Suppose
\cref{ass:basic,ass:localized-parametric-model,ass:localized-optimal-branches,ass:exploration-self-excitation,ass:envelope-stable-critic-class,ass:predictable-critic-design,ass:projected-two-scale-Bellman-approximation,ass:localized-MC-budget} hold.  Assume
$K>n_\star(K,\boldsymbol\rho)$ and choose
$n_0=n_\star(K,\boldsymbol\rho)$.  Then there is
an event
$\mathcal E_{\rm all}(\delta)$ with
\begin{equation}
  \Pp\left(\mathcal E_{\rm all}(\delta)^{\,c}\right)
  \leq
  \rho_{\rm cov}+\rho_\theta+\rho_\Sigma+\rho_{\rm MC}
  \label{eq:all-event-probability}
\end{equation}
on which, at every fitted epoch boundary $n_j<K$,
\begin{enumerate}[label=(\roman*),leftmargin=2.2em]
  \item $\widehat\Gamma_{n_j}\succeq\lambda_\mu I_p$, so the coverage check
  in \cref{alg:explore-enter-greedy} passes and no exploration fallback
  is used after the exploration episodes;
  \item the output satisfies
  \cref{def:two-scale-true-residual-certificate} with the rates
  \eqref{eq:verified-entry-residual-rate}--
  \eqref{eq:verified-local-residual-rate};
  \item the transfer rates
  \eqref{eq:entry-local-transfer-rates}, the entry condition
  $\diam(\mathcal A)D_j^{\rm glob}\leq\gamma/2$, and all conclusions of
  \cref{thm:backward-union-entry} hold, with
  $\pi^{(j)}\in\Pi_{\rm loc}(\varrho_{\rm cov})$.
\end{enumerate}
\end{theorem}

\begin{proof}[Proof outline]
Define
\[
  \mathcal E_{\rm all}(\delta)
  :=\mathcal E_{\rm cov}
    \cap\mathcal E_{\rm mdl}
    \cap\mathcal E_{\rm MC}.
\]
\Cref{lem:main-coverage-event} and
\eqref{eq:certificate-event-probability} give
\eqref{eq:all-event-probability} by a union bound, before any coverage
conclusion is imposed.
On this intersection, exploration gives conditional population
information of at least $2\lambda_\mu I_p$. Since
$n_0=n_\star(K,\boldsymbol\rho)$ meets the coverage sample requirement,
\eqref{eq:coverage-event-implication} gives empirical coverage at the
first fitted epoch. The same choice ensures $n_0\geq n_\Sigma$, so
\cref{thm:two-scale-true-residual} supplies the residual rates. It also
meets the sample requirement of \cref{cor:entry-threshold}, which then
gives entry and proximity. The first fitted policy therefore enters
the union of bins and lies in $\Pi_{\rm loc}(\varrho_{\rm cov})$.

Suppose these conclusions hold through epoch $j$.  Every policy used so
far is either the exploration policy or belongs to the self-exploration
neighborhood, so its population information is at least
$2\lambda_\mu I_p$.  By \eqref{eq:coverage-event-implication}, empirical
coverage holds at the next fitted endpoint $n_{j+1}<K$.  Since sample counts are nondecreasing,
both remaining thresholds still hold, and the verification and entry arguments apply
again.  Thus a first failure cannot occur.  The full induction, including
the coverage check and endpoint details, is in
\cref{app:coverage-entry-induction-proof}.
\end{proof}

\subsection{Proof of the main theorem's regret bound}
\label{subsec:main-regret-proof}

The preceding coverage theorem supplies entry and local gradient control
at every fitted epoch.  We now bound one policy's loss, sum the losses
over the doubling epochs, and evaluate the sizes of the constants.

\begin{proof}[Proof of \cref{thm:end-to-end-localized-regret}]
For the choice $n_0=n_\star(K,\boldsymbol\rho)<K$,
\cref{thm:coverage-entry-induction} supplies $\mathcal E_{\rm all}(\delta)$
with the probability in \eqref{eq:all-event-probability}.
Work on this event throughout.  At every fitted epoch, the entry and
gradient bounds of \cref{thm:backward-union-entry} and the rates
\eqref{eq:entry-local-transfer-rates} hold.

\medskip
\underline{\emph{Step 1: one-episode loss.}}
Fix a fitted epoch $j$.  Within an entered bin, strong concavity and
greedy stationarity give, for every $h<H$ and $x$,
\[
  V_h^\star(x)-Q_h^\star(x,\pi_h^{(j)}(x))
  \leq\frac{1}{2\delta\bar m}
    \|\nabla_a\Delta Q_{j,h}(x,\pi_h^{(j)}(x))\|^2
  \leq\frac{\delta}{2\bar m}(D_j^{\rm loc})^2,
\]
where the last step uses \eqref{eq:localized-action-gradient-transfer}.
The binwise comparison is detailed in
\cref{app:localized-one-episode-regret-proof}.
By \cref{lem:optimal-advantage-pdi}, summing under the learner's trajectory
law over $H=T/\delta$ stages and applying
\eqref{eq:entry-local-transfer-rates} gives
\begin{equation}
  J_\delta^\star-J_\delta(\pi^{(j)})
  \leq\frac{T}{2\bar m}(D_j^{\rm loc})^2
  \leq\frac{T\{C_{\mathrm{grad\_err}}(K,\boldsymbol\rho)\}^2}{2\bar m n_j}.
  \label{eq:localized-one-episode-loss-rate}
\end{equation}
The first inequality holds for any fitted policy under the hypotheses
of \cref{thm:backward-union-entry} alone.
The direct bound through $D_j^{\rm loc}$ uses only local
residual errors. The common upper bound also includes the decaying
entry-error coefficient, while the floor $b_{\rm entry}$ affects only
the entry condition.

\medskip
\underline{\emph{Step 2: summing the epoch losses.}}
Each exploration episode costs at most $2B_J$, and each fitted policy
is used for at most $n_j$ episodes.  Thus applying
\eqref{eq:localized-one-episode-loss-rate} gives \eqref{eq:end-to-end-regret-bound}:
\begin{align*}
  \Reg(K,\delta)
  &\leq
  2B_Jn_\star(K,\boldsymbol\rho)
  +\sum_{j=0}^{J_K}n_j
  \frac{T
  \{C_{\mathrm{grad\_err}}(K,\boldsymbol\rho)\}^2}{2\bar m n_j}\\
  &=
  2B_Jn_\star(K,\boldsymbol\rho)
  +\frac{T}{2\bar m}
  \{C_{\mathrm{grad\_err}}(K,\boldsymbol\rho)\}^2(J_K+1).
\end{align*}

\medskip
\underline{\emph{Step 3: sizes of the constants and the final rate.}}
For fixed $T$, the coefficient and threshold calculations in
\cref{app:localized-rate-bookkeeping-proof} give
\begin{align}
  \{C_{\mathrm{grad\_err}}(K,\boldsymbol\rho)\}^2
  &=\mathcal{O}\left(1+\frac{\ell_{\rm mdl}}{T}\right),
  \label{eq:transfer-numerator-size}\\
  n_\star(K,\boldsymbol\rho)
  &=\mathcal{O}\left(
      1+\frac{\ell_{\rm mdl}}{T}
      +\log\frac{2p\bar J_K}{\rho_{\rm cov}}
    \right).
  \label{eq:warm-start-size}
\end{align}
Substitute these bounds into \eqref{eq:end-to-end-regret-bound} and use
$J_K+1\leq\bar J_K=\mathcal{O}(\log(2K))$ to obtain, on $\mathcal E_{\rm all}(\delta)$,
\begin{align}
  \Reg(K,\delta)
  &=\mathcal{O}\left(
      1+\frac{\ell_{\rm mdl}}{T}
      +\log\frac{2p\bar J_K}{\rho_{\rm cov}}
      +(T+\ell_{\rm mdl})(J_K+1)
    \right)
  \notag\\
  &=\mathcal{O}\left(
    \left[
      1+p
      +\log\frac{n^2\log(2K)}
                      {\rho_\theta\wedge\rho_\Sigma}
      +\log\frac{p\log(2K)}{\rho_{\rm cov}}
    \right]\log(2K)
  \right).
  \label{eq:headline-explicit-rate}
\end{align}
This is $\mathcal{O}(\log K\log\log K)$ for fixed dimensions, structural constants, and confidence levels. All bounds are uniform in $\delta$. The dependence on $H$ is absorbed by the Monte Carlo budget.
\end{proof}

\begin{remark}[Dependence on the physical horizon]
\label{rem:physical-horizon-dependence}
The rates above are for fixed $T$.  Suppose the stage and terminal
rewards are uniformly bounded, $T\geq1$, $L:=L_E>0$, and all structural constants and $\delta_{\max}$
are held fixed as $T$ varies, with
$\varrho_{\rm cov}$ chosen explicitly as in \cref{app:optimal-envelope-regularity}.
Then, on $\mathcal E_{\rm all}(\delta)$ and for
$K>n_\star(K,\boldsymbol\rho)$,
\[
  \Reg(K,\delta)
  =\mathcal{O}\!\left(e^{2LT}(T+\ell_{\rm mdl})[T+\log(2K)]
     +T\log\frac{2p\bar J_K}{\rho_{\rm cov}}\right),
\]
with $e^{2LT}$ is replaced by $T^2$ when $L=0$.
This bound explicitly shows how the regret grows with the physical horizon $T$, with implicit constants that are independent of the discretization parameter $\delta$. See
\cref{app:localized-rate-bookkeeping-proof} for the proof.

This exponential factor comes from backward propagation of gradient errors used to certify pointwise bin entry and continued coverage. Note that since we consider a smaller policy class here, it may be possible to improve this exponential dependence, whereas in the general setting of Section \ref{sec:model-global-regret}, we showed that it matches the lower bound. We leave this question for future investigation. %
\end{remark}

\section{Numerical experiments}
\label{sec:localized-discussion}
\providecommand{\ourmethod}{\mbox{CD-FQ}}
In this section, we perform two sets of experiments. The first provides an end-to-end numerical validation of the theory in a model satisfying all assumptions (\cref{subsec:controlled-multibranch-model}); the second benchmarks \ourmethod{} against baseline methods in a queueing model (\cref{subsec:halfin-whitt-numerics}).

\subsection{Constructive verification: a model with two branch points}
\label{subsec:controlled-multibranch-model}

\paragraph{Model and critic family.}
Take $T=1$, $0<\delta\leq1$ with $H=1/\delta\in\mathbb N$,
$\mathcal A=[-\sqrt2,\sqrt2]$, and $X_0=\pi/2$ in every episode.
The dynamics and rewards are
\begin{equation}
  X_{h+1}=X_h+\delta\theta^\star\left(1+\frac{A_h^2}{4}\right)
                 +\sqrt\delta\,\sigma^\star\xi_{h+1},
  \qquad \xi_{h+1}\sim\mathcal N(0,1),
  \label{eq:example-controlled-model}
\end{equation}
\begin{equation}
  r(x,a)=-2(a^2-1)^2,\qquad u(x)=\cos x.
  \label{eq:example-rewards}
\end{equation}
The noises are independent across stages and episodes.  The unknown
parameters satisfy $|\theta^\star|\leq1$ and
$1\leq(\sigma^\star)^2\leq2$.
The drift feature is $\Phi_\mu(x,a)=1+a^2/4$.

The effective action $b=a^2\in[0,2]$ has a unique interior optimal value
$b_h^\star(x)$, giving
\(
  \mathcal A_h^\star(x)
  =\{-\sqrt{b_h^\star(x)},\sqrt{b_h^\star(x)}\}.
\)
The two branch points have identical rewards and transitions, while the
unknown drift affects their common magnitude.  We use critics of the form
$Q_{h,L,w}(x,a)=-2\delta(a^2-1)^2+P_{L,w}(x,a^2)$, where
$P_{L,w}$ is a real Fourier--Chebyshev expansion with state frequencies
and effective-action polynomial degree at most $L$, and
$q_L=(2L+1)(L+1)$ coefficients.

\begin{proposition}[Joint verification for the model with two branch points]
\label{prop:example-benchmark}
For \eqref{eq:example-controlled-model}--\eqref{eq:example-rewards},
the construction in \cref{app:controlled-multibranch-example} satisfies all
assumptions of \cref{thm:end-to-end-localized-regret}.
For fixed component confidence levels, \cref{alg:explore-enter-greedy}
with this construction and $n_0=n_\star(K,\boldsymbol\rho)<K$ exploration episodes
satisfies
\[
  \Reg(K,\delta)=\mathcal{O}(\log K\log\log K)
\]
with probability at least
$1-\rho_{\rm cov}-\rho_\theta-\rho_\Sigma-\rho_{\rm MC}$,
with an implied constant independent of $\delta$.
\end{proposition}
Since $\Phi_\mu(x,a)=1+a^2/4\geq1$, every policy in this example
satisfies $\Gamma(\pi)\geq1$. The experiment therefore illustrates
learning with multiple optimal actions in a setting where all
policies provide information about the unknown drift.

\paragraph{Experimental set-up and results.}
For each $H\in\{4,8,16,32,64\}$, we run $64$ independent seeds of
$K=8192$ episodes with $\theta^\star=0.75$ and $(\sigma^\star)^2=1.5$.
After $n_0=16$ exploration episodes at $a=1$, the learner refits at
$n=16,32,\ldots,4096$ and deploys each fitted policy for $n$ episodes,
with a fixed level $L=6$ ($91$ coefficients).  Unlike the construction in
\cref{prop:example-benchmark}, the implementation uses this practical choice
of $n_0$, a fixed critic space, deterministic labels, and numerical
constraint checks; details are in \cref{app:example-numerics}.

\Cref{fig:example-policy-gap} shows the mean deployed-policy gap
decaying at rate approximately $n^{-1}$ at every resolution; the mean
fitted log--log slope over $n=128,\ldots,4096$ is $-1.009$.  Final
learning regret at $K=8192$ ranges from $0.0136$ to $0.0168$ across $H$,
and all pairwise bootstrap intervals contain zero, so it shows no
detectable change across the sixteenfold grid refinement.  Exploration
regret approaches an extrapolated limit of $0.0170$ at approximately
first order in $\delta$, reflecting convergence of the Euler benchmark.
Cumulative regret grows approximately linearly in $\log K$
(\cref{app:example-numerics}).

\begin{figure}[!htbp]
  \centering
  \begin{minipage}[c]{0.65\linewidth}
    \centering
    \includegraphics[width=\linewidth]{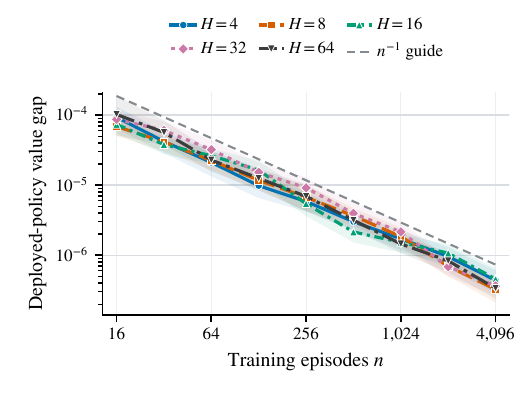}
  \end{minipage}\hfill
  \begin{minipage}[c]{0.32\linewidth}
    \caption{Mean deployed-policy gap against the true-model optimum
    after $n$ training episodes, in the model with two branch points.  Each
    curve averages 64 independent runs at resolution $H=1/\delta$; shading
    gives pointwise 95\% percentile bootstrap intervals from 2,000
    resamples of complete runs.  The gray dashed line is an $n^{-1}$ rate
    reference.}
    \label{fig:example-policy-gap}
  \end{minipage}
\end{figure}

\FloatBarrier
\subsection{Halfin--Whitt queueing example}
\label{subsec:halfin-whitt-numerics}

We next study service-speed control in a many-server queue under
Halfin--Whitt scaling \citep{halfin1981heavy}. To assess the framework's learning efficiency and robustness,   we examine cumulative regret, stability under timestep refinement, performance relative to two value-based baselines, and transferrability of Euler-trained policies to finite queues.

\subsubsection{Experimental setup}\label{subsec:hw-draft-setup}
The state is the diffusion-scaled number of customers relative to the
number of servers, and the action increases service speed.  With
$x^+=\max\{x,0\}$ and $x^-=\max\{-x,0\}$, training uses the Euler model
\begin{equation}
 X_{h+1}=X_h+\delta(-\beta+\mu X_h^- -\kappa A_h)
              +\sqrt{2\mu\delta}\,\varepsilon_h,
 \qquad \varepsilon_h\overset{\mathrm{iid}}{\sim}N(0,1),
 \label{eq:hw-draft-euler}
\end{equation}
over $H=T/\delta$ stages, with $T=5$, $X_0=0$, $\mu=1$,
$\beta=1/2$, and $\kappa=1$.  We minimize
$J_\delta(\pi)=\E_\pi\sum_{h=0}^{H-1}\delta\ell(X_h,A_h)$, where
$\ell(x,a)=qx^++ca^2/2$ and $q=c=1$, with zero terminal cost.
In the reward convention of \cref{sec:model-global-regret}, $r=-\ell$.
Actions lie on $\mathcal A_{0.25}=\{0,0.25,\ldots,2\}$.

Our learner, controlled-diffusion fitted-$Q$ (\ourmethod{}), estimates $(\beta,\kappa)$ and the diffusion variance from
fixed-stage observations, retaining the known drift term $\mu x^-$.
After initial exploration (64 episodes in
Experiments 1(a)--1(b), and 32 in 2(a)--2(b)), it fits constrained
cubic-spline critics using fresh Gaussian next-state samples from the estimated
model and updates the greedy policy at dyadic episode counts.
Numerical parameters are fixed across timesteps
within each study.

For the policy $\pi_k$ deployed in episode $k$, we measure
\begin{equation}
 R_K=\sum_{k=1}^K
 \{J_\delta(\pi_k)-J_\delta(\pi^{\mathrm{ref}}_\delta)\},
 \label{eq:hw-draft-excess}
\end{equation}
where $\pi^{\mathrm{ref}}_\delta$ is a numerical Euler reference on the
same action grid and timestep, computed by dynamic programming using the true model parameters. Thus $R_K$ is the regret \eqref{eq:regret} with $J_\delta^\star$ replaced by $J_\delta(\pi^{\mathrm{ref}}_\delta)$, and $J_\delta(\pi)-J_\delta(\pi^{\mathrm{ref}}_\delta)$ is the policy gap of $\pi$. 
These are evaluated using 32,768 paired simulations episodes, with the learned and reference policies receiving identical Gaussian noise draws within each pair.

Experiments 1(a) and 1(b) each use eight independent training
seeds per timestep, with separate seed sets. Experiments 2(a) and 2(b)
use 15 new matched seeds per timestep and a
512-episode budget; Experiment 2(b) transfers the policies from 2(a).
Reported intervals are marginal 95\% Student-$t$ intervals across the
training seeds. 

Appendix~\ref{app:hw-draft-scaling} gives the queue scaling,
Appendix~\ref{app:hw-draft-implementation} the implementation, and
Appendix~\ref{app:hw-draft-evaluation} the reference solver and
evaluation details.
This model lies outside our theory: the action set is a finite grid,
the cost $qx^+$ is not $C^1$, and the known drift term $\mu x^-$ is
unbounded and not $C^1$.  The experiments are therefore a stress test of
\ourmethod{} beyond the assumptions of \cref{thm:end-to-end-localized-regret}.

\subsubsection{Experiment 1(a): growth of cumulative regret}\label{subsec:hw-draft-1a}
We run $K=16384$ episodes at $\delta\in\{0.2,0.1,0.05\}$.
Figure~\ref{fig:hw-draft-long-horizon} plots $R_K-R_{64}$, excluding
the initial exploration cost.

\paragraph{Regret grows near-logarithmically in $K$.}
The mean curves are
approximately affine in $\log_2K$ at larger budgets.  The last doubling
adds $0.82$, $1.00$, and $0.97$ times the previous increment, whereas
$\sqrt K$ or linear growth would give about $1.41$ or $2$.  The fitted
slopes per doubling, $1.82$, $1.68$, and $1.41$
(Appendix~\ref{app:hw-draft-diagnostics}), do not increase as $\delta$
decreases. Both the near-logarithmic growth in $K$ and the stability
under timestep refinement are consistent with our theoretical
$O(\log K\log\log K)$ regret bound, whose implied constant is independent
of $\delta$.

\begin{figure}[!htbp]
  \centering
  \begin{minipage}[c]{0.68\linewidth}
    \centering
    \includegraphics[width=\linewidth]{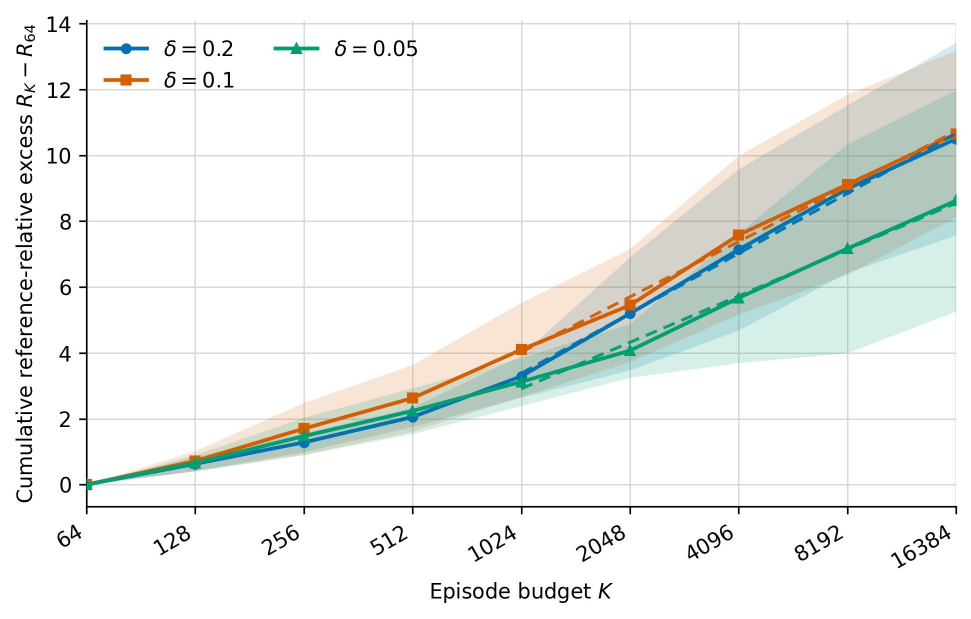}
  \end{minipage}\hfill
  \begin{minipage}[c]{0.3\linewidth}
    \caption{Experiment 1(a): cumulative regret after the 64 exploration
 episodes, at fixed physical horizon $T=5$.  Lines are seed means;
 shading gives 95\% intervals.
 Each timestep uses its own matched Euler reference.  Dashed line segments
 are linear fits in $\log_2 K$ over the last 5 values of $K=1024,\ldots,16384$.}
    \label{fig:hw-draft-long-horizon}
  \end{minipage}
\end{figure}

\subsubsection{Experiment 1(b): stability under timestep refinement}\label{subsec:hw-draft-1b}
We fix $K=2048$ and $T=5$ and vary $\delta$ from $0.2$
to $0.025$.

\paragraph{Stable regret as $\delta$ decreases.}
Table~\ref{tab:hw-draft-timestep}
shows comparable cumulative regret and terminal policy
gaps across different discretizations.
These results support stability over the tested range, while the number
of decisions per episode increases from 25 to 200.  

\paragraph{Little loss after exploration.}
On the scale of the problem, the terminal gaps are
$0.05$--$0.08\%$ of the reference cost. 
The 64 exploration episodes account for about $96\%$ of $R_{2048}$. Over the remaining $1984$ episodes, the average policy gap is $0.0034$--$0.0042$. This illustrates entry into the fast rate regime soon after an initial exploration phase, exploiting the diffusion structure of the underlying mode. 

\begin{table}[!htbp]
 \centering
 \caption{Experiment 1(b): timestep refinement at $T=5$ and $K=2048$.
 The terminal gap evaluates the recommendation fitted after
 episode 2048; cumulative regret includes only deployed policies.}
 \label{tab:hw-draft-timestep}
 \small
 \begin{tabular}{@{}rrcc@{}}
  \toprule
  $\delta$ & $H$ & $R_{2048}-R_{64}$ & Terminal gap \\
  \midrule
  0.200 & 25  & $7.64\ [4.06,11.22]$ & $0.001183\ [0.000507,0.001858]$ \\
  0.100 & 50  & $8.36\ [5.06,11.67]$ & $0.000851\ [0.000645,0.001057]$ \\
  0.050 & 100 & $7.37\ [4.75,9.99]$  & $0.000644\ [0.000380,0.000908]$ \\
  0.025 & 200 & $6.82\ [4.20,9.45]$  & $0.000824\ [0.000510,0.001139]$ \\
  \bottomrule
 \end{tabular}
\end{table}

\subsubsection{Experiment 2(a): common-budget Euler learning}
\label{subsec:hw-draft-2a}
We compare \ourmethod{} with two families of baselines at a common interaction budget.
\ourmethod{} assumes the parametric controlled-diffusion form of
\eqref{eq:hw-draft-euler}, which the Euler training model satisfies
exactly, and is given the known drift term $\mu x^-$; the baselines use
neither. Experiment 2(b) tests whether the advantage
persists when the learned policies are transferred to finite queues, which 
the diffusion only approximates.

The first baseline is backward fitted-$Q$ iteration following
\citet{duan2024taming} (DW). Source-style DW fits only the newest epoch's
episodes at each refit. DW + replay pools all past episodes, including
initial exploration, while re-using the other DW settings. The second
family is a finite-action implementation of little-$q$
\citep{jia2023qlearning}, which learns a value function and action score
from martingale orthogonality conditions. Online little-$q$ updates after
every episode; batch little-$q$ solves the conditions on all observed
data at dyadic refits. The main comparison uses online little-$q$'s native Gibbs policy at temperature
$\nu=0.2$. Batch little-$q$ also uses $\nu=0.2$ and ridge coefficient $\rho=10^{-5}$ with a greedy policy for the terminal comparison 
(Appendix~\ref{app:hw-draft-comparators}).

Every run receives the same type of fixed-stage tuple
$(X_h,A_h,\delta\ell(X_h,A_h),X_{h+1})$ on the same action grid,
$K=512$ episodes, and $n_0=32$ initial exploration episodes.
The 15 seed labels are matched across methods and timesteps, though trajectories can differ when actions differ.
This gives $KT=2560$ units of physical interaction and $512H$ observed
transitions per run at each horizon. Here \ourmethod{} uses exactly zero coefficient ridge. 
Timings and solver diagnostics appear in
Appendix~\ref{app:hw-draft-comparator-diagnostics}.

The primary endpoint is the terminal value gap: the expected cost of the
policy fitted after episode 512 minus that of the matched numerical Euler
reference. It is estimated on the Euler model using paired simulation paths.
Each reported
interval is a marginal 95\% Student-$t$ interval across the 15 independent training seeds.
Statistical estimation error from finite randomized training samples 
and Monte Carlo error from the finite evaluation samples are not estimated separately.

\begin{table}[!htbp]
 \centering
 \caption{Experiment 2(a): percentage terminal value gaps after $K=512$
 episodes. Each method entry is $100$ times the mean paired-simulation
 gap divided by the deterministic Euler reference cost $V_\delta^{\rm grid}$;
 both endpoints of its marginal 95\% interval across 15 training seeds
 are scaled by the same factor. All method entries are percentages;
 comparator entries are rounded to three significant digits. 
 Absolute gaps and reference costs
 appear in Table~\ref{tab:hw-draft-common-euler} in
 Appendix~\ref{app:hw-draft-comparator-diagnostics}.}
 \label{tab:hw-draft-common-euler-percent}
 \footnotesize
 \setlength{\tabcolsep}{5pt}
 \renewcommand{\arraystretch}{1.12}
 \begin{tabular}{@{}rccccc@{}}
  \toprule
  $\delta$ & \shortstack{\ourmethod{}\\(ours)} & \shortstack{DW\\without replay}
  & \shortstack{DW\\+ replay} & \shortstack{Online little-$q$\\native Gibbs}
  & \shortstack{Batch little-$q$\\greedy} \\
  \midrule
  0.200 & $0.21\;[0.11,0.30]$ & $162\;[139,185]$ & $44.2\;[39.5,49.0]$ & $43.3\;[43.1,43.5]$ & $25.9\;[22.5,29.2]$ \\
  0.100 & $0.17\;[0.12,0.23]$ & $196\;[181,212]$ & $62.0\;[54.0,70.0]$ & $46.1\;[45.9,46.3]$ & $24.1\;[21.8,26.3]$ \\
  0.050 & $0.18\;[0.14,0.23]$ & $224\;[216,232]$ & $89.6\;[78.0,101]$ & $47.7\;[47.6,47.9]$ & $24.8\;[22.9,26.6]$ \\
  0.025 & $0.12\;[0.08,0.16]$ & $233\;[226,241]$ & $129\;[118,139]$ & $48.5\;[48.3,48.6]$ & $23.7\;[21.7,25.6]$ \\
  \bottomrule
 \end{tabular}
\end{table}

\paragraph{\ourmethod{} has the smallest terminal value gap.}
Its mean gap is $0.12$--$0.21\%$ of the reference cost across the four
timesteps (Table~\ref{tab:hw-draft-common-euler-percent}). Replay reduces source-style DW's mean gap by
45--73\%, but its Euler gap exceeds online little-$q$'s at every tested
timestep. Batch little-$q$ has the smallest mean gap among the baselines, but it is still
about $24$--$26\%$ of the reference cost and
125--195 times \ourmethod{}'s gap. 

\paragraph{Coverage under the instantiated exploration schedule.}
All methods receive the same 32 uniform-action exploration episodes.
Thereafter \ourmethod{} deploys greedy policies; its smallest saved
information-matrix eigenvalue is $0.1104$, above the coverage threshold
$0.025$. This diagnostic concerns its cumulative design, including the
initial exploration data.

The online construction of
\citet[Section~3.2 and footnote~6]{duan2024taming} also uses an exploratory
burn-in followed by greedy deployment and blockwise FQI refits on the newest
epoch. It is presented as a conceptual consequence of their stability
framework rather than as a fully specified practical algorithm.  
Our source-style DW instantiates
this using the common 32-episode prefix. DW + replay
retains the prefix observations and improves performance without additional
exploration. Accordingly, the comparison
characterize our finite-budget instantiation and its data-reuse choices,
not the source paper's general stability framework or its conditional
guarantee.

\paragraph{The DW gaps increase under finer discretization.}
As $\delta$ decreases from $0.2$ to $0.025$, the source-style DW mean gap
increases from $2.2811$ to $3.1641$, and replay's from $0.6229$ to $1.7441$.
\ourmethod{}'s gap has no upward trend, while batch little-$q$'s greedy gap
remains between $0.3211$ and $0.3642$. Online little-$q$'s gap rises modestly,
from $0.6098$ to $0.6575$, but its policy remains close to initialization
(Appendix~\ref{app:hw-draft-comparator-diagnostics}). These observations
are consistent with the difficulty of regressing the ordinary
$Q$-function's $\mathcal{O}(\delta)$ action signal through $T/\delta$ stages.
They do not isolate that mechanism from differences in model knowledge,
function approximation, regularization, or computation.

\paragraph{\ourmethod{} also has the lowest realized training excess.}
For the complete learning trajectory, define
$\widehat R^{\rm real}_{512}=\sum_{k=1}^{512}C_k-512V_\delta^{\rm grid}$,
where $C_k$ is the observed cost of the policy deployed in episode $k$
and $V_\delta^{\rm grid}$ is the deterministic Euler Bellman-grid value.
This measure includes the common exploration phase. It equals $R_{512}$ in
\eqref{eq:hw-draft-excess} plus mean-zero episode sampling noise and an
offset between the reference policy's expected cost and the numerical
Bellman value, so its mean across runs estimates $\E[R_{512}]$. Appendix~\ref{app:hw-draft-realized-excess} gives
the exact relation and Table~\ref{tab:hw-draft-realized-excess}.
The mean ordering is \ourmethod{}, batch
little-$q$, online little-$q$, DW + replay, then source-style DW at every
timestep.
\ourmethod{}'s mean excess is $91.24$--$101.93$, versus $558.51$--$1254.88$
for replay. All primary comparator-minus-\ourmethod{} paired intervals
are positive. Removing the shared exploration prefix leaves
\ourmethod{} with mean realized excess $10.04$--$19.46$ over 480 episodes.

\FloatBarrier
\subsubsection{Experiment 2(b): transfer to finite queues}
\label{subsec:hw-draft-2b}
We freeze the terminal policies from the 512-episode, 15-seed Euler
comparison and deploy them directly at
$n\in\{100,400,1600\}$ servers. The queue has arrival rate
$n\mu-\beta\sqrt n$, per-busy-server service rate $\mu+\kappa A/\sqrt n$,
initial population $N_0=n$, and capacity $n+\lceil6\sqrt n\rceil$.
Each policy observes $X=(N-n)/\sqrt n$ and holds its action between $\delta$ boundaries. We evaluate physical cost
$J_n(\pi)=\E_\pi\int_0^T\ell(X_t,A_t)\,dt$ by finite-state CTMC 
uniformization, comparing with the numerical queue reference on the same
action grid and decision clock. Seed intervals reflect variability
in the learned policies; this evaluator has no Monte Carlo sampling error. 

\begin{table}[!htbp]
 \centering
 \caption{Experiment 2(b): physical-cost gaps of policies frozen after 512 Euler
 episodes over the matched fixed-stage queue reference. Entries give means and
 marginal 95\% intervals across 15 training seeds; comparator entries are rounded
 to three significant digits. Batch little-$q$ uses greedy extraction. Online
 little-$q$ is omitted because its terminal deployed policy remains effectively
 unchanged from initialization; see
 Appendix~\ref{app:hw-draft-comparator-diagnostics}. Reference costs range from
 $1.337$ to $1.392$.}
 \label{tab:hw-draft-transfer}
 \footnotesize
 \setlength{\tabcolsep}{8pt}
 \renewcommand{\arraystretch}{1.12}
 \begin{tabular}{@{}rrcccc@{}}
  \toprule
$\delta$ & $n$ & \shortstack{\ourmethod{}\\(ours)} & \shortstack{DW\\without replay} & \shortstack{DW\\+ replay} & \shortstack{Batch little-$q$\\greedy} \\
  \midrule
  0.200 & 100 & $0.0049\;[0.0036,0.0063]$ & $2.31\;[1.97,2.64]$ & $0.596\;[0.531,0.662]$ & $0.349\;[0.304,0.394]$ \\
   & 400 & $0.0040\;[0.0027,0.0053]$ & $2.29\;[1.96,2.62]$ & $0.608\;[0.543,0.673]$ & $0.357\;[0.310,0.404]$ \\
   & 1600 & $0.0038\;[0.0024,0.0052]$ & $2.28\;[1.95,2.61]$ & $0.614\;[0.550,0.677]$ & $0.360\;[0.312,0.408]$ \\
  \addlinespace[3pt]
  0.100 & 100 & $0.0039\;[0.0032,0.0045]$ & $2.75\;[2.52,2.97]$ & $0.844\;[0.733,0.956]$ & $0.320\;[0.291,0.349]$ \\
   & 400 & $0.0030\;[0.0024,0.0035]$ & $2.72\;[2.50,2.95]$ & $0.846\;[0.736,0.956]$ & $0.324\;[0.295,0.354]$ \\
   & 1600 & $0.0028\;[0.0021,0.0035]$ & $2.71\;[2.49,2.94]$ & $0.847\;[0.737,0.956]$ & $0.326\;[0.296,0.356]$ \\
  \addlinespace[3pt]
  0.050 & 100 & $0.0038\;[0.0026,0.0050]$ & $3.09\;[2.98,3.20]$ & $1.23\;[1.07,1.39]$ & $0.328\;[0.304,0.353]$ \\
   & 400 & $0.0028\;[0.0020,0.0036]$ & $3.07\;[2.96,3.18]$ & $1.22\;[1.06,1.38]$ & $0.331\;[0.306,0.355]$ \\
   & 1600 & $0.0026\;[0.0019,0.0032]$ & $3.06\;[2.95,3.17]$ & $1.22\;[1.06,1.38]$ & $0.332\;[0.308,0.356]$ \\
  \addlinespace[3pt]
  0.025 & 100 & $0.0031\;[0.0027,0.0035]$ & $3.20\;[3.09,3.30]$ & $1.77\;[1.63,1.91]$ & $0.314\;[0.289,0.338]$ \\
   & 400 & $0.0021\;[0.0017,0.0024]$ & $3.18\;[3.08,3.28]$ & $1.75\;[1.61,1.90]$ & $0.316\;[0.290,0.341]$ \\
   & 1600 & $0.0018\;[0.0014,0.0022]$ & $3.17\;[3.07,3.27]$ & $1.75\;[1.61,1.89]$ & $0.317\;[0.291,0.343]$ \\
  \bottomrule
 \end{tabular}
\end{table}

\paragraph{\ourmethod{} transfers to finite queues best.}
Its mean queue gap is $0.0018$--$0.0049$, the smallest in all 12
configurations (Table~\ref{tab:hw-draft-transfer}).
Batch little-$q$ is the strongest comparator in every cell ($0.314$--$0.360$). Replay improves on source-style DW
throughout.
\ourmethod{}'s mean queue gap decreases with $n$ at each timestep.
At $n=1600$, it is $0.0038$ versus the Euler gap $0.0029$ at
$\delta=0.2$, and $0.0018$ versus $0.0017$ at $\delta=0.025$.
These point estimates are consistent with a residual Euler discretization
effect: the queue approaches the diffusion, whereas training uses its
Euler discretization. They do not establish a convergence rate.
The comparator queue gaps remain within about $5\%$ of their respective
Euler gaps, indicating that transfer changes their reference-relative
costs much less than the separation between learned policies.

\begin{table}[!htbp]
 \centering
 \caption{Experiment 2(b): customer-level outcomes after 512 Euler training episodes. Entries are ranges of point estimates over the 12 $(\delta,n)$ cells, not uncertainty intervals. Waits and delay probabilities are pooled customer ratios; effort is integrated over $[0,T]$; physical cost is from uniformization. Figure~\ref{fig:hw-draft-operational-transfer} gives per-cell operational means and intervals.}
 \label{tab:hw-draft-operational}
 \footnotesize
 \setlength{\tabcolsep}{8pt}
 \renewcommand{\arraystretch}{1.12}
 \begin{tabular}{@{}lcccc@{}}
  \toprule
Method & Scaled mean wait & Delay probability & Effort & Physical cost \\
  \midrule
  \ourmethod{} (ours) & $0.165$--$0.189$ & $0.264$--$0.280$ & $0.511$--$0.538$ & $1.340$--$1.396$ \\
  DW without replay & $0.106$--$0.148$ & $0.162$--$0.209$ & $2.992$--$3.989$ & $3.672$--$4.532$ \\
  DW + replay & $0.145$--$0.284$ & $0.209$--$0.340$ & $0.651$--$2.378$ & $1.978$--$3.106$ \\
  Online little-$q$, native Gibbs & $0.322$--$0.339$ & $0.336$--$0.346$ & $0.409$ & $1.970$--$2.007$ \\
  Batch little-$q$, greedy & $0.199$--$0.258$ & $0.283$--$0.320$ & $0.533$--$0.683$ & $1.651$--$1.752$ \\
  \bottomrule
 \end{tabular}
\end{table}

\paragraph{\ourmethod{} balances congestion and effort.}
A separate customer-level simulation estimates scaled mean wait, delay
probability, and effort (Table~\ref{tab:hw-draft-operational}). Wait and
delay are pooled ratios over admitted arrivals, with 95\% seed-cluster
delta-method intervals ($t_{14}$), marginal or paired within seed;
customers still waiting at $T$ are served at zero effort, so effort cost
accumulates only through $T$. Source-style DW has shorter waits and lower
delay probability than \ourmethod{} but uses $5.8$--$7.4$ times the effort. Online little-$q$
uses 20--24\% less effort but roughly doubles the scaled mean wait.
Replay's tradeoff depends on resolution: at $\delta=0.025$ it waits less
than \ourmethod{} but at much greater effort. Batch little-$q$ has longer
waits and higher delay probability in all 12 cells without saving effort ($0.533$--$0.683$ vs.\
$0.511$--$0.538$). \ourmethod{} has the lowest physical cost throughout.
Note that these results concern frozen-policy transfer.

\FloatBarrier

\section*{Acknowledgments}
J. Blanchet gratefully acknowledges support from the Office of Naval Research (ONR) under award N00014-24-1-2655 and the National Science Foundation (NSF) under grants CCF-2312204 and CCF-2403007. R. Xu gratefully acknowledges support from the NSF CAREER Award DMS-2614933.
\bibliography{bibliography}

\newpage
\appendix

\section{Auxiliary analytic proofs}
\label{app:auxiliary-analytic-proofs}

\subsection{Performance-difference identity}
\label{app:performance-difference-proof}

Throughout the appendix, write
\[
  \mathcal D_h^\star(x,a):=V_h^\star(x)-Q_h^\star(x,a)\geq0,
  \qquad h<H,
\]
for the optimal-advantage loss, the negative of the conventional signed
advantage.  For a deterministic Borel Markov policy $\pi$, let
\[
  \begin{aligned}
    (\mathcal P_h^\pi f)(x)&:=(\mathcal P^\star f)(x,\pi_h(x)),
      \qquad V_H^\pi=u,\\
    V_h^\pi(x)&=\delta r(x,\pi_h(x))
      +(\mathcal P_h^\pi V_{h+1}^\pi)(x),\qquad h<H.
  \end{aligned}
\]
Writing $\mathbb E_{j,x}^\pi$ for expectation under the controlled chain
started from $X_j=x$, the statewise form of \cref{lem:optimal-advantage-pdi} is,
under \cref{ass:basic},
\begin{equation}
  V_j^\star(x)-V_j^\pi(x)
  =\mathbb E_{j,x}^\pi\left[
    \sum_{h=j}^{H-1}\mathcal D_h^\star(X_h,\pi_h(X_h))\right],
  \qquad j\leq H,
  \label{eq:statewise-optimal-advantage-pdi}
\end{equation}
where the sum is zero for $j=H$.

\begin{proof}[Proof of \cref{lem:optimal-advantage-pdi}]
Since $\pi_h$ is Borel, $x\mapsto P^\star(\cdot\mid x,\pi_h(x))$
is a Borel probability kernel. Backward induction gives Borel
functions $V_h^\pi$ and finite constants $C_h$ such that
\[
  |V_h^\star(x)|,\ |V_h^\pi(x)|\leq C_hW(x).
\]
Indeed, one may take $C_H=B_u$ and
$C_h=\delta B_r+(1+c_W\delta)C_{h+1}$.
For every Borel function $f$ satisfying $|f|\leq CW$,
kernel integration is measurable and
\[
  (\mathcal P_h^\pi|f|)(x)
  \leq C(1+c_W\delta)W(x).
\]
Together with
$\mathbb E_\pi W(X_h)\leq e^{c_Wt_h}\int W\,dp_0$,
these bounds ensure absolute integrability of the values,
their differences, and their next-state conditional expectations.
They therefore justify the tower property and the integrations
used below.
Set $D_h^\pi:=V_h^\star-V_h^\pi$, so $D_H^\pi=0$.
Subtracting the Bellman recursions gives
\begin{equation*}
  D_h^\pi(x)=\mathcal D_h^\star(x,\pi_h(x))
             +(\mathcal P_h^\pi D_{h+1}^\pi)(x).
\end{equation*}
Iteration and the Markov property prove
\eqref{eq:statewise-optimal-advantage-pdi}.  Integrating at $j=0$ against
$p_0$ gives \eqref{eq:episode-optimal-advantage-pdi}, since
$J_\delta^\star=\int V_0^\star\,dp_0$ and
$J_\delta(\pi)=\int V_0^\pi\,dp_0$.
\end{proof}

\subsection{Verification under strong concavity}
\label{app:global-geometry-proof}

\begin{proof}[Proof of \cref{prop:strong-concavity-sufficient}]
Assume the hypotheses of \cref{prop:strong-concavity-sufficient} and
fix $h<H$. For the envelope statement, compactness and continuity give
a maximizer $a_h^\star(x)$ at every $x$, unique by strong concavity.
If $x_i\to x$, compactness gives every subsequence of
$a_h^\star(x_i)$ a convergent further subsequence, whose limit is
optimal at $x$ by joint continuity, hence equal to $a_h^\star(x)$.
Thus $a_h^\star$ is continuous. Danskin's theorem gives
\eqref{eq:optimal-envelope-gradient}, whose right-hand side is continuous,
so $V_h^\star\in C^1$.

For the quadratic-advantage bound, fix $x$ and $a,b\in\mathcal A$.
Integrating the strong-concavity gradient inequality along
$a+t(b-a)$, $t\in[0,1]$, gives
\begin{equation}
  Q_h^\star(x,b)\leq Q_h^\star(x,a)
  +\langle\nabla_aQ_h^\star(x,a),b-a\rangle
  -\frac{\delta\bar m}{2}\|b-a\|^2.
  \label{eq:function-strong-concavity}
\end{equation}

Write $a^\star=a_h^\star(x)$; every $v\in N_{\mathcal A}(a)$ satisfies
$\langle v,a^\star-a\rangle\leq0$. Applying
\eqref{eq:function-strong-concavity} with $b=a^\star$ and completing
the square gives
\begin{equation*}
  \mathcal D_h^\star(x,a)
  \leq\langle\nabla_aQ_h^\star(x,a)-v,a^\star-a\rangle
          -\frac{\delta\bar m}{2}\|a-a^\star\|^2
  \leq\frac{\|\nabla_aQ_h^\star(x,a)-v\|^2}{2\delta\bar m}.
\end{equation*}
Taking the infimum over $v$ proves
\eqref{eq:quadratic-advantage-normal-cone}.

For value-gradient stability, we first derive a distance bound.
Optimality at $a^\star$ gives
\begin{equation*}
  \langle\nabla_aQ_h^\star(x,a^\star),a-a^\star\rangle\leq0.
\end{equation*}
Combining this with strong concavity and
$\langle v,a^\star-a\rangle\leq0$ gives
\[
  \delta\bar m\|a-a^\star\|^2
  \leq\langle\nabla_aQ_h^\star(x,a)-v,a^\star-a\rangle
  \leq\|\nabla_aQ_h^\star(x,a)-v\|\,\|a-a^\star\|.
\]
Dividing when $a\ne a^\star$ and taking the infimum over $v$ yields
\begin{equation}
  \|a-a^\star\|\leq\frac{d_h(x,a)}{\delta\bar m},
  \label{eq:distance-from-stationarity}
\end{equation}
which also holds when $a=a^\star$. The envelope identity, the
$\delta\bar L_{xa}$-Lipschitz bound on $\nabla_xQ_h^\star(x,\cdot)$, and
\eqref{eq:distance-from-stationarity} give
\[
  \|\nabla_xQ_h^\star(x,a)-\nabla V_h^\star(x)\|
  \leq\delta\bar L_{xa}\|a-a^\star\|
  \leq\frac{\bar L_{xa}}{\bar m}d_h(x,a).
\]
This proves \eqref{eq:value-gradient-stability} with
$\kappa=\bar L_{xa}/\bar m$ and completes the verification of
\cref{ass:benchmark-error-bounds}.
\end{proof}

\subsection{Greedy selection and the envelope identity}
\label{app:fitted-envelope-proof}

\begin{proof}[Proof of \cref{lem:fitted-envelope-identity}]
For $h<H$, define
\[
  \widehat{\mathcal M}_{k,h}(x)
  :=\argmax_{a\in\mathcal A}\widehat Q_{k,h}(x,a).
\]
Joint continuity and compactness give nonempty compact values; the
measurable maximum theorem on the Polish spaces $\mathbb R^n$ and
$\mathcal A$ supplies a Borel greedy selection.
By Danskin's directional formula and differentiability of the envelope,
\[
  \max_{a\in\widehat{\mathcal M}_{k,h}(x)}
    \langle\nabla_x\widehat Q_{k,h}(x,a),v\rangle
  =D\widehat V_{k,h}(x;v)
  =\langle\nabla\widehat V_{k,h}(x),v\rangle.
\]
For each active $a$, the inner product is at most the right-hand side;
applying this to $v$ and $-v$ gives equality for every
$v\in\mathbb R^n$, proving \eqref{eq:selected-envelope-identity} for
every maximizing action.
\end{proof}

\subsection{Differentiation of the Euler kernel}
\label{app:kernel-differentiation-proof}

\begin{proof}[Proof of \cref{lem:kernel-differentiation}]
For $\xi\sim\mathcal N(0,I_n)$,
\[
  (\mathcal P^\star v)(x,a)
  =\mathbb E v\bigl(x+\delta\mu^\star(x,a)
                    +\sqrt\delta(\Sigma^\star)^{1/2}\xi\bigr).
\]
Bounded $\nabla v$ implies at most linear growth and finiteness of this
expectation.  The fundamental theorem of calculus bounds directional
difference quotients of the integrand, in directions $z_x,z_a$, by
\[
  \|\nabla v\|_\infty(1+\delta L_{\mu,x})\|z_x\|,
  \qquad
  \delta\|\nabla v\|_\infty L_{\mu,a}\|z_a\|,
\]
respectively, for sufficiently small perturbations within the open
domain.  Dominated convergence therefore gives both derivative formulas
in the lemma; applied again to those formulas, with continuity of the
drift Jacobians and $\nabla v$, it proves continuity of the derivatives.
\end{proof}

\subsection{Nonuniform time grids}
\label{app:nonuniform-time-grids}

For a deterministic partition $0=t_0<\cdots<t_H=T$ with
$\Delta_h:=t_{h+1}-t_h\leq\delta$, \cref{thm:residual-regret-transfer}
retains its bounds relative to the grid's Euler optimum if the model,
rewards, and stagewise hypotheses use $\Delta_h$ in place of $\delta$,
with constants uniform over partitions. Backward induction gives
\[
  E_{k,h}
  \leq\Delta_h b_{E,k}+(1+L_E\Delta_h)E_{k,h+1}
  \leq b_{E,k}\Phi_{T-t_h}(L_E),
\]
and the episode loss is at most $Tc_k^2/(2\bar m)$ since
$\sum_h\Delta_h=T$. No minimum step size is needed, but residual derivative
bounds must scale with each $\Delta_h$ rather than the upper bound $\delta$.

\subsection{Sharpness of exponential physical time dependence}
\label{app:physical-time-lower-bound-proof}

\begin{proof}[Proof of \cref{prop:physical-time-lower-bound}]
Fix $\lambda,m,c_x>0$ and $\sigma\geq0$.  For $T>0$ and admissible
$\delta$ with $H=T/\delta$, take $X_0=0$ and
\begin{equation*}
  \mathcal A=[-1,1],\quad \mu^\star(x,a)=\lambda x+a,\quad
  \Sigma^\star=\sigma^2,\quad r(x,a)=-\frac m2a^2,\quad u\equiv0.
\end{equation*}
The zero policy is optimal, with $V_h^\star=0$ and
$Q_h^\star(x,a)=-\delta ma^2/2$.
The hypotheses of \cref{prop:strong-concavity-sufficient} hold with
$(\bar m,\bar L_{xa})=(m,0)$, so
\cref{ass:benchmark-error-bounds} holds with
$(\bar m,\kappa)=(m,0)$. Since
$(L_{\mu,x},L_{\mu,a})=(\lambda,1)$, we have $L_E=\lambda$.
The standing assumptions hold with $W=1$, $B_r=m/2$, and $B_u=c_W=0$.
These primitive constants are independent of $T,\delta$, while
$B_J=mT/2$.

Define
\begin{equation*}
  F_{j,\delta}:=\delta\sum_{i=0}^{j-1}(1+\lambda\delta)^i
  =\frac{(1+\lambda\delta)^j-1}{\lambda},\qquad j\geq0,
\end{equation*}
with $F_{0,\delta}=0$.  For $\eta_k=c_x/\sqrt{k}$, construct backward
from $\widehat V_{k,H}=0$:
\begin{equation*}
  \begin{aligned}
    \widehat Q_{k,h}(x,a)
    &:=-\frac{\delta m}{2}a^2
       +(\mathcal P^\star\widehat V_{k,h+1})(x,a)+\delta\eta_kx,\\
    \widehat V_{k,h}(x)&:=\max_{a\in[-1,1]}\widehat Q_{k,h}(x,a).
  \end{aligned}
\end{equation*}
If $\widehat V_{k,h+1}(x)=b_{k,h+1}x+q_{k,h+1}$ and
$\rho=1+\lambda\delta$, then
\[
  \widehat Q_{k,h}(x,a)
  =(\rho b_{k,h+1}+\delta\eta_k)x+q_{k,h+1}
    +\delta b_{k,h+1}a-\frac{\delta m}{2}a^2.
\]
Hence the envelope is affine, with
$b_{k,h}=\rho b_{k,h+1}+\delta\eta_k=\eta_kF_{H-h,\delta}$ and
$b_{k,H}=0$, and its unique greedy policy is
\begin{equation*}
  \pi_{k+1,h}(x)
  =\min\left\{\frac{\eta_kF_{H-h-1,\delta}}m,1\right\}.
\end{equation*}
Smooth critics and affine envelopes verify
\cref{def:fitted-envelope-interface}.  The true residual
$e_{k,h}(x,a)=-\delta\eta_kx$ verifies
\cref{def:true-residual-certificate} on the full state-action space with
$d_{x,k}=c_x/\sqrt{k}$ and $d_{a,k}=0$.  Moreover,
\begin{align*}
  \|\nabla\Delta V_{k,h}\|_\infty
  &=\frac{c_x}{\sqrt{k}}F_{H-h,\delta},
  \\
  \sup_x|\nabla_a\Delta Q_{k,h}(x,\pi_{k+1,h}(x))|
  &=\frac{\delta c_x}{\sqrt{k}}F_{H-h-1,\delta}.
\end{align*}
Thus \eqref{eq:action-error-bound} and \eqref{eq:E-recursion} hold with
equality, unaffected by the centered Gaussian noise.

The state-independent policy has exact loss
\[
  J_\delta^\star-J_\delta(\pi_{k+1})
  =\frac{\delta m}{2}\sum_{h=0}^{H-1}
    \min\left\{\frac{c_x^2F_{H-h-1,\delta}^2}{m^2k},1\right\}.
\]
Set
\begin{equation*}
  \mathcal B_{T,\delta}:=\frac{\delta}{2m}\sum_{j=0}^{H-1}F_{j,\delta}^2,
  \qquad
  k_0:=\max\left\{1,\left\lceil
    \left(\frac{c_xF_{H-1,\delta}}m\right)^2\right\rceil\right\}.
\end{equation*}
For $k\geq k_0$, every fraction in the minimum is at most one, so
\begin{equation*}
  J_\delta^\star-J_\delta(\pi_{k+1})
  =\frac{c_x^2}{k}\mathcal B_{T,\delta},\qquad k\geq k_0.
\end{equation*}
Choose $\pi_{1,h}=0$.  The earlier losses have a finite nonnegative sum
for fixed $T,\delta,c_x$, so for $K>k_0$,
\begin{equation}
  \Reg(K,\delta)\geq c_x^2\mathcal B_{T,\delta}
       \sum_{k=k_0}^{K-1}\frac1k,
  \qquad
  \lim_{K\to\infty}\frac{\Reg(K,\delta)}{c_x^2\log K}
  =\mathcal B_{T,\delta}.
  \label{eq:physical-time-regret-lower-bound}
\end{equation}
The threshold $k_0$ may grow exponentially with $T$, but each episode's
loss is at most $mT/2$; this coefficient describes the small-residual regime.

At fixed $T$, $F_{j,\delta}-\Phi_{j\delta}(\lambda)\to0$ uniformly for
$0\leq j\leq H$ along admissible refinements.  Riemann summation gives
\begin{align*}
  \mathcal B_{T,\delta}\longrightarrow\mathcal B_T
  &:=\frac1{2m}\int_0^T\Phi_t(\lambda)^2\,dt\nonumber\\
  &=\frac1{2m\lambda^2}\left\{
    \frac{e^{2\lambda T}-1}{2\lambda}
    -\frac{2(e^{\lambda T}-1)}\lambda+T\right\}
    \sim\frac{e^{2\lambda T}}{4m\lambda^3}\quad(T\to\infty).
\end{align*}
Together with \eqref{eq:physical-time-regret-lower-bound}, this proves
\eqref{eq:physical-time-continuum-lower-bound}.
Any bound
\[
  \Reg(K,\delta)\leq A(T,\delta,c_x)+c_x^2M(T)\log K
\]
valid for all fitted-critic sequences under
\cref{thm:residual-regret-transfer} with
$c_a=0$, $A$ independent of $K$, and $M$ uniform in $\delta$, must satisfy
$M(T)\geq\mathcal B_T$: first send $K\to\infty$, then $\delta\downarrow0$.
Thus no subexponential regret coefficient suffices for this class.
The refinement limit also attains $\Phi_T(\lambda)$ for gradient transfer.
These sharpness conclusions concern $\kappa=0$; they establish
neither the full exponent for $\kappa>0$, nor necessity of the extra
factor $T$ in the upper bound, nor a minimax lower bound over learning
algorithms, including the algorithm in \cref{sec:localized-algorithm}.
\end{proof}

\subsection{Reward growth and moment bounds}
\label{app:parametric-growth-moment-proof}

\begin{proof}[Proof of \cref{lem:parametric-growth-moment}]
The reward bounds follow from the fundamental theorem of calculus and
$1,\|x\|\leq W(x)$, while $W(x)\leq1+\|x\|$ ensures
$\int W\,dp_0<\infty$.  Since $\|\nabla W\|\leq1$ and
$0\preceq\nabla^2W\preceq I_n$, Taylor's inequality and the centered
Gaussian increment give
\begin{align*}
  \int W(y)\,P^\star(dy\mid x,a)
  &\leq W(x+\delta\mu^\star(x,a))
       +\frac\delta2\operatorname{tr}\Sigma^\star\\
  &\leq W(x)+\delta c_W\leq(1+c_W\delta)W(x),
\end{align*}
using $\|\mu^\star\|_\infty\leq B_{\Phi_\mu}B_\theta$ and
$\operatorname{tr}\Sigma^\star\leq n\Lambda_0$.
This is \eqref{eq:weight-moment-drift}. For clause (iii) of
\cref{ass:basic}, $r$, $u$, $W$, and
$\mu^\star=\Phi_\mu\theta^\star$ are continuous. For
$\xi\sim\mathcal N(0,I_n)$, define the Euler draw
$Y_{x,a}:=x+\delta\mu^\star(x,a)+\sqrt\delta(\Sigma^\star)^{1/2}\xi$.
Then $W(Y_{x,a})\leq1+\|x\|+\delta B_{\Phi_\mu}B_\theta
+\sqrt{\delta\Lambda_0}\,\|\xi\|$. This bound is integrable and locally
uniform in $(x,a)$, so dominated convergence makes
$(x,a)\mapsto\int W\,dP^\star(\cdot\mid x,a)$ continuous.
\end{proof}

\subsection{Optimal-envelope regularity and self-exploration}
\label{app:optimal-envelope-regularity}

\begin{proof}[Proof of \cref{lem:optimal-envelope-regularity}]
Fix an admissible $\delta$.  We first construct the Bellman recursion
using only \cref{ass:localized-parametric-model}.  Set
$L_{\mu,x}=B_\theta L_{\Phi_\mu,x}$ and
$L_h=\operatorname{Lip}(V_h^\star)$, starting from the globally Lipschitz
terminal reward $V_H^\star=u$, with $L_H\leq\|\nabla u\|_\infty$.
If $V_{h+1}^\star$ is finite and globally Lipschitz, then
$|V_{h+1}^\star(z)|\leq|V_{h+1}^\star(0)|+L_{h+1}\|z\|$.
The finite first moments of the Euler transitions, uniformly over
$a\in\mathcal A$ for fixed $x$, make the Bellman expectation finite.
Coupling transitions with the same Gaussian noise, using continuity of
$r$ and $\mu^\star$, gives joint continuity of $Q_h^\star$ on
$\mathbb R^n\times\mathcal A^{\mathrm o}$. Compactness of $\mathcal A$
then gives a finite maximum and nonempty compact optimal-action sets
with a closed graph. The same coupling yields
\begin{equation}
  L_h\leq\delta\|\nabla_xr\|_\infty
       +(1+\delta L_{\mu,x})L_{h+1}.
  \label{eq:optimal-value-lipschitz-recursion}
\end{equation}
This closes the Lipschitz induction and proves continuity of every
$V_h^\star$.  The construction is pointwise in the starting state and
does not integrate against $p_0$ or use \cref{ass:basic}.

We now prove $C^1$ regularity by backward induction, again starting
from $V_H^\star=u\in C^1$ with bounded gradient.
If $V_{h+1}^\star$ has these properties, the
drift and reward regularity in \cref{ass:localized-parametric-model} and
\cref{lem:kernel-differentiation} give
$Q_h^\star=\delta r+\mathcal P^\star V_{h+1}^\star
\in C^1(\mathbb R^n\times\mathcal A^{\mathrm o})$.

Fix $x$ and $a\in\mathcal A_h^\star(x)$.  Continuity gives
$V_h^\star(y)-Q_h^\star(y,a)\to0$ as $y\to x$, so for nearby $y$ this
gap is less than $\delta\gamma$ and \eqref{eq:union-entry-margin}
forces $a\in\mathcal U_h^\star(y)$.
Choose a bin there containing $a$, with optimal branch point $b_y$.
Strong concavity on the convex bin and optimality at $b_y$ imply
\[
  \frac{\delta\bar m}{2}\|b_y-a\|^2
  \leq Q_h^\star(y,b_y)-Q_h^\star(y,a)
  =V_h^\star(y)-Q_h^\star(y,a),
\]
hence $b_y\to a$, without assuming differentiability of $V_h^\star$.
Optimality at $x$ and $y$ now gives the sandwich
\[
  Q_h^\star(y,a)-Q_h^\star(x,a)
  \leq V_h^\star(y)-V_h^\star(x)
  \leq Q_h^\star(y,b_y)-Q_h^\star(x,b_y).
\]
By the fundamental theorem of calculus, joint $C^1$ regularity, and
$b_y\to a$, both endpoints equal
$\langle\nabla_xQ_h^\star(x,a),y-x\rangle+o(\|y-x\|)$.
Thus $V_h^\star$ is Fr\'echet differentiable; since $a$ was arbitrary,
\eqref{eq:optimal-active-gradient-coherence} follows.

For $x_i\to x$, choose $a_i\in\mathcal A_h^\star(x_i)$.
Every subsequence has a further subsequence with $a_i\to a$; continuity
makes $a$ optimal at $x$.  Along it,
$\nabla V_h^\star(x_i)=\nabla_xQ_h^\star(x_i,a_i)
\to\nabla_xQ_h^\star(x,a)=\nabla V_h^\star(x)$.
This common subsequential limit proves continuity of the gradient.

Since $V_h^\star$ is globally $L_h$-Lipschitz,
$\|\nabla V_h^\star\|_\infty\leq L_h<\infty$, closing the
bounded-gradient induction.  Geometric summation in
\eqref{eq:optimal-value-lipschitz-recursion} yields
\[
  \|\nabla V_h^\star\|_\infty
  \leq e^{L_{\mu,x}(T-t_h)}\|\nabla u\|_\infty
       +\Phi_{T-t_h}(L_{\mu,x})\|\nabla_xr\|_\infty.
\]
For $L\geq0$, $(1+\delta L)^{H-h}\leq e^{L(T-t_h)}$ and
$\delta\sum_{i=0}^{H-h-1}(1+\delta L)^i\leq\Phi_{T-t_h}(L)$;
at $L=0$ the sum equals $T-t_h$.  The bound is uniform in $h,\delta$.
\end{proof}

\begin{proof}[Proof of \cref{lem:optimal-policy-neighborhood-excitation}]
With $L_{\mu,a}=B_\theta L_{\Phi_\mu,a}$, define the
information-sensitivity constant
\[
  L_\Gamma
  :=2B_{\Phi_\mu}L_{\Phi_\mu,a}
    +B_{\Phi_\mu}^2L_{\mu,a}\sqrt{T/\lambda_0}.
\]
Fix $\varrho>0$, admissible $\delta$, and
$\pi\in\Pi_{\rm loc}(\varrho)$. We show that
$\lambda_{\min}\Gamma(\pi)\geq4\lambda_{\rm loc}-L_\Gamma\varrho$.
By \cref{lem:optimal-envelope-regularity}, the optimal-action
correspondence has nonempty compact values and a closed graph.
The measurable minimum theorem therefore gives a Borel nearest-optimal
policy with
$\pi_h^\star(x)\in\mathcal A_h^\star(x)$ and
$\|\pi_h(x)-\pi_h^\star(x)\|\leq \varrho$.
No policy continuity is required.

Let $\mathbb P_\pi,\mathbb P_{\pi^\star}$ be the state-trajectory laws
on $(X_0,\ldots,X_H)$ with common initial law $p_0$.
At the same state, action Lipschitzness bounds the difference of
transition means by $\delta L_{\mu,a}\varrho$, while their common covariance
is at least $\delta\lambda_0I_n$.
The Gaussian relative entropy formula, chain rule, and Pinsker inequality
give
\[
  D_{\rm KL}(\mathbb P_\pi\|\mathbb P_{\pi^\star})
  \leq\frac{T L_{\mu,a}^2\varrho^2}{2\lambda_0},
  \qquad
  \|\mathbb P_\pi-\mathbb P_{\pi^\star}\|_{\rm TV}
  \leq\frac{L_{\mu,a}\varrho}{2}\sqrt{\frac T{\lambda_0}},
\]
where total variation is the supremum over measurable events.
Replacing the action at a fixed state changes
$\Phi_\mu^\top\Phi_\mu$ by at most
$2B_{\Phi_\mu}L_{\Phi_\mu,a}\varrho$ in operator norm; the trajectory average
evaluated at $\pi^\star$ has norm at most $B_{\Phi_\mu}^2$.
Comparing actions first and then trajectory laws gives
\[
  \|\Gamma(\pi)-\Gamma(\pi^\star)\|_{\rm op}
  \leq2B_{\Phi_\mu}L_{\Phi_\mu,a}\varrho
      +2B_{\Phi_\mu}^2\|\mathbb P_\pi-\mathbb P_{\pi^\star}\|_{\rm TV}
  \leq L_\Gamma \varrho.
\]
By \eqref{eq:optimal-policy-self-excitation},
$\lambda_{\min}\Gamma(\pi)\geq4\lambda_{\rm loc}-L_\Gamma \varrho$.
The same bound holds for the infimum over $\Pi_{\rm loc}(\varrho)$.
Any positive radius satisfying
$L_\Gamma\varrho_{\rm cov}\leq2\lambda_{\rm loc}$ therefore yields
\eqref{eq:local-self-excitation}. In particular, choose
\[
  \varrho_{\rm cov}=\frac{2\lambda_{\rm loc}}{L_\Gamma}
  \quad\text{if }L_\Gamma>0;
\]
if $L_\Gamma=0$, any fixed positive radius is valid. This choice uses
only known bounds; all constants are independent of $\delta$ because
the transition bound accumulates through $H\delta=T$. For the
physical-horizon calculation in \cref{rem:physical-horizon-dependence},
keep the radius fixed as $T$ varies when $L_\Gamma=0$.
\end{proof}

\section{Detailed concentration and event proofs}
\label{app:detailed-concentration}

All deviation events below are defined without coverage assumptions;
coverage is used only to invert the realized model Gram matrix.

\subsection{Filtrations and fresh randomness}

Let $\mathcal H_{k-1}$ contain all real and synthetic observations and
learner randomization available before episode $k$.
Conditional on this history, the episode's policy is fixed,
$X_{k,0}\sim p_0$ is fresh, and its transition innovations are independent.
In the chronological increment filtration,
$\Phi_\mu(X_i,A_i)$ is $\mathcal F_{i-1}$-measurable and
$\eta_i\mid\mathcal F_{i-1}\sim\mathcal N(0,\Sigma^\star)$.

Before a synthetic fit, the level, feasible class, learned kernel,
stage design, and continuation $\widehat V_{j,h+1}$ are fixed.
Conditional on this history, the draws have the independent laws in
\eqref{eq:epoch-synthetic-transitions}.  The doubling endpoints are
deterministic functions of $(K,n_0)$, and \eqref{eq:epoch-count-bound}
bounds their number by $\bar J_K$ and the number of fits by $\bar J_KH$.

\subsection{Adaptive matrix coverage}
\label{app:coverage-concentration-proof}
\label{subsec:empirical-model-coverage}

For episode $k$, set
\[
  \mathsf G_k:=\frac1H\sum_{h=0}^{H-1}
  \Phi_\mu(X_{k,h},A_{k,h})^\top
  \Phi_\mu(X_{k,h},A_{k,h}),
  \qquad 0\preceq\mathsf G_k\preceq B_{\Phi_\mu}^2I_p.
\]

\begin{lemma}[Adaptive matrix concentration at epoch boundaries]
\label{lem:adaptive-matrix-coverage}
Let $\mathcal H_{k-1}$ be the history before episode $k$, and fix
$\lambda>0$. Consider at most $L$ prescribed episode-count endpoints
$N$, and set $C_{\rm cov}:=28/3$.  There is an event $\mathcal E_{\rm dev}$ with
$\Pp((\mathcal E_{\rm dev})^c)\leq\rho_{\rm cov}$ such that,
simultaneously at every endpoint satisfying
\begin{equation}
  N
  \geq
  C_{\rm cov}
  \frac{B_{\Phi_\mu}^4}{\lambda^2}
  \log\frac{2pL}{\rho_{\rm cov}},
  \label{eq:coverage-sample-threshold}
\end{equation}
one has
\begin{equation}
  \frac1N\sum_{k=1}^N
  \left(
    \mathsf G_k-
    \E[\mathsf G_k\mid\mathcal H_{k-1}]
  \right)
  \succeq-\lambda I_p.
  \label{eq:coverage-martingale-deviation}
\end{equation}
Consequently, whenever
\begin{equation}
  \frac1N\sum_{k=1}^N
  \E[\mathsf G_k\mid\mathcal H_{k-1}]
  \succeq2\lambda I_p,
  \label{eq:conditional-population-information}
\end{equation}
the empirical information satisfies, on the same event, 
\begin{equation}
  \frac1N\sum_{k=1}^N\mathsf G_k
  \succeq\lambda I_p.
  \label{eq:empirical-information-lower-bound}
\end{equation}
\end{lemma}

\begin{proof}
The matrices
$Z_k=\mathsf G_k-\E[\mathsf G_k\mid\mathcal H_{k-1}]$
are martingale differences satisfying
\[
  \|Z_k\|_{\rm op}\leq2B_{\Phi_\mu}^2,
  \qquad
  \left\|\sum_{k=1}^N
    \E[Z_k^2\mid\mathcal H_{k-1}]\right\|_{\rm op}
  \leq4NB_{\Phi_\mu}^4.
\]
The matrix Freedman inequality \citep[Theorem~1.2]{tropp2011freedman}
therefore gives
\[
  \Pp\!\left(\lambda_{\min}\!\left(\sum_{k=1}^NZ_k\right)\leq-t\right)
  \leq p\exp\!\left\{
    -\frac{t^2}{2(4NB_{\Phi_\mu}^4+2B_{\Phi_\mu}^2t/3)}\right\}.
\]
For $0<\lambda\leq B_{\Phi_\mu}^2$, take $t=N\lambda$.  Since
$2B_{\Phi_\mu}^2t/3\leq2NB_{\Phi_\mu}^4/3$, the bound is at most
$p\exp\{-N\lambda^2/(C_{\rm cov}B_{\Phi_\mu}^4)\}$ with $C_{\rm cov}=28/3$,
which is at most $\rho_{\rm cov}/(2L)$ under
\eqref{eq:coverage-sample-threshold}.
For $\lambda>B_{\Phi_\mu}^2$, the deviation bound holds deterministically.
A union bound over at most $L$ endpoints satisfying
\eqref{eq:coverage-sample-threshold} proves
\eqref{eq:coverage-martingale-deviation}; adding
\eqref{eq:conditional-population-information} gives
\eqref{eq:empirical-information-lower-bound}.
\end{proof}

\paragraph{Coverage event for the prescribed experiment.}
Define the deterministic coverage threshold by
\begin{equation}
  n_{\rm gram}:=\left\lceil C_{\rm cov}\frac{B_{\Phi_\mu}^4}{\lambda_\mu^2}
     \log\frac{2p\bar J_K}{\rho_{\rm cov}}\right\rceil,
  \label{eq:warm-start-thresholds}
\end{equation}
with $C_{\rm cov}=28/3$. The empirical matrix satisfies
$\widehat\Gamma_N=N^{-1}\sum_{k=1}^N\mathsf G_k$.
Define
\begin{equation}
  \mathcal E_{\rm cov}
  :=\bigcap_{\substack{j:n_j<K\\n_j\geq n_{\rm gram}}}
  \left\{
    \widehat\Gamma_{n_j}
    \succeq\frac1{n_j}\sum_{k=1}^{n_j}
      \E[\mathsf G_k\mid\mathcal H_{k-1}]
      -\lambda_\mu I_p
  \right\}.
  \label{eq:main-coverage-deviation-event}
\end{equation}
By the definition of $n_{\rm gram}$, every endpoint in the
intersection meets \eqref{eq:coverage-sample-threshold}, so
\cref{lem:adaptive-matrix-coverage} with $\lambda=\lambda_\mu$ and
$L=\bar J_K$ gives $\mathcal E_{\rm dev}\subseteq\mathcal E_{\rm cov}$
and hence
\[
  \Pp((\mathcal E_{\rm cov})^c)\leq\rho_{\rm cov}.
\]
If $\pi_k$ is the policy deployed in episode $k$, the independent reset
gives
\[
  \E[\mathsf G_k\mid\mathcal H_{k-1}]=\Gamma(\pi_k),
\]
where $\mathcal H_{k-1}$ is the pre-episode history.
Substitution in \eqref{eq:main-coverage-deviation-event} proves
\eqref{eq:coverage-event-implication} and hence
\cref{lem:main-coverage-event}.  Its population premise follows
initially from \eqref{eq:exploration-population-coverage}, and subsequently
from \eqref{eq:local-self-excitation} once entry places deployed policies
in $\Pi_{\rm loc}(\varrho_{\rm cov})$; see \cref{thm:coverage-entry-induction}.

\subsection{Model estimation and kernel derivative transfer}
\label{app:kernel-derivative-transfer-proof}

Define the model event using the Gaussian noise in
\eqref{eq:epoch-increment-regression}.  Put
\[
  Z_j:=\sum_{i\leq n_jH}
       \Phi_\mu(X_i,A_i)^\top\eta_i,
  \qquad
  U_j:=\frac1{n_jH}\sum_{i\leq n_jH}
       (\eta_i\eta_i^\top-\Sigma^\star).
\]
For a sufficiently large universal constant $c_{\rm mdl}$, let 
$ \mathcal E_{\rm mdl} :=\mathcal E_\theta\cap\mathcal E_\Sigma$ where
\begin{align}
  \mathcal E_\theta
  &:=\bigcap_{j:n_j<K}
    \left\{\|Z_j\|\leq c_{\rm mdl}B_{\Phi_\mu}
      \sqrt{\Lambda_0 n_jH\ell_{\rm mdl}}\right\},
  \notag\\
  \mathcal E_\Sigma
  &:=\bigcap_{j:n_j<K}
    \left\{\|U_j\|_{\rm F}\leq c_{\rm mdl}\Lambda_0 n
      \left(\sqrt{\frac{\ell_{\rm mdl}}{n_jH}}
                  +\frac{\ell_{\rm mdl}}{n_jH}\right)\right\}.
  \label{eq:model-noise-event}
\end{align}
The intersections run over the prescribed epoch endpoints of the fixed
$(K,\delta)$ experiment.

\begin{lemma}[Epochwise kernel and derivative transfer]
\label{lem:epoch-model-transfer}
Suppose \cref{ass:localized-parametric-model} holds. For a sufficiently
large universal $c_{\rm mdl}$, the event
\eqref{eq:model-noise-event} satisfies
\begin{equation}
  \Pp(\mathcal E_\theta^c)\leq\rho_\theta,
  \qquad
  \Pp((\mathcal E_\Sigma)^c)\leq\rho_\Sigma,
  \qquad
  \Pp((\mathcal E_{\rm mdl})^c)
  \leq\rho_\theta+\rho_\Sigma.
  \label{eq:model-noise-event-probability}
\end{equation}
There are finite nonnegative structural constants $D_x,D_a$ such that,
on $\mathcal E_{\rm mdl}$, simultaneously at every epoch with
\[
  \widehat\Gamma_{n_j}\succeq\lambda_\mu I_p,
  \qquad
  n_jT\geq\delta_{\max}\ell_{\rm mdl},
\]
and for every $v\in C^{1,1}(\mathbb R^n)$ satisfying
$\|\nabla v\|_\infty\leq G_0$ and
$[\nabla v]_{\rm Lip}\leq M_0$,
\begin{equation}
  \frac{1}{\delta}
  \|\nabla_s(\mathcal{P}^\star-\widehat{\mathcal{P}}_j)v\|_\infty
  \leq
  D_s\sqrt{\frac{\ell_{\rm mdl}}{n_jT}},
  \qquad s\in\{x,a\}.
  \label{eq:epoch-normalized-model-transfer}
\end{equation}
The constants depend only on the model bounds, the fixed timestep ceiling
$\delta_{\max}$, the coverage level $\lambda_\mu$, and the envelope bounds
$G_0,M_0$; they are independent of
$(j,h,\delta,K,\boldsymbol\rho)$ and uniform over
$0<\delta\leq\delta_{\max}$.  The same event works for all such $v$,
including the adaptive fitted continuations.
\end{lemma}

\begin{proof}[Proof of \cref{lem:epoch-model-transfer}]
\emph{Step 1: model estimation.}
Write $N=n_jH=n_jT/\delta$,
$\Phi_i=\Phi_\mu(X_i,A_i)$, and $S_j=\sum_{i\leq N}\Phi_i^\top\Phi_i$.
For unit $v\in\mathbb R^p$ and $\lambda\in\mathbb R$, predictability gives
\[
  \E\!\left[e^{\lambda v^\top\Phi_i^\top\eta_i}
       \mid\mathcal F_{i-1}\right]
  \leq\exp\!\left(\frac{\lambda^2\Lambda_0\|\Phi_iv\|^2}{2}\right).
\]
Thus every unit projection of the score $Z_j$ in
\eqref{eq:model-noise-event} is sub-Gaussian with variance proxy
$NB_{\Phi_\mu}^2\Lambda_0$.  A $1/2$-net of size at most $5^p$
\citep[Corollary~4.2.13]{vershynin2018high} and a union bound over
$\bar J_K$ endpoints yield $\Pp(\mathcal E_\theta^c)\leq\rho_\theta$.
Entrywise Gaussian-product Bernstein bounds
\citep[Lemma~2.7.7 and Corollary~2.8.3]{vershynin2018high}, union-bounded
over $n^2\bar J_K$ entries and endpoints, give
$\Pp(\mathcal E_\Sigma^c)\leq\rho_\Sigma$.
The factor $\ell_{\rm mdl}$ in \eqref{eq:localized-model-log-factor}
bounds both $p\log5$ and
$\log[n^2\bar J_K/(\rho_\theta\wedge\rho_\Sigma)]$ up to universal
multipliers. Increasing $c_{\rm mdl}$ absorbs the numerical factors,
and a union bound proves \eqref{eq:model-noise-event-probability}.

On $\mathcal E_{\rm mdl}$, fix an epoch with
$S_j/N\succeq\lambda_\mu I_p$.
The unprojected parameter error is $\delta^{-1/2}S_j^{-1}Z_j$.
Nonexpansiveness of projection onto the ball containing $\theta^\star$ gives
\begin{equation}
  \|\widehat\theta_j-\theta^\star\|
  \leq\frac{\|Z_j\|}{N\lambda_\mu\sqrt\delta}
  \leq C_\theta\sqrt{\frac{\ell_{\rm mdl}}{n_jT}}
  =:D_{\theta,j},
  \qquad
  C_\theta:=\frac{c_{\rm mdl}B_{\Phi_\mu}\sqrt{\Lambda_0}}{\lambda_\mu}.
  \label{eq:epoch-theta-rate}
\end{equation}

Put $z=\ell_{\rm mdl}/N$,
$a_j=c_{\rm mdl}\Lambda_0 n(\sqrt z+z)$, and
$q_{\Sigma,j}=n\Lambda_0+\sqrt n\,a_j$.
On $\mathcal E_\Sigma$, $\|U_j\|_{\rm F}\leq a_j$ and
$N^{-1}\sum_i\|\eta_i\|^2\leq q_{\Sigma,j}$.
For $d_i=\Phi_i(\theta^\star-\widehat\theta_j)$, expansion gives
\[
  \overline\Sigma_j-\Sigma^\star
  =U_j+\frac{\sqrt\delta}{N}\sum_{i\leq N}
       (\eta_id_i^\top+d_i\eta_i^\top)
      +\frac\delta N\sum_{i\leq N}d_id_i^\top.
\]
Cauchy--Schwarz gives
$\|N^{-1}\sum_i\eta_id_i^\top\|_{\rm F}
 \leq B_{\Phi_\mu}D_{\theta,j}\sqrt{q_{\Sigma,j}}$
without assuming independence of $d_i$ and $\eta_i$.  Hence
\begin{equation*}
  \|\overline\Sigma_j-\Sigma^\star\|_{\rm F}
  \leq a_j+2\sqrt\delta B_{\Phi_\mu}D_{\theta,j}\sqrt{q_{\Sigma,j}}
       +\delta B_{\Phi_\mu}^2D_{\theta,j}^2.
\end{equation*}
Since $D_{\theta,j}\leq C\sqrt{z/\delta}$,
$q_{\Sigma,j}\leq C(1+\sqrt z+z)$, and $z^{3/4}\leq\sqrt z+z$,
the right-hand side is at most $C(\sqrt z+z)$.
Metric projection in Frobenius norm onto $\mathcal S_0\ni\Sigma^\star$
is nonexpansive, so
\begin{equation}
  \|\widehat\Sigma_j-\Sigma^\star\|_{\rm F}
  \leq C_\Sigma\left\{
       \sqrt{\frac{\delta\ell_{\rm mdl}}{n_jT}}
       +\frac{\delta\ell_{\rm mdl}}{n_jT}\right\}.
  \label{eq:epoch-covariance-rate}
\end{equation}
The constants $C_\theta,C_\Sigma$ depend only on model bounds and
$\lambda_\mu$; both estimates hold simultaneously at all covered epochs.

\medskip
\noindent\emph{Step 2: kernel coupling and derivatives.}
Define
\[
  \omega_j:=\delta B_{\Phi_\mu}\|\widehat\theta_j-\theta^\star\|
    +\frac{\sqrt\delta}{2\sqrt{\lambda_0}}
      \|\widehat\Sigma_j-\Sigma^\star\|_{\rm F}.
\]
The common covariance lower bound implies
\[
  \|(\Sigma^\star)^{1/2}-\widehat\Sigma_j^{1/2}\|_{\rm F}
  \leq\frac{\|\Sigma^\star-\widehat\Sigma_j\|_{\rm F}}
                {2\sqrt{\lambda_0}}.
\]
Coupling the transitions with the same $\xi\sim\mathcal N(0,I_n)$ and
applying Cauchy--Schwarz gives, uniformly in $(x,a)$,
\begin{equation*}
\begin{aligned}
  W_1\bigl(P^\star(\cdot\mid x,a),\widehat P_j(\cdot\mid x,a)\bigr)
  &\leq\delta B_{\Phi_\mu}\|\widehat\theta_j-\theta^\star\|\\
  &\quad+\sqrt\delta\,
        \|(\Sigma^\star)^{1/2}-\widehat\Sigma_j^{1/2}\|_{\rm F}
  \leq\omega_j.
\end{aligned}
\end{equation*}

Let $v\in C^{1,1}(\mathbb R^n)$ with
$\|\nabla v\|_\infty\leq G_v$ and $[\nabla v]_{\rm Lip}\leq M_v$.
Since the covariances are constant in $(x,a)$,
\cref{lem:kernel-differentiation} gives
\begin{align*}
  \nabla_x(\mathcal P^\star-\widehat{\mathcal P}_j)v
  &=(I_n+\delta\partial_x\mu^\star)^\top
      (\mathcal P^\star-\widehat{\mathcal P}_j)[\nabla v]
    +\delta\{\partial_x(\mu^\star-\widehat\mu_j)\}^\top
      \widehat{\mathcal P}_j[\nabla v],\\
  \nabla_a(\mathcal P^\star-\widehat{\mathcal P}_j)v
  &=\delta(\partial_a\mu^\star)^\top
      (\mathcal P^\star-\widehat{\mathcal P}_j)[\nabla v]
    +\delta\{\partial_a(\mu^\star-\widehat\mu_j)\}^\top
      \widehat{\mathcal P}_j[\nabla v].
\end{align*}
The coupling bounds the expectation difference by $M_v\omega_j$,
while $\|\widehat{\mathcal P}_j[\nabla v]\|_\infty\leq G_v$ and
$\|\partial_s(\mu^\star-\widehat\mu_j)\|_\infty
 \leq L_{\Phi_\mu,s}\|\widehat\theta_j-\theta^\star\|$.
Using $\|I_n+\delta\partial_x\mu^\star\|_{\rm op}\leq1+\delta L_{\mu,x}$
and $\|\partial_a\mu^\star\|_{\rm op}\leq L_{\mu,a}$ yields
\begin{align}
  \|\nabla_x(\mathcal P^\star-\widehat{\mathcal P}_j)v\|_\infty
  &\leq(1+\delta L_{\mu,x})M_v\omega_j
  +\delta G_vL_{\Phi_\mu,x}\|\widehat\theta_j-\theta^\star\|,
  \label{eq:epoch-model-transfer-x}\\
  \|\nabla_a(\mathcal P^\star-\widehat{\mathcal P}_j)v\|_\infty
  &\leq\delta L_{\mu,a}M_v\omega_j
  +\delta G_vL_{\Phi_\mu,a}\|\widehat\theta_j-\theta^\star\|.
  \label{eq:epoch-model-transfer-a}
\end{align}

\medskip
\noindent\emph{Step 3: normalization.}
Choose
\begin{align*}
  D_\omega&:=B_{\Phi_\mu}C_\theta+\frac{C_\Sigma}{\sqrt{\lambda_0}},\\
  D_x&:=(1+\delta_{\max}L_{\mu,x})M_0D_\omega
       +G_0L_{\Phi_\mu,x}C_\theta,\\
  D_a&:=\delta_{\max}L_{\mu,a}M_0D_\omega
       +G_0L_{\Phi_\mu,a}C_\theta.
\end{align*}
If $n_jT\geq\delta_{\max}\ell_{\rm mdl}$, then
$\delta\ell_{\rm mdl}/(n_jT)\leq1$, so \eqref{eq:epoch-covariance-rate}
gives $\|\widehat\Sigma_j-\Sigma^\star\|_{\rm F}
\leq2C_\Sigma\sqrt{\delta\ell_{\rm mdl}/(n_jT)}$.
Together with \eqref{eq:epoch-theta-rate}, this yields
\begin{equation*}
  \omega_j\leq\delta D_\omega\sqrt{\frac{\ell_{\rm mdl}}{n_jT}}.
\end{equation*}
For $G_v\leq G_0$ and $M_v\leq M_0$, substitution into
\eqref{eq:epoch-model-transfer-x}--\eqref{eq:epoch-model-transfer-a},
division by $\delta$, and $\delta\leq\delta_{\max}$ prove
\eqref{eq:epoch-normalized-model-transfer}.
These pathwise bounds are uniform in $v$, so they also apply to
adaptive fitted continuations.
\end{proof}

\subsection{Conditional Monte Carlo fits}
\label{app:MC-concentration-proof}
\label{subsec:localized-MC-budgets}

\begin{lemma}[Constrained least-squares fitting bounds for \cref{def:two-scale-true-residual-certificate}]
\label{lem:two-scale-fitted-residual}
Suppose
\cref{ass:envelope-stable-critic-class,ass:predictable-critic-design,ass:projected-two-scale-Bellman-approximation,ass:localized-MC-budget}
hold.  There is
an event $\mathcal E_{\rm MC}$, with
$\Pp((\mathcal E_{\rm MC})^c)\leq\rho_{\rm MC}$, on which, simultaneously
for all epoch-stage fits actually performed by the algorithm,
\begin{align}
  \sup_{x,a}
  \|\nabla_a(G_{j,h}-\widehat Q_{j,h})(x,a)\|
  &\leq
  \delta\left(
    b_{\rm entry}+\frac{c_a^{\rm glob}+c_{{\rm MC},a}}{\sqrt{n_j}}
  \right),
  \label{eq:fitted-entry-certificate}\\
  \sup_x\sup_{a\in\mathcal U_h^\star(x)}
  \|\nabla_s(G_{j,h}-\widehat Q_{j,h})(x,a)\|
  &\leq
  \frac{\delta\bigl(c_s^{\rm loc}+c_{{\rm MC},s}\bigr)}{\sqrt{n_j}},
  \qquad s\in\{x,a\}.
  \label{eq:fitted-local-certificate}
\end{align}
The displayed sums combine the finite nonnegative approximation
and Monte Carlo coefficients from
\cref{ass:projected-two-scale-Bellman-approximation,ass:localized-MC-budget}.
After the dimension and confidence factor $\kappa_j^Q$ has been absorbed by the
simulation budget \eqref{eq:localized-MC-normalized-budget}, these sums
are independent of $(j,h,\delta,K,\boldsymbol\rho)$.
\end{lemma}

\begin{proof}[Proof of \cref{lem:two-scale-fitted-residual}]
Fix a fit and condition on its pre-fit history, including the learned
kernel, the continuation, the sample counts, and the entire design.
By \cref{ass:predictable-critic-design}, all of these are fixed before
the fresh stage-$h$ labels are sampled.  Write
$G_{j,h}=\delta r+\widehat{\mathcal P}_j\widehat V_{j,h+1}$.
The parameter and covariance projections give
$\|\widehat\theta_j\|\leq B_\theta$ and
$\widehat\Sigma_j\in\mathcal S_0$ pathwise, so $\widehat P_j$ belongs to
\eqref{eq:admissible-model-family}; since
$\widehat V_{j,h+1}\in\mathcal V_{h+1,j}$, the approximation assumption
applies with $P=\widehat P_j$ and $v=\widehat V_{j,h+1}$.  Let
$\bar w_{j,h},\bar Q_{j,h}$ be the resulting noiseless projection and
critic, which are fixed under the conditioning.

For $\psi_i=\Psi_{h,j}(Z_{j,h,i})$, set
$\zeta_{j,h,i}=Y_{j,h,i}^Q-G_{j,h}(Z_{j,h,i})$ and
$U_{j,h}=\sum_i\psi_i\zeta_{j,h,i}$.
By \eqref{eq:verified-fitted-envelope-bounds}, Gaussian concentration
\citep[Theorem~5.5]{boucheron2013concentration} applies to the
$G_0$-Lipschitz continuation.  Averaging $M_{j,h}$ independent draws gives
conditionally independent, centered, sub-Gaussian noise variables with variance proxy
$\sigma_{j,h}^2\leq G_0^2\delta\Lambda_0/M_{j,h}$.
Under the conditioning, $S_{j,h}^Q=\sum_i\psi_i\psi_i^\top$ is fixed
and positive definite. Apply the self-normalized inequality
\citep[Theorem~1]{abbasiyadkori2011improved} conditionally with initial
matrix $V=S_{j,h}^Q$. The final normalizing matrix is $2S_{j,h}^Q$,
with determinant ratio $2^{q_j}$; rescaling its norm gives, with
conditional failure probability at most $\alpha\in(0,1)$,
\[
  \|U_{j,h}\|_{(S_{j,h}^Q)^{-1}}
  \leq G_0\sqrt{\frac{\delta\Lambda_0}{M_{j,h}}}
    \left[2q_j\log2+4\log\frac1\alpha\right]^{1/2}.
\]
The initial matrix is used only for concentration and adds no penalty
to the regression objective.
With $\alpha=\rho_{\rm MC}/(C_{\rm MC}\bar J_KH)$ and
$\kappa_j^Q$ from \cref{ass:localized-MC-budget}, define
\begin{equation*}
  \varrho_{j,h}^{\rm MC}
  :=G_0\sqrt{\frac{\delta\Lambda_0\kappa_j^Q}{M_{j,h}}}.
\end{equation*}
Since the conditional bound holds for every pre-fit history, iterated
expectation and a union bound over at most $\bar J_KH$ performed fits
give an unconditional event $\mathcal E_{\rm MC}$,
with complement probability at most
$\bar J_KH\alpha=\rho_{\rm MC}/C_{\rm MC}\leq\rho_{\rm MC}$, on which
$\|U_{j,h}\|_{(S_{j,h}^Q)^{-1}}\leq\varrho_{j,h}^{\rm MC}$ for every
performed fit.

Let $d=\widehat w_{j,h}-\bar w_{j,h}$ and
$b_i=G_{j,h}(Z_{j,h,i})-\bar Q_{j,h}(Z_{j,h,i})$.
The variational inequality for \eqref{eq:noiseless-constrained-projection}
gives $\langle d,\sum_i\psi_ib_i\rangle\leq0$.
Testing the variational inequality for the noisy constrained
least-squares fit \eqref{eq:epoch-constrained-least-squares} against $\bar w_{j,h}$
therefore gives
\begin{equation*}
  \|d\|_{S_{j,h}^Q}^2
  \leq\left\langle d,\sum_i\psi_i(\zeta_{j,h,i}+b_i)\right\rangle
    \leq\langle d,U_{j,h}\rangle
\leq\|d\|_{S_{j,h}^Q}\|U_{j,h}\|_{(S_{j,h}^Q)^{-1}}.
\end{equation*}
The approximation residuals $b_i$ need not vanish: their contribution
is nonpositive by the noiseless projection's variational inequality.
Dividing when $d\ne0$ (the case $d=0$ is immediate) yields
\begin{equation*}
  \|\widehat w_{j,h}-\bar w_{j,h}\|_{S_{j,h}^Q}
  \leq\|U_{j,h}\|_{(S_{j,h}^Q)^{-1}}
  \leq \varrho_{j,h}^{\rm MC}.
\end{equation*}
By \eqref{eq:deterministic-critic-gram} and
\eqref{eq:localized-MC-normalized-budget}, for $s\in\{x,a\}$,
\[
\begin{aligned}
  \sup_{x,a}\|\nabla_s(\widehat Q_{j,h}-\bar Q_{j,h})(x,a)\|
  &\leq L_{\Psi,s,j}\|d\|\\
  &\leq\frac{L_{\Psi,s,j}\varrho_{j,h}^{\rm MC}}
             {\sqrt{c_{Q,j}N_{j,h}^Q}}
  \leq\frac{\delta c_{{\rm MC},s}}{\sqrt{n_j}}.
\end{aligned}
\]
Adding the approximation bounds
\eqref{eq:projected-entry-approximation}--\eqref{eq:projected-local-approximation}
proves \eqref{eq:fitted-entry-certificate}--\eqref{eq:fitted-local-certificate}.
\end{proof}

\subsection{Verification of two-scale residual control and event assembly}
\label{app:two-scale-residual-verification-proof}

\begin{proof}[Proof of \cref{thm:two-scale-true-residual}]
Using the model-transfer constants from
\cref{lem:epoch-model-transfer} and the fitting bounds from
\cref{lem:two-scale-fitted-residual}, set
\begin{equation}
  C_{\rm res}(K,\boldsymbol\rho)
  :=\max\Biggl\{c_a^{\rm glob}+c_{{\rm MC},a}
                   +D_a\sqrt{\frac{\ell_{\rm mdl}}T},\max_{s\in\{x,a\}}\left(
      c_s^{\rm loc}+c_{{\rm MC},s}
        +D_s\sqrt{\frac{\ell_{\rm mdl}}T}\right)\Biggr\}.
  \label{eq:verified-certificate-numerators}
\end{equation}
On $\mathcal E_{\rm mdl}\cap\mathcal E_{\rm MC}$,
fix an epoch with $\widehat\Gamma_{n_j}\succeq\lambda_\mu I_p$ and
$n_jT\geq\delta_{\max}\ell_{\rm mdl}$.
By \cref{prop:algorithm-output-interface}, the model bound applies to
every fitted continuation; adding it to each fitting bound
through \eqref{eq:epoch-true-residual-decomposition} and bounding the
resulting coefficient by $C_{\rm res}(K,\boldsymbol\rho)$ gives
\eqref{eq:verified-entry-residual-rate}--\eqref{eq:verified-local-residual-rate},
verifying \cref{def:two-scale-true-residual-certificate}.
Each of the three component expressions is a fixed nonnegative
constant plus a fixed multiple of $\sqrt{\ell_{\rm mdl}/T}$, proving
\eqref{eq:verified-certificate-numerator-size} and independence of
$(j,h,\delta)$.  A union bound gives \eqref{eq:certificate-event-probability}.
\end{proof}

\section{Detailed proofs and calibration for the main theorem}
\label{app:main-theorem-details}

We give the deterministic proofs and calibrations used in
\cref{sec:main-theorem-proof}, using the concentration bounds of
\cref{app:detailed-concentration}.

\subsection{Verification of fitted-envelope regularity}
\label{app:fitted-envelope-verification-proof}

\begin{proof}[Proof of \cref{prop:algorithm-output-interface}]
\emph{Feasible critics.}
Fix $Q_{h,j,w}$ and write
$V_{h,j,w}(x):=\max_{a\in\mathcal A}Q_{h,j,w}(x,a)$.
In this proof, let
$\mathcal M_{h,j,w}(x):=\argmax_{a\in\mathcal A}Q_{h,j,w}(x,a)$
denote its maximizing-action set.
The class assumption gives a $C^1$ extension of $Q_{h,j,w}$, and
compactness of $\mathcal A$ gives nonempty compact active sets
$\mathcal M_{h,j,w}(x)$.
Taking $y=x$ in \eqref{eq:fitted-active-gradient-Lipschitz} gives a
common active state gradient $g_w(x)$.  Hence Danskin's formula yields
\[
  D V_{h,j,w}(x;z)
  =\max_{a\in\mathcal M_{h,j,w}(x)}
     \langle\nabla_xQ_{h,j,w}(x,a),z\rangle
  =\langle g_w(x),z\rangle.
\]
The same class condition makes $g_w$ continuous and bounded by $G_0$,
so $V_{h,j,w}\in C^1$, \cref{def:fitted-envelope-interface} holds, and
\cref{lem:fitted-envelope-identity} gives, for every
$a\in\mathcal M_{h,j,w}(x)$,
\begin{equation}
  \nabla V_{h,j,w}(x)=\nabla_xQ_{h,j,w}(x,a),
  \qquad
  \|\nabla V_{h,j,w}\|_\infty\leq G_0.
  \label{eq:deterministic-fitted-envelope-identity}
\end{equation}
Then \eqref{eq:fitted-active-gradient-Lipschitz} gives
$[\nabla V_{h,j,w}]_{\rm Lip}\leq M_0$.
The measurable maximum theorem applied to the jointly continuous
$Q_{h,j,w}(x,a)$ on the fixed compact action set gives a Borel selector
$(w,x)\mapsto a_{h,j}(w,x)\in\mathcal M_{h,j,w}(x)$.

\emph{Fitted outputs.}
The constrained least-squares objective in
\eqref{eq:epoch-constrained-least-squares} is continuous on the nonempty
compact convex set $\mathcal C_{h,j}$.  By
\cref{ass:predictable-critic-design}, its Hessian is
\[
  2S_{j,h}^Q\succeq2c_{Q,j}N_{j,h}^QI_{q_j},
\]
so its minimizer is unique.  By the measurable argmin theorem, this
minimizer $\widehat w_{j,h}$ is measurable in the predictable design
and labels, so backward induction from $u$ makes every fitted
coefficient measurable in the history and synthetic draws.
Feasibility transfers the envelope bounds to all fitted
outputs; the terminal bounds follow from the class assumption.
Composing $a_{h,j}(w,x)$ with $\widehat w_{j,h}$ gives the jointly Borel
greedy policy and its identity
\eqref{eq:deterministic-fitted-envelope-identity}.
\end{proof}

\subsection{Backward bin entry and local transfer}
\label{app:backward-union-entry-proof}

In the epoch notation of \cref{subsec:main-criteria-consequences},
subtracting the Bellman equations gives
\begin{equation}
  \Delta Q_{j,h}
  =\mathcal{P}^\star\Delta V_{j,h+1}-e_{j,h}.
  \label{eq:epoch-canonical-error-identity}
\end{equation}

\begin{proof}[Proof of \cref{thm:backward-union-entry}]
\Cref{lem:optimal-envelope-regularity,lem:parametric-growth-moment}
give $C^1$ optimal envelopes with bounded gradients and verify
\cref{ass:true-drift}.  Within this proof, set
\[
  E_{j,h}:=\sup_x\|\nabla\Delta V_{j,h}(x)\|,
  \qquad
  A_{j,h}:=\sup_x
  \|\nabla_a\Delta Q_{j,h}(x,\pi_h^{(j)}(x))\|.
\]
For the backward recursion in this proof, abbreviate
\begin{equation*}
  b_{E,j}^{\rm loc}
  :=d_{x,j}^{\rm loc}+\kappa d_{a,j}^{\rm loc},
  \qquad \kappa=\frac{\bar L_{xa}}{\bar m}.
\end{equation*}
Proceed backward from $E_{j,H}=0$, assuming entry at later stages and
\eqref{eq:localized-value-gradient-transfer} at $h+1$.
The bound $E_{j,h+1}<\infty$ gives at most linear growth of
$\Delta V_{j,h+1}$ and hence integrability under every Euler transition.
The optimal continuation also has at most linear growth by
\cref{lem:optimal-envelope-regularity}, so the fitted continuation is
integrable.  Thus \eqref{eq:epoch-canonical-error-identity} holds on
$\mathbb R^n\times\mathcal A^{\mathrm o}$, and
\cref{lem:kernel-differentiation} applies to the error.

\emph{Current-stage entry.}
By \cref{lem:kernel-differentiation}, differentiating
\eqref{eq:epoch-canonical-error-identity} gives
\begin{equation}
  \nabla_a\Delta Q_{j,h}(x,a)
  =\delta\{\partial_a\mu^\star(x,a)\}^\top
  \mathcal{P}^\star[\nabla\Delta V_{j,h+1}](x,a)
  -\nabla_ae_{j,h}(x,a).
  \label{eq:entry-action-identity}
\end{equation}
The global action-residual bound and the induction hypothesis imply
\begin{equation*}
  \sup_{x,a}\|\nabla_a\Delta Q_{j,h}(x,a)\|
  \leq\delta D_j^{\rm glob}.
\end{equation*}
For $a^\star\in\mathcal A_h^\star(x)$, convexity of $\mathcal A$ yields
\begin{equation*}
  |\Delta Q_{j,h}(x,a)-\Delta Q_{j,h}(x,a^\star)|
  \leq\delta\diam(\mathcal A)D_j^{\rm glob}.
\end{equation*}
Thus \eqref{eq:union-entry-margin} and the entry condition
$\diam(\mathcal A)D_j^{\rm glob}\leq\gamma/2$ give
$\widehat Q_{j,h}(x,a^\star)-\widehat Q_{j,h}(x,a)
\geq\delta\gamma/2>0$ whenever $a\notin\mathcal U_h^\star(x)$.
Every full-action fitted maximizer lies in the union, so the local
residual bound now applies in \eqref{eq:entry-action-identity}:
\begin{equation}
  A_{j,h}
  \leq\delta d_{a,j}^{\rm loc}+\delta L_{\mu,a}E_{j,h+1}.
  \label{eq:localized-A-recursion}
\end{equation}

\emph{Local geometric comparison.}
Fix $a\in\mathcal U_h^\star(x)$ and the optimal branch point
$b=a_{h,\ell}^\star(x)$ in a bin containing $a$.
Strong concavity, $\langle\nabla_aQ_h^\star(x,b),a-b\rangle\leq0$,
and $\langle v,b-a\rangle\leq0$ for $v\in N_{\mathcal A}(a)$ give
\[
  \delta\bar m\|a-b\|^2
  \leq\langle\nabla_aQ_h^\star(x,a)-v,b-a\rangle
  \leq\|\nabla_aQ_h^\star(x,a)-v\|\,\|a-b\|.
\]
Divide by $\|a-b\|$ when $a\ne b$ and take the infimum over $v$.
Using \eqref{eq:benchmark-stationarity-residual} and
$b\in\mathcal A_h^\star(x)$ gives, also when $a=b$,
\begin{equation}
  \dist(a,\mathcal A_h^\star(x))
  \leq\|a-b\|
  \leq\frac{d_h(x,a)}{\delta\bar m}.
  \label{eq:distance-to-optimal-set}
\end{equation}
The optimal-envelope identity at $b$ and
\eqref{eq:branch-mixed-gradient} on the same bin give
\begin{equation*}
  \|\nabla_xQ_h^\star(x,a)-\nabla V_h^\star(x)\|
  \leq\delta\bar L_{xa}\|a-b\|
  \leq\frac{\bar L_{xa}}{\bar m}d_h(x,a).
\end{equation*}

\emph{Close the backward recursion.}
By \cref{lem:fitted-envelope-identity,lem:optimal-envelope-regularity},
\begin{equation*}
  \nabla\Delta V_{j,h}(x)
  =\nabla_x\Delta Q_{j,h}(x,\pi_h^{(j)}(x))
+
  \nabla_xQ_h^\star(x,\pi_h^{(j)}(x))
  -\nabla V_h^\star(x).
\end{equation*}
Full-action greedy optimality gives
$\nabla_a\widehat Q_{j,h}(x,\pi_h^{(j)}(x))
\in N_{\mathcal A}(\pi_h^{(j)}(x))$.
Using this normal vector in \eqref{eq:benchmark-stationarity-residual},
\begin{equation}
  d_h(x,\pi_h^{(j)}(x))
  \leq\|\nabla_a\Delta Q_{j,h}(x,\pi_h^{(j)}(x))\|.
  \label{eq:localized-greedy-stationarity}
\end{equation}
Differentiate \eqref{eq:epoch-canonical-error-identity} in $x$ and use
the local state-residual bound, the preceding envelope comparisons,
and \eqref{eq:localized-A-recursion} to obtain
\begin{equation*}
  E_{j,h}
  \leq\delta d_{x,j}^{\rm loc}
       +(1+\delta L_{\mu,x})E_{j,h+1}
       +\frac{\bar L_{xa}}{\bar m}A_{j,h}
  \leq\delta b_{E,j}^{\rm loc}
       +(1+\delta L_E)E_{j,h+1}.
\end{equation*}
Backward summation proves \eqref{eq:localized-value-gradient-transfer}.
Substitute into \eqref{eq:localized-A-recursion} to obtain
\eqref{eq:localized-action-gradient-transfer};
\eqref{eq:distance-to-optimal-set} and
\eqref{eq:localized-greedy-stationarity} then give
\eqref{eq:localized-policy-set-distance}.
\end{proof}

\subsection{Entry and exploration thresholds}
\label{app:entry-threshold-calibration}

Fix the common residual coefficient $C_{\rm res}(K,\boldsymbol\rho)$
from \cref{thm:two-scale-true-residual}. Define the common gradient-error
coefficient by
\begin{equation}
  C_{\mathrm{grad\_err}}(K,\boldsymbol\rho)
  :=\left[1+L_{\mu,a}\Phi_T(L_E)
    \left(1+\frac{\bar L_{xa}}{\bar m}\right)\right]
    C_{\rm res}(K,\boldsymbol\rho).
  \label{eq:gradient-error-coefficient}
\end{equation}
This choice gives the bounds in
\eqref{eq:entry-local-transfer-rates}, with $b_{\rm entry}$ added only to
the global entry radius.  Set
$\gamma_0:=\gamma/2-\diam(\mathcal A)b_{\rm entry}>0$ and define
\begin{equation*}
  n_{\rm ent}
  :=
  \left\lceil
    \{C_{\mathrm{grad\_err}}(K,\boldsymbol\rho)\}^2
    \max\left\{
      \frac{\diam(\mathcal A)^2}{\gamma_0^2},
      \frac{1}{\bar m^2\varrho_{\rm cov}^2}
    \right\}
  \right\rceil.
\end{equation*}

\begin{proof}[Proof of \cref{cor:entry-threshold}]
Substituting
\eqref{eq:verified-entry-residual-rate}--\eqref{eq:verified-local-residual-rate}
into \eqref{eq:entry-transfer-radius}
gives \eqref{eq:entry-local-transfer-rates}.
Combining \eqref{eq:gradient-error-coefficient} with
\eqref{eq:verified-certificate-numerator-size} gives
\begin{equation}
  C_{\mathrm{grad\_err}}(K,\boldsymbol\rho)
  =\mathcal{O}\!\left(
    [1+\Phi_T(L_E)]
    \left[1+\sqrt{\frac{\ell_{\rm mdl}}T}\right]
  \right),
  \label{eq:transfer-numerator-dependence}
\end{equation}
for fixed $T$; the structural coefficients and $\varrho_{\rm cov}$ need not
remain uniformly controlled as $T$ varies.
For $n_j\geq n_{\rm ent}$, \eqref{eq:entry-local-transfer-rates} gives
\[
  \diam(\mathcal A)D_j^{\rm glob}
  \leq\diam(\mathcal A)b_{\rm entry}+\gamma_0=\frac\gamma2,
  \qquad
  D_j^{\rm loc}\leq\bar m \varrho_{\rm cov}.
\]
Thus \cref{thm:backward-union-entry} gives entry and
$\pi^{(j)}\in\Pi_{\rm loc}(\varrho_{\rm cov})$.
These conclusions persist for nondecreasing sample counts whenever
\cref{def:two-scale-true-residual-certificate} and the rate bounds
\eqref{eq:verified-entry-residual-rate}--\eqref{eq:verified-local-residual-rate}
hold.
\end{proof}

\paragraph{Remaining exploration thresholds.}
The thresholds $n_\Sigma$ and $n_{\rm gram}$ are defined in
\eqref{eq:model-sample-threshold} and
\eqref{eq:warm-start-thresholds}, with $C_{\rm cov}=28/3$ the constant
from \cref{lem:adaptive-matrix-coverage}.
The first threshold guarantees $n_jT\geq\delta_{\max}\ell_{\rm mdl}$;
the second gives the event probability and coverage implication in
\cref{app:coverage-concentration-proof}.
The overall exploration threshold is
\begin{equation}
  n_\star(K,\boldsymbol\rho)
  :=\max\{n_{\rm ent},n_{\rm gram},n_\Sigma\}.
  \label{eq:overall-warm-start-threshold}
\end{equation}
Thus $n_j\geq n_\star(K,\boldsymbol\rho)$ meets the model sample,
entry and proximity, and coverage sample requirements simultaneously.

\subsection{Local one-episode policy loss}
\label{app:localized-one-episode-regret-proof}

\begin{proof}[Proof of the first inequality in \eqref{eq:localized-one-episode-loss-rate}]
Assume the hypotheses of \cref{thm:backward-union-entry}.
For $h<H$, $x$, and $a\in\mathcal U_h^\star(x)$, choose the globally
optimal branch point $b$ in a bin containing $a$.  The segment $[a,b]$
lies in that convex bin, where strong concavity holds, and every
$v\in N_{\mathcal A}(a)$ satisfies $\langle v,b-a\rangle\leq0$.
The square-completion argument of \cref{app:global-geometry-proof},
followed by the infimum over this global normal cone, therefore gives
\begin{equation*}
  V_h^\star(x)-Q_h^\star(x,a)
  \leq\frac{d_h(x,a)^2}{2\delta\bar m},
  \qquad a\in\mathcal U_h^\star(x).
\end{equation*}

Entry permits this comparison at the fitted greedy action.
By \eqref{eq:localized-greedy-stationarity} and
\eqref{eq:localized-action-gradient-transfer}, the right-hand side is at
most $\delta(D_j^{\rm loc})^2/(2\bar m)$.
The performance-difference identity in \cref{lem:optimal-advantage-pdi},
summed over $H=T/\delta$ stages, gives the first inequality in
\eqref{eq:localized-one-episode-loss-rate}.
\end{proof}

\subsection{Size of the fixed-horizon constants}
\label{app:localized-rate-bookkeeping-proof}

\begin{proof}[Details for \eqref{eq:transfer-numerator-size} and \eqref{eq:warm-start-size}]
For fixed $T$, squaring
\eqref{eq:transfer-numerator-dependence} and using
$(1+\sqrt z)^2\leq2(1+z)$ gives \eqref{eq:transfer-numerator-size}.
The entry threshold is the ceiling of the squared common coefficient
times the maximum of two fixed structural multipliers, hence
$n_{\rm ent}=\mathcal{O}(1+\ell_{\rm mdl}/T)$.
Likewise, $\lceil z\rceil\leq z+1$ gives
\[
  n_\Sigma=\mathcal{O}(1+\ell_{\rm mdl}/T),
  \qquad
  n_{\rm gram}=\mathcal{O}\!\left(1+\log\frac{2p\bar J_K}{\rho_{\rm cov}}\right).
\]
Taking the maximum proves \eqref{eq:warm-start-size}.

\end{proof}

\paragraph{General physical-horizon factors.}
To expose the continuation, entry, and exploration factors suppressed
by the fixed-$T$ orders, put
\begin{align*}
  A_T&:=1+L_{\mu,a}\left(1+\frac{\bar L_{xa}}{\bar m}\right)
                  \Phi_T(L_E),\\
  Q_T&:=\max\left\{
     \frac{\diam(\mathcal A)^2}{\gamma_0^2},
     \frac{1}{\bar m^2\varrho_{\rm cov}^2}\right\}.
\end{align*}
If the structural multipliers in
\cref{lem:epoch-model-transfer,lem:two-scale-fitted-residual} remain
bounded as $T$ varies, the coefficient choices in
\eqref{eq:verified-certificate-numerators} and
\eqref{eq:gradient-error-coefficient}
give
\begin{align*}
  \{C_{\mathrm{grad\_err}}(K,\boldsymbol\rho)\}^2
  &=\mathcal{O}\!\left(A_T^2\left[1+\frac{\ell_{\rm mdl}}T\right]\right),\\
  n_\star(K,\boldsymbol\rho)
  &=\mathcal{O}\!\left(
      1+\frac{\delta_{\max}\ell_{\rm mdl}}T
      +\frac{B_{\Phi_\mu}^4}{\lambda_\mu^2}
         \log\frac{2p\bar J_K}{\rho_{\rm cov}}
      +Q_TA_T^2\left[1+\frac{\ell_{\rm mdl}}T\right]
    \right),\\
  \Reg(K,\delta)
  &=\mathcal{O}\!\left(
      B_Jn_\star(K,\boldsymbol\rho)
      +\frac{A_T^2}{\bar m}(T+\ell_{\rm mdl})(J_K+1)
    \right).
\end{align*}
The last bound holds on $\mathcal E_{\rm all}(\delta)$ for
$K>n_\star(K,\boldsymbol\rho)$, with all $\mathcal{O}(\cdot)$ multipliers
independent of $(T,K,\boldsymbol\rho,\delta)$ under the stated uniform
control.  These orders follow, respectively, by squaring the coefficient
bounds, substituting into the three thresholds, and applying
\eqref{eq:end-to-end-regret-bound}.

For fixed drift and curvature bounds with $L_E>0$,
$A_T^2=\mathcal{O}(e^{2L_E T})$; if $L_E=0$, then
$A_T^2=\mathcal{O}((1+T)^2)$ because $\Phi_T(0)=T$.
With fixed primitive bounds and positive lower bounds on
$\gamma_0,\bar m,\lambda_{\rm loc}$, the choice
$\varrho_{\rm cov}=2\lambda_{\rm loc}/L_\Gamma$ from
the calibration in \cref{app:optimal-envelope-regularity} gives
$Q_T=\mathcal{O}(1+T)$ since $L_\Gamma=\mathcal{O}(1+\sqrt T)$; when $L_\Gamma=0$, take a
fixed positive radius.  Exploration also retains the return factor
$B_J=(TB_r+B_u)e^{c_WT}\int W\,dp_0$.
The theorem assumes only fixed-horizon control: structural multipliers
may have further $T$-dependence, so uniformity over timestep refinements
does not imply a uniform or polynomial bound in $T$.

\subsection{Coverage and entry at all fitted epochs}
\label{app:coverage-entry-induction-proof}

\begin{proof}[Details of the induction in \cref{thm:coverage-entry-induction}]
Work on $\mathcal E_{\rm all}(\delta)$, with the probability bound
established in the main proof.  Independent resets give
\[
  \E[\mathsf G_k\mid\mathcal H_{k-1}]=\Gamma(\pi_k)
\]
in the notation of \cref{app:coverage-concentration-proof}, under the
original probability law; the success event is used only for pathwise
implications.

At $n_0\geq\max\{n_{\rm gram},n_\Sigma,n_{\rm ent}\}$, exploration
gives conditional information at least $2\lambda_\mu I_p$.
Thus \eqref{eq:main-coverage-deviation-event} gives empirical coverage,
and \cref{thm:two-scale-true-residual,cor:entry-threshold} give the
residual certificate, bin entry, and membership in
$\Pi_{\rm loc}(\varrho_{\rm cov})$.

Inductively, every policy deployed before the next fitted endpoint
$n_{j+1}<K$ is either $\pi^{\rm exp}$ or an earlier fitted policy in
$\Pi_{\rm loc}(\varrho_{\rm cov})$.
By \eqref{eq:exploration-population-coverage} and
\eqref{eq:local-self-excitation}, each has conditional information at
least $2\lambda_\mu I_p$, so the same deviation event gives coverage.
Since $n_{j+1}\geq n_0$, all thresholds remain satisfied and the same
verification and entry results apply.  Every fitted endpoint therefore
passes the coverage check.  If $n_{j+1}\geq K$, no further fit or
coverage check is performed.
\end{proof}

\section{Verification of the controlled example with two branch points}
\label{app:controlled-multibranch-example}

We verify \cref{prop:example-benchmark} with constants uniform in
$0<\delta\leq1$.  In the effective action $b=a^2\in[0,2]$, the
Bellman target for an admissible kernel $P=P_{\theta,\sigma^2}$ is
\begin{equation}
  F_h^{P,v}(x,b)
  =-2\delta(b-1)^2+
    \E v\left(x+\delta\theta\left(1+\frac b4\right)
                    +\sqrt\delta\,\sigma\xi\right),
  \qquad G_h^{P,v}(x,a)=F_h^{P,v}(x,a^2).
  \label{eq:example-effective-target}
\end{equation}
Write $F_h^\star=F_h^{P^\star,V_{h+1}^\star}$.

\subsection{Feasible critics and envelope regularity}
\label{subsec:example-projected-sieve}
\label{app:example-envelope-geometry}

For $t_h=h\delta$, set
\[
  \mathsf B_h=1+\frac{21}{10}(1-t_h),\qquad
  \mathsf g_h=1+\frac1{10}(1-t_h),\qquad
  \mathsf M_h=1+\frac15(1-t_h).
\]
Require each effective critic $F$ to be real, $2\pi$-periodic in $x$,
and $C^2$ on $\mathbb R\times[0,2]$, with
\begin{align}
  |F|&\leq\mathsf B_h,
  & |F_x|&\leq\mathsf g_h,\nonumber\\
  |F_{xx}|&\leq\mathsf M_{h+1}+\delta/10,
  & |F_{xb}|&\leq\delta(\mathsf M_{h+1}/4+1/10),\nonumber\\
  -6\delta&\leq F_{bb}\leq-2\delta,
  & |F_b(x,1)|&\leq3\delta/2.
  \label{eq:example-shape-constraints}
\end{align}
Let $T_m(\cos t)=\cos(mt)$ and use the normalized bases
\[
\begin{aligned}
  \varphi_L(x)&=(1,\sqrt2\cos x,\sqrt2\sin x,\ldots,
                   \sqrt2\cos Lx,\sqrt2\sin Lx),\\
  \tau_L(b)&=(1,\sqrt2T_1(b-1),\ldots,\sqrt2T_L(b-1)).
\end{aligned}
\]
For $L\geq2$, take
\begin{equation*}
  \Psi_L(x,a)=\varphi_L(x)\otimes\tau_L(a^2),\qquad
  Q_{h,L,w}(x,a)=-2\delta(a^2-1)^2+\Psi_L(x,a)^\top w,
\end{equation*}
with $q_L=(2L+1)(L+1)$ coefficients.  The feasible set
$\mathcal C_{h,L}$ consists of coefficients $w$ with $\|w\|\leq6$
whose effective critics satisfy \eqref{eq:example-shape-constraints};
it is compact and convex and
contains $w=0$.  Extending the Chebyshev factors affinely beyond
$b\in[0,2]$, matching values and first derivatives, gives features
that are $C^1$ on $\mathbb R\times\mathcal A^{\rm o}$, where
$\mathcal A^{\rm o}=(-2,2)$, and satisfy
\begin{equation*}
  \|\Psi_L\|_\infty,\ L_{\Psi,x,L},\ L_{\Psi,a,L}
  \leq D_L:=8(L+1)^3\sqrt{q_L}.
\end{equation*}
Indeed, $|T_m'|\leq m^2$ on $[-1,1]$ bounds each extended tensor
component and its $x$ and $a$ derivatives by
$2(1+2L^2)$, $2L(1+2L^2)$, and $8L^2$, respectively.

\paragraph{Envelope bounds.}
Strong concavity and the slope bound at $b=1$ give a unique maximizer
$b_F(x)\in[1/4,7/4]$.  The implicit function theorem yields
\[
  v_F'=F_x(x,b_F),\qquad
  v_F''=F_{xx}(x,b_F)-\frac{F_{xb}(x,b_F)^2}{F_{bb}(x,b_F)},
  \qquad v_F(x):=F(x,b_F(x)).
\]
Since $\mathsf M_{h+1}\leq6/5$, the second derivative is bounded by
$\mathsf M_{h+1}+9\delta/50\leq\mathsf M_h$.  Thus $v_F$ is
periodic and $C^2$, with
\begin{equation}
  \|v_F\|_\infty\leq\mathsf B_h,\qquad
  \|v_F'\|_\infty\leq\mathsf g_h,\qquad
  \|v_F''\|_\infty\leq\mathsf M_h.
  \label{eq:example-envelope-bounds}
\end{equation}
The critic has state gradient $v_F'(x)$ at both maximizing actions
$\pm\sqrt{b_F(x)}$.  The offset $q_{0,h,L}(x,a)=-2\delta(a^2-1)^2$
satisfies $\|\nabla_x q_{0,h,L}\|_\infty=0$ and
$\|\nabla_a q_{0,h,L}\|_\infty\leq48$ on
$\mathbb R\times\mathcal A^{\rm o}$ for $0<\delta\leq1$.
Thus \cref{ass:envelope-stable-critic-class} holds with
$G_0=11/10$ and $M_0=6/5$.
Also $Q_{aa}(x,0)=2F_b(x,0)\geq\delta$, so global action concavity
is not imposed.

\paragraph{Bellman feasibility.}
For the target $F=F_h^{P,v}$ with a next-stage continuation $v$
satisfying \eqref{eq:example-envelope-bounds}, differentiation of
\eqref{eq:example-effective-target} gives
\begin{align*}
  |F|&\leq\mathsf B_h-\delta/10,
  & |F_x|&\leq\mathsf g_h-\delta/10,\nonumber\\
  |F_{xx}|&\leq\mathsf M_{h+1},
  & |F_{xb}|&\leq\delta\mathsf M_{h+1}/4,\nonumber\\
  |F_{bb}+4\delta|&\leq3\delta/40,
  & |F_b(x,1)|&\leq11\delta/40.
\end{align*}
These bounds use $|\theta|\leq1$ and $\delta\leq1$ and leave at
least $\delta/100$ slack in every constraint.  Hence any periodic
$C^2$ function $f$ with $\|f-F\|_{C^2}\leq\delta/100$ is
shape-feasible, where $\|f\|_{C^2}:=\max_{p+r\leq2}\|\partial_x^p\partial_b^r f\|_\infty$.
Starting from $u=\cos$, backward induction gives the same bounds
for the optimal targets and envelopes.

\subsection{Bin geometry and exploration}
\label{app:example-benchmark-geometry}

The optimal effective action satisfies
$|b_h^\star(x)-1|\leq\mathsf g_{h+1}/8\leq11/80$.
For $\eta=1/16$, use the bins
\[
  B_{h,+}^\star(x)=
  [\sqrt{b_h^\star(x)-\eta},\sqrt{b_h^\star(x)+\eta}],
  \qquad B_{h,-}^\star(x)=-B_{h,+}^\star(x).
\]
Their endpoints are continuous, and each bin contains exactly one
optimal branch point.  On their union, $a^2\geq11/16$ and
$|F_{h,b}^\star|\leq6\delta\eta$, giving
\[
  Q_{h,aa}^\star=2F_{h,b}^\star+4a^2F_{h,bb}^\star
       \leq-19\delta/4,\qquad
  |Q_{h,xa}^\star|\leq(3\sqrt2/5)\delta<\delta.
\]
Outside the union, strong concavity in $b$ gives
$V_h^\star-Q_h^\star\geq\delta|a^2-b_h^\star|^2\geq\delta\eta^2$.
Thus \cref{ass:localized-optimal-branches} holds with
\begin{equation*}
  M_\star=2,\qquad \bar m=4,\qquad
  \bar L_{xa}=1,\qquad \gamma=1/256.
\end{equation*}
Since $\Phi_\mu^2\geq1$, population and empirical information are
at least one under every policy.  Taking
$\pi_h^{\rm exp}=1$, $\lambda_{\rm exp}=1/2$, and
$\lambda_{\rm loc}=\lambda_\mu=1/4$ verifies
\cref{ass:exploration-self-excitation}; the coverage check always
passes.  The constants $B_{\Phi_\mu}=2$,
$L_{\Phi_\mu,a}=L_{\mu,a}=1$, and $T=\lambda_0=1$ give
$L_\Gamma=8$ and the valid radius $\varrho_{\rm cov}=1/16$.
The unknown drift affects the optimal magnitude: differentiating the
last-stage first-order condition at $\theta=0$ gives
$\partial_\theta b_{H-1}^{\star,\theta}(x)|_{\theta=0}
=-e^{-\delta\sigma^2/2}\sin(x)/16$, which is nonzero at $x=\pi/2$.

\subsection{Synthetic design and projected approximation}
\label{app:example-design-projection}

Use the $q_L$ tensor anchors
\begin{equation*}
  x_i=\frac{2\pi i}{2L+1},\quad i=0,\ldots,2L,
  \qquad
  a_m=\sqrt{1+\cos\frac{(m+1/2)\pi}{L+1}},\quad m=0,\ldots,L.
\end{equation*}
Discrete Fourier and cosine orthogonality give
$q_L^{-1}\sum_{i,m}\Psi_L(x_i,a_m)\Psi_L(x_i,a_m)^\top=I_{q_L}$.
Complete repetitions therefore yield $S_{j,h}^Q=N_{j,h}^QI_{q_L}$,
verifying \cref{ass:predictable-critic-design} with $c_{Q,L}=1$.
Positive-action anchors suffice because the features are even in $a$.

Define
\begin{equation*}
  \varepsilon_L(\delta)
  :=8\left\{\sum_{k>L}k^2e^{-\delta k^2/2}
                   +(L+1)^3(e/4)^L\right\}.
\end{equation*}
For a fit using level $L$, we show that whenever
$\varepsilon_L(\delta)\leq\delta/100$, the noiseless constrained
projection satisfies, uniformly over admissible
kernels, feasible next-stage continuations (with $v=u$ at the terminal
stage), and repeated anchor designs,
\begin{equation}
  \sup_{x,a}|\partial_s(G_h^{P,v}-\bar Q_{j,h}^{P,v})(x,a)|
  \leq(3+D_L)\varepsilon_L(\delta),\qquad s\in\{x,a\}.
  \label{eq:example-projection-bound}
\end{equation}
Let $\widehat v_k$ be the Fourier coefficients of $v$, with
$|\widehat v_k|\leq\|v\|_\infty\leq31/10$.  Approximate the
Gaussian continuation by
\[
  S_L(x,b)=\sum_{|k|\leq L}\widehat v_k e^{-\delta\sigma^2k^2/2}
       e^{ik(x+5\delta\theta/4)}
       \sum_{m=0}^L\frac{(ik\delta\theta/4)^m(b-1)^m}{m!}.
\]
Conjugate symmetry makes $S_L$ real and in the tensor span.
Since admissible kernels have $\sigma^2\geq\lambda_0=1$, the
discarded Fourier modes contribute at most
$(31/5)\sum_{k>L}k^2e^{-\delta k^2/2}$ in $C^2$ norm.
For retained modes, Taylor's remainder and
$(L+1)!\geq((L+1)/e)^{L+1}$ give, for $p+r\leq2$,
\[
  |k|^p\left|\partial_b^r\left(
  e^{ic(b-1)}-\sum_{m=0}^L\frac{(ic(b-1))^m}{m!}\right)\right|
  \leq\frac{|k|^p|c|^{L+1}}{(L+1-r)!}
  \leq(L+1)^2(e/4)^{L+1},\quad c=\frac{k\delta\theta}{4}.
\]
Summing over $|k|\leq L$ bounds this contribution by
$8(L+1)^3(e/4)^L$.  Therefore
$F_L=-2\delta(b-1)^2+S_L$ has $C^2$ error at most
$\varepsilon_L$ and is shape-feasible.  Let $w_L$ be the coefficient
vector of $S_L$ in the tensor basis.  Orthonormality under the
uniform--Chebyshev product probability measure gives
$\|w_L\|=\|S_L\|_{L^2}\leq31/10+\varepsilon_L<6$.

For this fit, write $N=N_{j,h}^Q$, $Z_i=Z_{j,h,i}$, and
$\bar w=\bar w_{j,h}^{P,v}$, and put $\Delta=\bar w-w_L$.
The variational inequality and the identity Gram matrix imply
\[
  N\|\Delta\|^2
  \leq\sum_i\Psi_L(Z_i)^\top\Delta\,
                   (G_h^{P,v}-Q_{h,L,w_L})(Z_i)
  \leq N\|\Delta\|\varepsilon_L.
\]
Thus $\|\Delta\|\leq\varepsilon_L$.  Composing with $b=a^2$
multiplies the direct derivative error by at most $2\sqrt2<3$,
and replacing $w_L$ by $\bar w$ adds at most $D_L\varepsilon_L$,
proving \eqref{eq:example-projection-bound}.
Choose the deterministic, nondecreasing level rule
\begin{equation*}
  L_j:=\min\left\{L\geq2:
      (3+D_L)\varepsilon_L(\delta)
       \leq\frac{\delta}{100\sqrt{n_j}}\right\}.
\end{equation*}
For fixed $\delta>0$, it is finite because $\varepsilon_L$ decays
exponentially while $D_L$ grows polynomially.  Hence
\cref{ass:projected-two-scale-Bellman-approximation} holds with
$b_{\rm entry}=0$ and
$c_a^{\rm glob}=c_x^{\rm loc}=c_a^{\rm loc}=1/100$.

\subsection{Budgets and completion of the verification}
\label{subsec:example-budgets-regret}
\label{app:example-budget-verification}

\begin{proposition}[Simulation calibration]
\label{prop:example-budget-calibration}
Write $q_j=q_{L_j}$, and let $\kappa_j^Q$ be the
confidence factor in \cref{ass:localized-MC-budget}.  At every stage,
repeat the complete anchor set and take
\begin{equation*}
  N_{j,h}^Q=N_j:=q_j\left\lceil\frac{n_j}{q_j}\right\rceil,
 \quad
  M_{j,h}:=\left\lceil
        \frac{G_0^2\Lambda_0D_{L_j}^2\kappa_j^Q}{\delta}\right\rceil,
\end{equation*}
where $G_0=11/10$, $\Lambda_0=2$, and $C_{\rm MC}\geq2$ is the constant
in \cref{ass:localized-MC-budget}.  That assumption holds with
$c_N=c_{{\rm MC},x}=c_{{\rm MC},a}=1$.
\end{proposition}

\begin{proof}
We have $S_{j,h}^Q=N_jI$, $c_{Q,j}=1$, and
$N_j\geq\max\{n_j,q_j\}$.  The specified constants and $D_{L_j}>0$ give
$M_{j,h}\geq1$.  Using $L_{\Psi,s,j}\leq D_{L_j}$, for $s\in\{x,a\}$,
\[
  \frac{L_{\Psi,s,j}G_0}{\delta\sqrt{N_j}}
       \sqrt{\frac{\delta\Lambda_0\kappa_j^Q}{M_{j,h}}}
       \leq\frac1{\sqrt{N_j}}
       \leq\frac1{\sqrt{n_j}}.
\]
This verifies both simulation-budget inequalities.
\end{proof}

\begin{proof}[Completion of the proof of \cref{prop:example-benchmark}]
The bounded rewards and continuous Gaussian kernel verify
\cref{ass:basic} with $W=1$, $c_W=0$, $B_r=2$, $B_u=1$, and $B_J=3$.
On $\mathcal A^{\rm o}=(-2,2)$, the drift feature has bounds
$(B_{\Phi_\mu},L_{\Phi_\mu,x},L_{\Phi_\mu,a})=(2,0,1)$ and
$|r_a|\leq48$.  Together with $u=\cos$,
$B_\theta=\lambda_0=1$, and
$\Lambda_0=2$, this verifies
\cref{ass:localized-parametric-model}.  The preceding construction
verifies the bin geometry, exploration, critic, design, approximation, and
budget assumptions.  Their structural and normalized error constants
are independent of $\delta$, so
\cref{thm:end-to-end-localized-regret} proves the claim.
\end{proof}

This construction permits timestep-dependent sieve levels and synthetic
budgets; it uses $HN_jM_{j,0}$ synthetic transitions per epoch.  It assumes
exact constrained regression and effective-action maximization, with
any Borel choice of sign; the fixed sieve and numerical constraint
checks of Section~\ref{subsec:controlled-multibranch-model}
approximate it in practice.

\subsection{Numerical details and supplementary results}
\label{app:example-numerics}

The implementation in \cref{subsec:controlled-multibranch-model} uses
$128$ periodic state points, $16$ Fourier modes for envelopes, and
Fourier damping for Gaussian expectations. On repeated tensor anchors,
fits minimize summed squared errors plus $\delta^2\|w\|^2/(144D_6^2)$.
Fourier-amplitude bounds and differentiated Chebyshev coefficients verify
\eqref{eq:example-shape-constraints} and $\|w\|\leq6$ over
$\mathbb R\times[0,2]$, subject to floating-point arithmetic. Every
fitted critic passed these checks.
\Cref{fig:example-cumulative-regret} shows mean cumulative regret, and
\cref{tab:example-final-regret} reports its decomposition at $K=8192$.

\begin{figure}[!htbp]
  \centering
  \begin{minipage}[c]{0.68\linewidth}
    \centering
    \includegraphics[width=\linewidth]{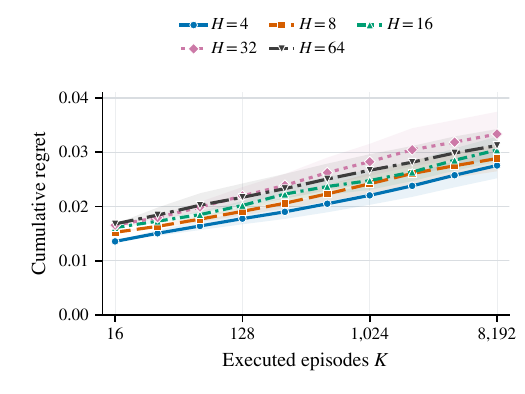}
  \end{minipage}\hfill
  \begin{minipage}[c]{0.3\linewidth}
    \caption{Mean cumulative regret through $K$ executed episodes in the
  model with two branch points, including the 16 exploration episodes.
  Each curve averages 64 independent runs at resolution $H=1/\delta$;
  shading gives pointwise 95\% percentile bootstrap intervals from 2,000
  resamples of complete runs.}
    \label{fig:example-cumulative-regret}
  \end{minipage}
\end{figure}

\begin{table}[!htbp]
  \centering
  \caption{Final cumulative regret at $K=8192$. Exploration regret is
  deterministic; learning regret is reported as its mean and 95\% percentile
  bootstrap interval from 64 independent runs and 2,000 complete-run resamples
  per resolution. Totals may differ from component sums due to rounding.}
  \label{tab:example-final-regret}
  \begin{tabular}{@{}rrrcr@{}}
    \toprule
    Steps $H$  & Exploration regret (16 episodes) & Mean learning regret & 95\% interval
      & Total regret \\
    \midrule
     4  & 0.01356 & 0.01399 & $[0.01163,\,0.01678]$ & 0.02755 \\
     8  & 0.01524 & 0.01360 & $[0.01125,\,0.01621]$ & 0.02884 \\
    16 & 0.01611 & 0.01429 & $[0.01207,\,0.01676]$ & 0.03040 \\
    32 & 0.01656 & 0.01679 & $[0.01348,\,0.02090]$ & 0.03335 \\
    64 & 0.01678 & 0.01445 & $[0.01202,\,0.01748]$ & 0.03124 \\
    \bottomrule
  \end{tabular}
\end{table}

\section{Implementation and evaluation of the Halfin--Whitt experiments}
\label{app:hw-draft-details}

\subsection{Queue scaling and experimental scope}
\label{app:hw-draft-scaling}

Consider an untruncated queue with $n$ servers and customer count $N_t$. Arrivals
are Poisson with rate $n\mu-\beta\sqrt n$, and each of the
$\min\{N_t,n\}$ busy servers completes service at rate
$\mu+\kappa A_t/\sqrt n$.  For $X_t^{(n)}=(N_t-n)/\sqrt n$,
the limiting drift and infinitesimal variance are
$-\beta+\mu x^- -\kappa a$ and $2\mu$, giving the controlled diffusion
\[
 dX_t=(-\beta+\mu X_t^- -\kappa A_t)\,dt+\sqrt{2\mu}\,dW_t,
 \qquad A_t\in[0,2],
\]
with physical cost
$\E\int_0^T(qX_t^++cA_t^2/2)\,dt$ and no terminal cost.
Experiments 1(a)--2(a) use the Euler model \eqref{eq:hw-draft-euler}
with stage cost $\delta\ell$; $n$ does not enter once its coefficients
are fixed. Experiment 2(b) transfers the resulting policies to finite
queues. These studies use nonsmooth coefficients and a finite action
grid; they do not verify the main theorem's smoothness and approximation
hypotheses. The numerical reference uses the learner's action grid.

\subsection{Estimation, critic, and computational budgets}
\label{app:hw-draft-implementation}

The learner observes $(X_h,A_h,\delta\ell(X_h,A_h),X_{h+1})$, retains
all observations across fits, and subtracts the known drift
contribution $\mu X_h^-$ before estimating
$\theta=(\beta,\kappa)^\top$. Index observed transitions by $i$, with
$X_{i,+}$ denoting the successor of $X_i$. With
$\phi(x,a)=(-1,-a)^\top$, the estimator first solves
\[
 \widetilde\theta\in\arg\min_{\theta\in\R^2}
 \sum_i\left\{
 \frac{X_{i,+}-X_i}{\delta}-\mu X_i^-
                  -\phi(X_i,A_i)^\top\theta\right\}^2,
\]
and radially projects onto the Euclidean ball of radius 4; this is not
a constrained least-squares solve. The diffusion variance estimate is
the second moment of residual increments divided by $\delta$, clipped
to $[0.25,8]$. The fitted Gaussian kernel supplies synthetic Bellman
successors.

Using the spline coordinate $u(x)=(1+\tanh(x/4))/2$ and clamped cubic
B-splines $B_j$ in $u$, the cost critic is
\begin{equation}
 C_h(x,a)=\sum_{j=1}^{m}B_j(u(x))
       (w_{0,hj}+a w_{1,hj}+a^2w_{2,hj})
+(H-h)\delta q\,\operatorname{softplus}_{0.25}(x)
                       +\frac{\delta c}{2}a^2,
 \label{eq:hw-draft-critic}
\end{equation}
where $\operatorname{softplus}_{s}(x)=s\log(1+e^{x/s})$.
Bellman targets use the exact stage cost $\delta\ell$, despite the smooth known
offset in \eqref{eq:hw-draft-critic}.  The coefficient constraints are
\[
 |w_{0,hj}|\le20,\qquad |w_{1,hj}|\le20\delta,
 \qquad 0\le w_{2,hj}\le10\delta,
\]
and the coefficient ridge is $10^{-4}\delta^2$ in Experiments 1(a)--1(b)
and exactly zero at every fit in Experiments 2(a)--2(b).
Zero ridge leaves the box constraints,
parameter projection, covariance clipping, and coverage check in place.
The nonnegative quadratic coefficients and the known action-cost term
make the critic strongly convex in the continuous action; deployment
minimizes it over $\mathcal A_{0.25}$.

The spline refinement level after $k$ episodes is
$\lfloor\log_2(k/n_0)/3\rfloor$.  In Experiments 1(a)--1(b),
$n_0=64$ gives 16 segments at fits through episode 256 and 32 at
episode 512.  In Experiments 2(a)--2(b), the common $n_0=32$
advances this first refinement to episode 256, with 16 segments at
fits through episode 128 and 32 at episodes 256 and 512.  These
choices give $m=19$ and $m=35$ basis functions and critic dimension
$d_Q=3m$.  At a fit after $k$ observed episodes, $2m$ state anchors
equally spaced in $u\in[0.005,0.995]$ are crossed with actions
$0,1,2$.  The resulting $6m$ rows are tiled until their count $N$
is at least $k$, and each row uses $M=\lceil d_Q/\delta\rceil$
fresh Gaussian successors at each stage.  At the terminal fit
$k=512$, $H=50$, this gives $N=630$ and $M=1050$, or about
$3.3\times10^7$ synthetic draws.  With critic and anchor dimensions
fixed, both $H$ and $M$ scale as $1/\delta$, so the eightfold
timestep refinement in Experiment 1(b) raises the leading synthetic
work per fit by about a factor of 64.

Experiments 1(a)--1(b) use 64 uniform-action episodes and dyadic
refits at $k=64,128,256,\ldots$. All saved coverage checks pass,
with minimum eigenvalues $0.0839$ in the three long-horizon cohorts
and $0.0947$ in the 32 Experiment 1(b) runs. In Experiment 2(a),
all five methods share 32 exploration episodes, with no additional
exploration. \ourmethod{} refits at $k=32,64,128,256,512$ and holds
policies fixed between refits. Every saved coverage check passes the
$0.025$ threshold; the minimum information eigenvalue is $0.1104$.
The fit after episode 512 is used only for terminal Euler and queue
evaluation; training costs use fits through episode 256. Likewise,
Experiment 1(b)'s terminal fit after episode 2048 does not enter
$R_{2048}$.

\subsection{Euler reference and evaluation protocol}
\label{app:hw-draft-evaluation}

For each timestep, backward dynamic programming computes the Euler
Bellman-grid value $V_\delta^{\rm grid}$ on $x\in[-8,12]$ with state
spacing $0.02$ and the nine actions in $\mathcal A_{0.25}$. Gaussian integration uses
piecewise-linear value interpolation with linear extrapolation and
a nine-standard-deviation tail truncation. The reference policy
$\pi_\delta^{\rm ref}$ uses the action at the nearest state knot, so its deployed actions
stay on the declared grid; it is a numerical benchmark, not a
certified optimum over continuous states and actions.

A resolution check over state spacings $0.02$, $0.01$, and $0.005$
found a paired cost difference of $-1.2\times10^{-5}\pm3.3\times10^{-5}$
between the coarsest and finest reference policies. This supports the
resolution at the reported gap scale but does not certify all reference
errors uniformly in $\delta$.

Euler policy gaps use 32,768 evaluation paths with common Brownian
increments for each learned policy and its matched reference.
Experiments 1(a)--1(b) weight each deployed-policy gap by its epoch
length; reusing gaps correlates Monte Carlo errors across cumulative
checkpoints. Experiment 2(a) evaluates terminal policies this way but
uses observed training costs for cumulative excess, without
re-evaluating intermediate policies. For Experiment 1(b)'s terminal
recommendations, per-run paired Monte Carlo standard errors range
from $5.8\times10^{-5}$ to $2.4\times10^{-4}$.

Experiments 1(a) and 1(b) use separate eight-seed cohorts, matched
across timesteps within each cohort, and Student-$t$ intervals with
seven degrees of freedom. Experiments 2(a)--2(b) share 15 prespecified
seeds across methods and timesteps; all 60 timestep--seed runs per
method are included, with two-sided 95\% Student-$t$ intervals using
14 degrees of freedom. Paired intervals use within-seed
comparator-minus-\ourmethod{} differences. Independent training
seeds are the replicates; timesteps, episodes, and evaluation paths
are not treated as independent training replicates. Intervals are
marginal and unadjusted for multiple comparisons; they include Euler evaluation noise but do not propagate
reference-solver bias.

Numerical settings are selected on separate development seeds and
fixed across timesteps within each study. Comparison methods share
exogenous training streams and the 32-episode exploration prefix;
subsequent trajectories can differ with their actions.

\subsection{Growth fits and endpoint uncertainty}
\label{app:hw-draft-diagnostics}

Table~\ref{tab:hw-draft-long-endpoint} gives uncertainty at the final
checkpoint of Figure~\ref{fig:hw-draft-long-horizon}.
Least-squares fits of the five seed-mean checkpoints
$K=1024,2048,4096,8192,16384$ against $\log_2 K$ have slopes
$1.82$, $1.68$, and $1.41$ per doubling and $R^2$ values
$0.998$, $0.996$, and $0.994$ for
$\delta=0.2,0.1,0.05$, respectively.  These fits are descriptive: the
cumulative measurements are dependent, not independent observations
for a rate test.  Nor do the overlapping endpoint intervals establish
that smaller timesteps improve performance.

\begin{table}[htbp]
 \centering
 \caption{Experiment 1(a): cumulative regret after exploration at
 $K=16384$.  Intervals are pointwise 95\% Student-$t$ intervals
 across eight training seeds.}
 \label{tab:hw-draft-long-endpoint}
 \small
 \begin{tabular}{@{}r r c@{}}
\toprule
$\delta$ & $H$ & $R_{16384}-R_{64}$, mean [pointwise 95\% CI] \\
\midrule
0.200 & 25 & 10.503 [7.580, 13.426] \\
0.100 & 50 & 10.653 [8.127, 13.179] \\
0.050 & 100 & 8.622 [5.264, 11.980] \\
\bottomrule
\end{tabular}
\end{table}

\FloatBarrier

\subsection{Comparison-method implementations}
\label{app:hw-draft-comparators}

The four baselines use Experiment 2(a)'s observation format and action
grid; their terminal policies are frozen for Experiment 2(b).

\paragraph{Source-style DW (without replay).}
Source-style DW uses real Euler successors and the backward cost target
$\delta\ell(X_h,A_h)+\min_{a\in\mathcal A_{0.25}}\widehat C_{h+1}(X_{h+1},a)$;
the final-stage cost is known exactly.  Its 33 features tensorize
linear state hats on $\{-10,-4,-2,-1,0,0.5,1,2,3,4,6\}$ with the
first three Legendre polynomials in $a-1$. The objective is mean-squared
loss plus $10^{-5}$ times the coefficient norm squared. After 32 common
uniform-action episodes it refits at $k=32,64,128,256$, followed by
a terminal fit at $k=512$, on the newest epoch's episodes only, and deploys
a deterministic greedy policy between fits.  It estimates no transition
model and uses no synthetic successors.  The frozen Euler critic,
including its knots, is unchanged during queue transfer.

\paragraph{DW + replay.}
This adaptation pools all observed episodes, including the 32
exploration episodes, and otherwise retains source-style DW's
features, ridge, action grid, refit schedule, and deterministic greedy
deployment. It preserves exploratory observations without action
mixing, at the cost of larger regressions.

\paragraph{Online little-$q$.}
Let $F(t,x)$ tensorize 11 linear time hats in $t/T$ with 16-segment
cubic state splines in $(1+\tanh(x/4))/2$. In the reward convention
$r=-\ell$ used in the main text, the value function and action score are
\[
 V_\vartheta(t,x)=(1-t/T)F(t,x)^\top\vartheta,\qquad
 s_\psi(t,x,a)=aF(t,x)^\top\psi-ca^2/2.
\]
On the common action grid, set
$q_\psi=s_\psi-\nu\log\sum_{a'}\exp(s_\psi(t,x,a')/\nu)$
and $\pi_\psi(a\mid t,x)=\exp(q_\psi(t,x,a)/\nu)$.
Using the pre-update parameters, the episode residuals are
\[
 e_h=V_\vartheta(t_{h+1},X_{h+1})-V_\vartheta(t_h,X_h)
       -\delta\ell(X_h,A_h)-\delta q_\psi(t_h,X_h,A_h).
\]
The updates add $\alpha_{V,k}\sum_h\nabla_\vartheta V_\vartheta\,e_h$
and $\alpha_{q,k}\sum_h\nabla_\psi q_\psi\,e_h$ to the two parameter
vectors after every episode, including the 32 common uniform-action
exploration episodes. The parameters are selected on a separate
development seed and fixed across timesteps:
\[
 \nu=0.2,\qquad
 \alpha_{V,k}=10^{-3}(1+(k-1)/64)^{-0.6},\qquad
 \alpha_{q,k}=5\times10^{-4}(1+(k-1)/64)^{-0.6},
\]
with no gradient clipping or coefficient ridge.
The native Gibbs policy is the primary output; maximizing $s_\psi$
gives the secondary greedy policy.  Both are evaluated under
unregularized cost.

\paragraph{Batch little-$q$.}
This variant retains online little-$q$'s parameterization, action
grid, and 32 exploration episodes, but solves the discretized
martingale orthogonality conditions in batches.
At a refit after $k$ episodes, let
$e_{ih}(\vartheta,\psi)$ denote the residual above for stage $h$ of
episode $i\le k$.  The fit solves
\begin{equation}
 \sum_{i=1}^k\sum_{h=0}^{H-1}\nabla_\vartheta V_\vartheta(t_h,X_{ih})\,
   e_{ih}(\vartheta,\psi)=\lambda_k\vartheta,
 \qquad
 \sum_{i=1}^k\sum_{h=0}^{H-1}F(t_h,X_{ih})\,(A_{ih}-\bar a_{ih})\,
   e_{ih}(\vartheta,\psi)=\lambda_k\psi,
 \label{eq:hw-draft-batch-little-q}
\end{equation}
where $\bar a_{ih}$ is the generating policy's mean action at
$(t_h,X_{ih})$ (equal to $1$ during exploration) and
$\lambda_k=\rho kT$. The martingale condition holds under any behavior
policy, allowing all episodes to be pooled; both instruments in
\eqref{eq:hw-draft-batch-little-q} are known at decision time.
Unlike online little-$q$'s iterate-dependent instrument
$\nabla_\psi q_\psi=F(t_h,X_h)(A_h-\E_{\pi_\psi}[a\mid t_h,X_h])$,
the batch behavior-policy instrument is fixed by the data.
Ridge keeps unvisited basis coefficients at zero.

The equations are linear in $\vartheta$ and smooth in $\psi$; we solve
them by damped Newton iteration with the exact Jacobian, warm-started
from the previous fit, stopping when the relative step falls below
$10^{-10}$.  Refits occur at $k=32,64,128,256$ and after episode
512, with the native Gibbs policy deployed between refits.  Training
costs therefore describe those deployed Gibbs policies, including
when the reported terminal extraction is greedy.

We use $(\nu,\rho)=(0.2,10^{-5})$ for all batch results.
These parameters are selected on a separate development seed at
$H\in\{25,200\}$ by minimizing the larger terminal greedy gap
among converged settings, and are fixed for the reported runs.

\subsection{Comparison diagnostics and realized training excess}
\label{app:hw-draft-comparator-diagnostics}

\paragraph{Terminal value gaps.}
Table~\ref{tab:hw-draft-common-euler} gives the absolute terminal value
gaps and deterministic reference costs underlying the percentage results
in Table~\ref{tab:hw-draft-common-euler-percent}.

\begin{table}[!htbp]
 \centering
 \caption{Experiment 2(a): terminal value gaps after $K=512$ episodes.
 Method entries give the mean and its marginal 95\% Student-$t$ confidence
 interval across 15 training seeds; comparator entries are rounded to three
 significant digits. The reference-cost column gives the deterministic Euler
 Bellman-grid value $V_\delta^{\rm grid}$. Calculations use the unrounded source estimates.}
 \label{tab:hw-draft-common-euler}
 \scriptsize
 \setlength{\tabcolsep}{4pt}
 \renewcommand{\arraystretch}{1.12}
 \begin{tabular}{@{}rrccccc@{}}
  \toprule
$\delta$ & {\shortstack{Reference cost\\$V_\delta^{\rm grid}$}} & \shortstack{\ourmethod{}\\(ours)} & \shortstack{DW\\source style} & \shortstack{DW\\+ replay} & \shortstack{Online little-$q$\\native Gibbs} & \shortstack{Batch little-$q$\\greedy} \\
  \midrule
  0.200 & {$1.408281$} & $0.0029\;[0.0016,0.0043]$ & $2.28\;[1.96,2.60]$ & $0.623\;[0.556,0.690]$ & $0.610\;[0.607,0.613]$ & $0.364\;[0.317,0.411]$ \\
  0.100 & {$1.378613$} & $0.0024\;[0.0017,0.0032]$ & $2.71\;[2.49,2.92]$ & $0.855\;[0.744,0.965]$ & $0.636\;[0.633,0.639]$ & $0.332\;[0.301,0.362]$ \\
  0.050 & {$1.363004$} & $0.0025\;[0.0019,0.0031]$ & $3.05\;[2.94,3.16]$ & $1.22\;[1.06,1.38]$ & $0.651\;[0.649,0.652]$ & $0.338\;[0.313,0.363]$ \\
  0.025 & {$1.355958$} & $0.0017\;[0.0011,0.0022]$ & $3.16\;[3.06,3.27]$ & $1.74\;[1.60,1.89]$ & $0.658\;[0.655,0.660]$ & $0.321\;[0.295,0.348]$ \\
  \bottomrule
 \end{tabular}
\end{table}

\paragraph{Online little-$q$ diagnostics.}
Every evaluated online little-$q$ greedy terminal policy selects
action zero.  Its Euler gaps are
$1.0813\ [1.0755,1.0870]$, $1.1291\ [1.1243,1.1339]$,
$1.1541\ [1.1503,1.1578]$, and $1.1652\ [1.1615,1.1689]$
for $\delta=0.2,0.1,0.05,0.025$. Its queue gaps from uniformization are
$1.03909$--$1.13474$, with no training-seed variation for these
identical policies.  Across the 60 Euler evaluations, the native
sampled mean action is $0.2836$--$0.2846$, close to $0.2840$
for the Gibbs policy with zero learned score,
$\pi(a)\propto\exp\{-ca^2/(2\nu)\}$ on $\mathcal A_{0.25}$.
Thus action rankings change little at this budget, and comparisons
with greedy methods also include entropy-induced randomization cost.

\paragraph{DW coverage.}
For source-style DW, fixed resets and deterministic greedy deployment
give a rank-one first-stage design at the later refits
$k=64,128,256,512$ in all 60 runs.  Both DW variants have first-stage
rank three at the initial fit $k=32$; replay preserves that rank at
later fits by retaining exploratory observations; the full feature dimension is
33, and the fixed initial state prevents full first-stage rank even
with replay. Replay improves the
mean gap at every timestep, but the result depends on data reuse and
coverage and does not isolate an intrinsic property of the estimator
under its theoretical assumptions.

\paragraph{Batch little-$q$ diagnostics.}
Across the four timesteps, mean terminal gaps of the native Gibbs policy range
from $0.5267$ to $0.5414$, compared with $0.3211$--$0.3642$
after greedy extraction. Training costs use the deployed Gibbs
policies. Of 300 fits, 196 meet the stopping rule and 104 do not,
including three of the 60 terminal fits. All prespecified runs,
including these unconverged fits, remain in the reported estimates.

\paragraph{Computational cost.}
Across the four timesteps, mean learning wall times range from
$8.78$ to $708.17$ seconds for \ourmethod{}, $0.96$ to $7.87$
for source-style DW, $0.97$ to $7.80$ for DW + replay,
$2.97$ to $21.64$ for online little-$q$, and $80.25$ to
$175.27$ for batch little-$q$. Timings were collected separately
under potentially different machine loads and do not establish an
equal-computation or simultaneous speed comparison.

\paragraph{Realized excess after exploration.}
Write $J_H^{\rm ref}=V_\delta^{\rm grid}$ for the deterministic Euler
Bellman-grid value and abbreviate
$\widehat R_k^{\rm real}$ by $\widehat R_k$. Table~\ref{tab:hw-draft-expected-cumulative}
reports Experiment 2(a)'s realized excess over episodes 33--512; the full-budget
quantity appears in Table~\ref{tab:hw-draft-realized-excess}. They are
\[
 \widehat R_{512}=\sum_{k=1}^{512} C_k-512J_H^{\rm ref},
 \qquad
 \widehat R_{512}-\widehat R_{32}
 =\sum_{k=33}^{512} C_k-480J_H^{\rm ref}.
\]
Both use observed costs, including Gibbs action sampling for the two
little-$q$ methods, and the deterministic Bellman value rather than
a Monte Carlo reference estimate. Post-exploration, \ourmethod{}'s
mean is $10.04$--$19.46$; intervals at the two coarser timesteps cross
zero and are retained without clipping. Replay reduces DW's mean
at every timestep, though its mean grows from $477.32$ to $1171.23$
under refinement.

\begin{table}[htbp]
 \centering
 \caption{Experiment 2(a): realized post-exploration excess
 $\widehat R_{512}-\widehat R_{32}$ over episodes 33--512.
 Entries are means and 95\% Student-$t$ intervals across the same
 15 seeds as Table~\ref{tab:hw-draft-common-euler}.  Both
 little-$q$ columns describe native Gibbs data collection.}
 \label{tab:hw-draft-expected-cumulative}
 \scriptsize
 \setlength{\tabcolsep}{2pt}
 \renewcommand{\arraystretch}{1.12}
 \begin{tabular}{@{}rccccc@{}}
  \toprule
  $\delta$ & \ourmethod{} & \shortstack{DW\\(source style)} & DW + replay
  & \shortstack{Online\\little-$q$} & \shortstack{Batch\\little-$q$} \\
  \midrule
  0.200 & $10.04\;[-7.03,27.12]$ & $1263.89\;[1186.36,1341.42]$ & $477.32\;[449.72,504.91]$ & $301.06\;[275.36,326.76]$ & $293.42\;[272.29,314.55]$ \\
  0.100 & $15.12\;[-4.18,34.43]$ & $1367.83\;[1280.00,1455.67]$ & $705.24\;[643.89,766.60]$ & $321.53\;[293.12,349.94]$ & $294.66\;[268.41,320.91]$ \\
  0.050 & $19.46\;[7.96,30.96]$ & $1469.01\;[1433.00,1505.02]$ & $958.45\;[910.29,1006.61]$ & $338.93\;[321.11,356.75]$ & $297.07\;[285.21,308.92]$ \\
  0.025 & $18.29\;[4.68,31.90]$ & $1490.74\;[1422.62,1558.86]$ & $1171.23\;[1137.89,1204.58]$ & $341.55\;[318.67,364.42]$ & $294.25\;[271.24,317.26]$ \\
  \bottomrule
 \end{tabular}
\end{table}
\FloatBarrier

\paragraph{Full-budget realized excess.}
\label{app:hw-draft-realized-excess}
Table~\ref{tab:hw-draft-realized-excess} includes the common 32
exploration episodes. Let $\mathcal H_{k-1}$ be the information
available before episode $k$, including its deployed policy $\pi_k$.
Then $\E[C_k\mid\mathcal H_{k-1}]=J_\delta(\pi_k)$, averaging over
trajectory noise and Gibbs action sampling. Thus, with $R_K$ from
\eqref{eq:hw-draft-excess},
\begin{equation*}
 \widehat R_K
 =R_K+\sum_{k=1}^K\{C_k-J_\delta(\pi_k)\}
       +K\{J_\delta(\pi_\delta^{\rm ref})-J_H^{\rm ref}\}.
\end{equation*}
The sum is a mean-zero martingale sampling term; the last term is
the numerical benchmark offset. Consequently,
\[
 \E[\widehat R_K]
 =\E[R_K]+K\{J_\delta(\pi_\delta^{\rm ref})-J_H^{\rm ref}\}.
\]
The policy-dependent $R_K$ remains random. When the benchmark values
coincide, the independent-run mean of $\widehat R_K$ is unbiased for
$\E[R_K]$; this need not hold conditional on the entire adaptive
policy sequence, since later
policies depend on earlier observations. If conditional episode-cost
variances are bounded by $\sigma^2$, the martingale term has standard
deviation at most $\sigma\sqrt K$ and can obscure logarithmic regret
growth in individual runs. Confidence intervals across seeds include learning,
trajectory, and Gibbs-action variation, but not the numerical
benchmark offset. The distinction from Experiment 1's epoch-weighted
policy-gap estimates and the paired-seed protocol are described in
Appendix~\ref{app:hw-draft-evaluation}.

\begin{table}[!htbp]
 \centering
 \caption{Experiment 2(a): realized cumulative excess at $K=512$, including the common 32 exploration episodes. Entries give means and marginal 95\% Student-$t$ confidence intervals across training seeds, including trajectory and action-sampling variation. Both little-$q$ arms collect data with Gibbs policies.}
 \label{tab:hw-draft-realized-excess}
 \scriptsize
 \setlength{\tabcolsep}{2pt}
 \renewcommand{\arraystretch}{1.12}
 \begin{tabular}{@{}rccccc@{}}
  \toprule
$\delta$ & \ourmethod{} & \shortstack{DW\\(without replay)} & \shortstack{DW\\+ replay} & \shortstack{Online little-$q$\\native Gibbs} & \shortstack{Batch little-$q$\\native Gibbs} \\
  \midrule
  0.200 & $91.24\;[74.93,107.54]$ & $1345.08\;[1266.14,1424.03]$ & $558.51\;[530.46,586.56]$ & $382.26\;[357.34,407.17]$ & $374.61\;[353.45,395.78]$ \\
  0.100 & $97.02\;[77.64,116.40]$ & $1449.73\;[1362.83,1536.63]$ & $787.13\;[725.44,848.83]$ & $403.42\;[374.79,432.05]$ & $376.55\;[349.92,403.19]$ \\
  0.050 & $101.54\;[89.52,113.56]$ & $1551.09\;[1515.42,1586.75]$ & $1040.53\;[990.44,1090.62]$ & $421.01\;[402.69,439.32]$ & $379.14\;[367.09,391.20]$ \\
  0.025 & $101.93\;[86.92,116.95]$ & $1574.38\;[1506.64,1642.13]$ & $1254.88\;[1218.90,1290.85]$ & $425.19\;[401.62,448.77]$ & $377.89\;[354.35,401.44]$ \\
  \bottomrule
 \end{tabular}
\end{table}
\FloatBarrier

\subsection{Finite-queue transfer evaluation}
\label{app:hw-draft-transfer-evaluation}

For each $(\delta,n)$, backward dynamic programming minimizes over the
nine grid actions held for a full stage.  The same finite-state
uniformization routine evaluates the frozen policies and the fixed-stage
reference, integrating physical cost over $[0,T]$ with Poisson truncation
tolerance $10^{-16}$.  For online little-$q$, it averages the action-specific
stage costs and transitions under the Gibbs probabilities at the stage's
initial state; the sampled action is then held throughout that stage.

The intervals in Table~\ref{tab:hw-draft-transfer} use the 15
training seeds, with no Monte Carlo evaluation error; reference and
policy costs remain numerical quantities subject to truncation.

\paragraph{Customer-level evaluation.}
Figure~\ref{fig:hw-draft-operational-transfer} uses 16,384 event-driven
paths per policy and training seed.  Actions are chosen only at stage
boundaries, with an independent Gibbs draw for each online little-$q$
path and stage.  Methods share offered-arrival and service-hazard
streams within each seed cluster.  For admitted arrivals before $T$,
we record full FCFS waits $W_j$, including zero waits.  Arrivals stop at
$T$ and action zero drains the admitted cohort to avoid censoring;
effort cost is integrated only over $[0,T]$.  Scaled mean wait and
delay probability are the pooled ratios
$\sqrt n\sum_j W_j/\sum A$ and $\sum D/\sum A$, where $A$
counts admissions and $D$ counts admitted arrivals finding all
servers busy.  Their intervals use the seed-cluster delta method
with $t_{14}$ across 15 clusters, and paired comparisons preserve
within-seed covariance.  Customers are not treated as independent
training replicates.

\begin{figure}[htbp]
 \centering
 \includegraphics[width=\linewidth]{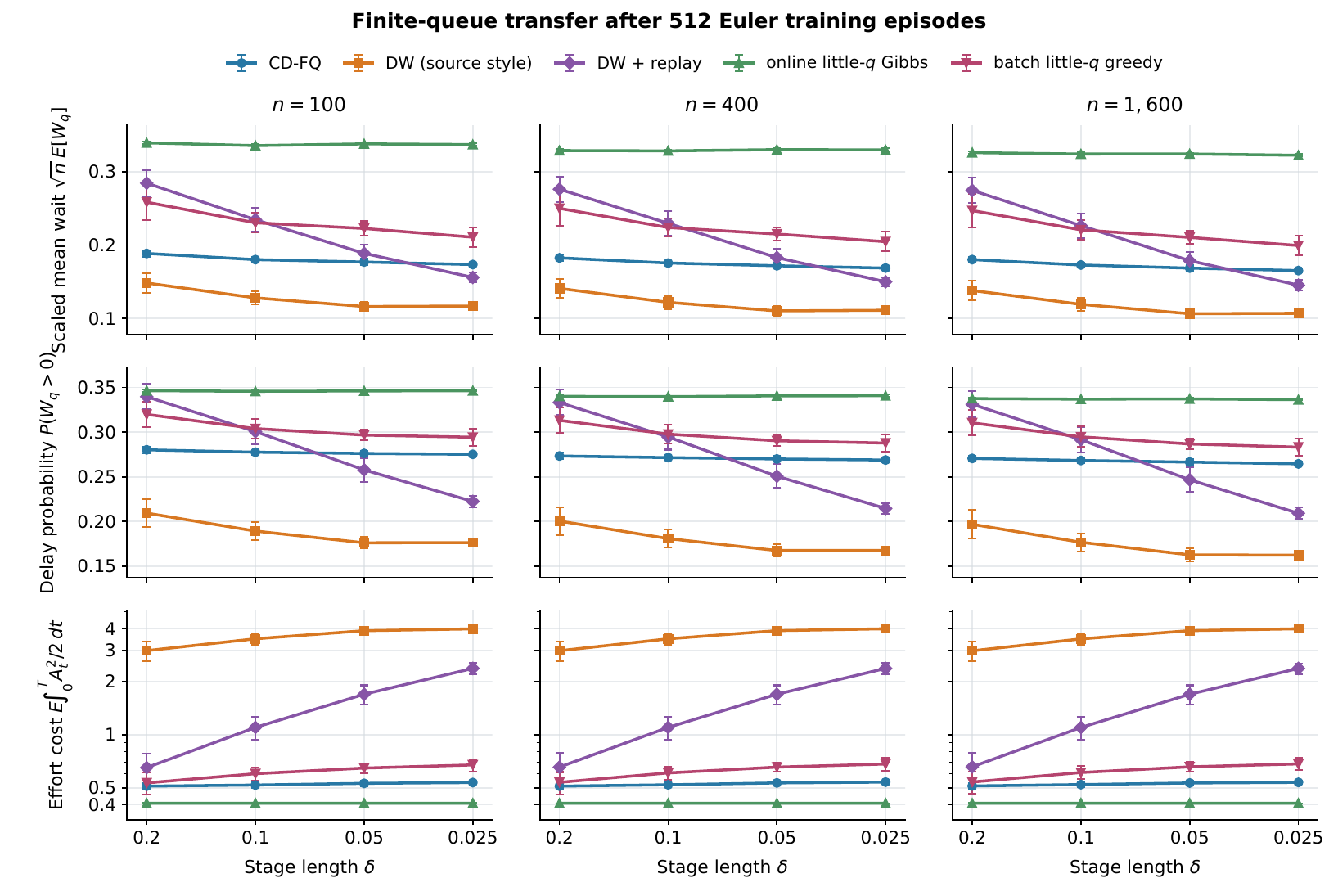}
 \caption{Experiment 2(b): service and effort after $K=512$ Euler
 training episodes and $n_0=32$ common exploration episodes.
 The five frozen policies are \ourmethod{} with zero coefficient
 ridge, source-style DW, DW + replay, online little-$q$ native
 Gibbs, and batch little-$q$ greedy, all on their training
 decision clocks.  Columns are queue sizes; rows show scaled mean
 wait, delay probability, and effort cost (log scale).  Points are
 pooled customer ratios or seed means; bars are pointwise 95\%
 intervals using 15 seed clusters and 16,384 paths per policy/seed.}
 \label{fig:hw-draft-operational-transfer}
\end{figure}

The paired effort intervals for batch little-$q$ relative to
\ourmethod{} cross zero at $\delta=0.2$; all operational intervals
are marginal and unadjusted. The largest pooled blocking-probability
estimate is $6.46\times10^{-5}$. Customer-level sampled costs were
checked against uniformization in all 12 timestep--queue-size cells,
with maximum absolute standardized discrepancy $3.20$.

\FloatBarrier

\clearpage
\section{Notation and constants for localized learning}
\label{app:localized-notation}

\Cref{tab:localized-notation} summarizes the notation in
\cref{sec:localized-algorithm,sec:main-theorem-proof}; references give
full definitions. Here $s\in\{x,a\}$ indexes state or action derivatives,
$\mathrm{glob}$ denotes the full
action set, $\mathrm{loc}$ the optimal-bin union, and $\mathrm{ent}$ bin
entry. Residual and action-gradient radii are normalized by $\delta$.
Timestep uniformity is for fixed $T$, with events defined separately
for each fixed-$\delta$ experiment.

\begingroup
\small
\setlength{\tabcolsep}{4pt}
\renewcommand{\arraystretch}{1.06}
\begin{longtable}{@{}>{\raggedright\arraybackslash}p{0.26\linewidth}
  >{\raggedright\arraybackslash}p{0.53\linewidth}
  >{\raggedright\arraybackslash}p{\dimexpr0.21\linewidth-4\tabcolsep\relax}@{}}
\caption{Notation and constants for Sections~\ref{sec:localized-algorithm}
and~\ref{sec:main-theorem-proof}. References point to definitions or
the corresponding bounds.}
\label{tab:localized-notation}\\
\toprule
Notation & Meaning & Reference \\
\midrule
\endfirsthead
\multicolumn{3}{@{}l}{\tablename~\thetable\ (continued)}\\
\toprule
Notation & Meaning & Reference \\
\midrule
\endhead
\midrule
\multicolumn{3}{r@{}}{\emph{Continued on next page}}\\
\endfoot
\bottomrule
\endlastfoot

\multicolumn{3}{@{}l}{\textbf{Sampling, model, and coverage}}\\*
\addlinespace[1.5pt]
$T$, $\delta$, $H$; $n$, $d$, $p$
& Physical horizon, timestep, stages $H=T/\delta$; state, action, and drift-parameter dimensions.
& \cref{sec:model-global-regret,ass:localized-parametric-model} \\
\addlinespace[1.5pt]
$K$, $n_0$, $n_j$; $J_K$, $\bar J_K$
& Episode budget, exploration count, and observed episodes $n_j=2^jn_0$ before epoch $j$; last epoch index and epoch-count bound.
& \eqref{eq:doubling-schedule},
\eqref{eq:epoch-count-bound} \\
\addlinespace[1.5pt]
$\Phi_\mu$, $\theta^\star$, $B_\theta$
& Known drift features, unknown parameter, and its norm bound.
& \cref{ass:localized-parametric-model} \\
\addlinespace[1.5pt]
$B_{\Phi_\mu}$, $L_{\Phi_\mu,s}$, $L_{\mu,s}$
& Feature magnitude and derivative bounds; drift-derivative bounds $L_{\mu,s}=B_\theta L_{\Phi_\mu,s}$.
& \cref{ass:localized-parametric-model,lem:parametric-growth-moment} \\
\addlinespace[1.5pt]
$\Sigma^\star$, $\lambda_0$, $\Lambda_0$, $\mathcal S_0$
& True covariance, known lower and upper eigenvalue bounds, and admissible covariance set.
& \cref{ass:localized-parametric-model};
\eqref{eq:admissible-model-family} \\
\addlinespace[1.5pt]
$\widehat\theta_j$, $\widehat\Sigma_j$, $\widehat P_j$
& Projected drift-parameter and covariance estimates, and fitted Euler kernel; $\widehat{\mathcal P}_j$ is its expectation operator.
& \eqref{eq:epoch-drift-estimator}--\eqref{eq:epoch-learned-kernel} \\
\addlinespace[1.5pt]
$\mathcal A_h^\star(x)$, $B_{h,\ell}^\star(x)$,
$a_{h,\ell}^\star(x)$, $\mathcal U_h^\star(x)$
& Optimal-action set, local convex bins, unique optimum in each bin, and bin union.
& \cref{ass:localized-optimal-branches} \\
\addlinespace[1.5pt]
$M_\star$, $\bar m$, $\bar L_{xa}$, $\gamma$
& Bin-count bound; local action curvature, mixed-gradient bound, and off-bin margin, each normalized by $\delta$.
& \cref{ass:localized-optimal-branches} \\
\addlinespace[1.5pt]
$\Gamma(\pi)$, $\widehat\Gamma_{n_j}$
& Population information averaged over stages; empirical information averaged over $n_jH$ observed transitions.
& \cref{ass:exploration-self-excitation,subsec:estimation-scheme} \\
\addlinespace[1.5pt]
$\pi^{\rm exp}$, $\lambda_{\rm exp}$,
$\lambda_{\rm loc}$, $\lambda_\mu$
& Exploration policy and excitation constants; coverage threshold $\lambda_\mu=\min\{\lambda_{\rm exp},\lambda_{\rm loc}\}$.
& \cref{ass:exploration-self-excitation} \\
\addlinespace[1.5pt]
$\Pi_{\rm loc}(\varrho)$, $\varrho_{\rm cov}$
& Policies uniformly within distance $\varrho$ of optimal actions; timestep-independent radius ensuring $\Gamma(\pi)\succeq2\lambda_{\rm loc}I_p$ for $\pi\in\Pi_{\rm loc}(\varrho_{\rm cov})$.
& \cref{lem:optimal-policy-neighborhood-excitation} \\
\addlinespace[1.5pt]
$\boldsymbol\rho$
& Failure budgets for coverage, drift, covariance, and Monte Carlo events, with sum less than one.
& \cref{ass:localized-MC-budget} \\
\addlinespace[1.5pt]
$\ell_{\rm mdl}$
& Model confidence factor; $n$ in its definition is the state dimension.
& \eqref{eq:localized-model-log-factor} \\

\addlinespace[6pt]
\multicolumn{3}{@{}l}{\textbf{Critic class, design, and simulation}}\\*
\addlinespace[1.5pt]
$q_j$; $Q_{h,j,w}$,
$q_{0,h,j}$, $\Psi_{h,j}$, $\mathcal C_{h,j}$
& Coefficient dimension, critic affine in $w$, known offset, features, and feasible coefficient set.
& \eqref{eq:epoch-critic-sieve} \\
\addlinespace[1.5pt]
$L_{\Psi,s,j}$, $s\in\{x,a\}$
& Feature-derivative bounds, uniform in stage and timestep at each fixed sieve level.
& \cref{ass:envelope-stable-critic-class} \\
\addlinespace[1.5pt]
$G_0$, $M_0$
& Uniform magnitude and Lipschitz bounds on fitted value gradients.
& \eqref{eq:verified-fitted-envelope-bounds} \\
\addlinespace[1.5pt]
$Z_{j,h,i}$, $N_{j,h}^Q$, $M_{j,h}$
& Predictable design rows, row count, and fresh successors per row; $N_{j,h}^QM_{j,h}$ transitions per fit.
& \eqref{eq:epoch-synthetic-transitions} \\
\addlinespace[1.5pt]
$S_{j,h}^Q$, $c_{Q,j}$
& Critic Gram matrix and conditioning constant: $S_{j,h}^Q\succeq c_{Q,j}N_{j,h}^QI_{q_j}$.
& \eqref{eq:deterministic-critic-gram} \\
\addlinespace[1.5pt]
$G_h^{P,v}$, $\bar Q_{j,h}^{P,v}$
& Exact Bellman target and its noiseless constrained least-squares projection.
& \eqref{eq:noiseless-constrained-projection} \\
\addlinespace[1.5pt]
$\widehat Q_{j,h}$, $\widehat V_{j,h}$, $\pi_h^{(j)}$
& Fitted critic, full-action maximum, and greedy policy.
& \eqref{eq:epoch-constrained-least-squares},
\eqref{eq:epoch-greedy-policy} \\
\addlinespace[1.5pt]
$c_N$, $c_{{\rm MC},s}$
& Design-size lower-bound coefficient; normalized Monte Carlo derivative-error coefficients.
& \cref{ass:localized-MC-budget} \\
\addlinespace[1.5pt]
$\kappa_j^Q$, $C_{\rm MC}$
& Dimension and confidence factor for Monte Carlo fits; universal allocation constant.
& \cref{ass:localized-MC-budget} \\

\addlinespace[6pt]
\multicolumn{3}{@{}l}{\textbf{Approximation and true-residual bounds}}\\*
\addlinespace[1.5pt]
$b_{\rm entry}$, $c_a^{\rm glob}$,
$c_s^{\rm loc}$
& Nondecaying global approximation floor, global decay coefficient, and local decay coefficients.
& \cref{ass:projected-two-scale-Bellman-approximation} \\
\addlinespace[1.5pt]
$G_{j,h}$, $e_{j,h}$
& Learned-model Bellman target and true-kernel residual.
& \eqref{eq:epoch-true-residual},
\eqref{eq:epoch-true-residual-decomposition} \\
\addlinespace[1.5pt]
$d_{a,j}^{\rm glob}$, $d_{s,j}^{\rm loc}$
& Global action-derivative and local $s$-derivative radii for the residual; local bounds hold on $\mathcal U_h^\star(x)$.
& \cref{def:two-scale-true-residual-certificate} \\
\addlinespace[1.5pt]
$C_{\rm res}(K,\boldsymbol\rho)$
& Common coefficient in the $n_j^{-1/2}$ residual-derivative bounds; the global action bound also includes $b_{\rm entry}$.
& \eqref{eq:verified-certificate-numerators};
\eqref{eq:verified-entry-residual-rate}--\eqref{eq:verified-local-residual-rate} \\

\addlinespace[6pt]
\multicolumn{3}{@{}l}{\textbf{Backward transfer and policy accuracy}}\\*
\addlinespace[1.5pt]
$\Delta Q_{j,h}$, $\Delta V_{j,h}$
& Critic and value errors; $\Delta V_{j,H}=0$.
& \cref{subsec:main-criteria-consequences} \\
\addlinespace[1.5pt]
$\kappa$, $L_E$, $\Phi_\tau(L)$
& Local geometry ratio, gradient-propagation rate, and accumulation factor, with $\Phi_\tau(0)=\tau$.
& \cref{subsec:main-criteria-consequences,subsec:direct-log-regret} \\
\addlinespace[1.5pt]
$D_j^{\rm loc}$, $D_j^{\rm glob}$
& Local and global action-gradient radii after backward transfer; the latter controls entry.
& \eqref{eq:entry-transfer-radius} \\
\addlinespace[1.5pt]
$C_{\mathrm{grad\_err}}(K,\boldsymbol\rho)$
& Common coefficient in the $n_j^{-1/2}$ action-gradient bounds; the global bound also includes $b_{\rm entry}$.
& \eqref{eq:gradient-error-coefficient} \\

\addlinespace[6pt]
\multicolumn{3}{@{}l}{\textbf{Exploration thresholds and success events}}\\*
\addlinespace[1.5pt]
$n_\Sigma$
& Covariance-estimation sample threshold.
& \eqref{eq:model-sample-threshold} \\
\addlinespace[1.5pt]
$n_\star(K,\boldsymbol\rho)$
& Overall exploration threshold for model estimation, entry, proximity, and coverage.
& \eqref{eq:overall-warm-start-threshold} \\
\addlinespace[1.5pt]
$\mathcal E_\theta$, $\mathcal E_\Sigma$,
$\mathcal E_{\rm mdl}$
& Drift-score and covariance-noise deviation events; $\mathcal E_{\rm mdl}=\mathcal E_\theta\cap\mathcal E_\Sigma$, with failure probability at most $\rho_\theta+\rho_\Sigma$.
& \eqref{eq:model-noise-event} \\
\addlinespace[1.5pt]
$\mathcal E_{\rm MC}$
& Simultaneous Monte Carlo fitting event; failure probability at most $\rho_{\rm MC}$.
& \cref{lem:two-scale-fitted-residual} \\
\addlinespace[1.5pt]
$\mathcal E_{\rm cov}$
& Event transferring average population coverage to empirical coverage at eligible epoch endpoints; failure probability at most $\rho_{\rm cov}$.
& \cref{lem:main-coverage-event} \\
\addlinespace[1.5pt]
$\mathcal E_{\rm all}(\delta)$
& Intersection of coverage, model, and Monte Carlo events supporting the coverage--entry induction.
& \eqref{eq:all-event-probability} \\
\end{longtable}
\endgroup

The coefficients $C_{\rm res}(K,\boldsymbol\rho)$ and
$C_{\mathrm{grad\_err}}(K,\boldsymbol\rho)$ depend on
$(K,\boldsymbol\rho)$ only through $\ell_{\rm mdl}$ when the horizon,
structural parameters, and tuning parameters are fixed. Suppressing
these dependencies does not imply uniformity over those parameters.

\end{document}